\documentclass[12pt]{article}

\usepackage[letterpaper,margin=1in]{geometry}

\AtBeginDocument{%
	\setlength{\abovedisplayskip}{6pt plus 2pt minus 1pt}%
	\setlength{\belowdisplayskip}{6pt plus 2pt minus 1pt}%
	\setlength{\abovedisplayshortskip}{3pt plus 2pt}%
	\setlength{\belowdisplayshortskip}{4pt plus 2pt minus 1pt}%
}

\usepackage[utf8]{inputenc}
\usepackage[T1]{fontenc}
\usepackage{lmodern}
\usepackage{microtype}
\usepackage{amsmath,amsthm,amsfonts,amssymb,mathrsfs,mathtools}
\usepackage{bbm}

\usepackage{graphicx}
\usepackage{float}
\usepackage{caption}
\usepackage{subcaption}
\usepackage{booktabs}
\numberwithin{figure}{section}

\usepackage{xcolor}
\usepackage{comment}
\usepackage{enumitem}
\usepackage{titlesec}
\titlespacing*{\section}{0pt}{6pt plus 1pt minus 1pt}{3pt plus 1pt minus 1pt}
\titlespacing*{\subsection}{0pt}{5pt plus 1pt minus 1pt}{2pt plus 1pt minus 1pt}
\setlist{topsep=2pt plus 1pt minus 1pt,itemsep=1pt,parsep=0pt}
\makeatletter
\def\thm@space@setup{%
	\thm@preskip=5pt plus 1pt minus 1pt%
	\thm@postskip=5pt plus 1pt minus 1pt%
}
\makeatother

\usepackage[round,authoryear]{natbib}

\usepackage[hidelinks]{hyperref}

\theoremstyle{plain}

\newtheorem{theorem}{Theorem}[section]
\newtheorem{lemma}[theorem]{Lemma}
\newtheorem{proposition}[theorem]{Proposition}
\newtheorem{corollary}[theorem]{Corollary}

\theoremstyle{definition}
\newtheorem{definition}[theorem]{Definition}
\newtheorem{example}[theorem]{Example}

\newtheorem*{namedthm}{\namedthmname}
\newcounter{namedthm}
\makeatletter
\newenvironment{named}[1]
{\def\namedthmname{Assumption #1}%
	\refstepcounter{namedthm}%
	\namedthm\def\@currentlabel{#1}}
{\endnamedthm}
\makeatother

\theoremstyle{remark}
\newtheorem{remark}[theorem]{Remark}

\newcommand{\defeq}{\coloneqq}
\newcommand{\ga}{\vartheta}

\newcommand{\E}{\mathbb{E}}

\newcommand{\Cov}{\mathrm{Cov}}

\newcommand{\var}{\mathbb{V}\mathrm{ar}}

\newcommand{\op}{\mathrm{op}}
\newcommand{\ind}{\boldsymbol{1}}

\newcommand{\cov}{\mathrm{cov}}

\newcommand{\Od}{\mathrm O(d)}
\newcommand{\PS}{\mathrm P}
\newcommand{\QS}{Q}

\title{
	A bootstrap approach for testing invariance under parameterized group actions: orthogonal reflections and axial symmetry}

\author{
	Alejandro Cholaquidis\thanks{Centro de Matemáticas, Facultad de Ciencias, Universidad de la República, Uruguay. Email: \texttt{acholaquidis@cmat.edu.uy}}
	\and
	Juan Cuesta-Albertos\thanks{Departamento de Matemáticas, Estadística y Computación, Universidad de Cantabria, Spain. Email: \texttt{juan.cuesta@unican.es}}
	\and
	Ricardo Fraiman\thanks{Centro de Matemáticas, Facultad de Ciencias, Universidad de la República, Uruguay. Email: \texttt{rfraiman@cmat.edu.uy}}
	\and
	Manuel Hernández-Banadik\thanks{Instituto de Estadística, Facultad de Ciencias Económicas y de Administración, Universidad de la República, Uruguay. Email: \texttt{manuel.hernandez@fcea.edu.uy}}
}

\date{}

\begin{document}
		
		\maketitle
		
	\begin{abstract}
		Testing whether a multivariate distribution is invariant under an orthogonal transformation is a classical problem when the transformation is fixed in advance. We address a harder situation: the transformation is unknown and must be inferred from the data. We frame this as invariance under a group action whose representation is indexed by an unknown parameter. We work in $\mathbb{R}^d$ under a simple-spectrum assumption on the covariance matrix $\Sigma$. Under this assumption, any orthogonal transformation that preserves the distribution must commute with $\Sigma$, and therefore must be a {reflection through a subspace spanned by a subset of the principal directions. This reduces the search over the orthogonal group to a finite family of candidate reflections, one for each subset of principal axes.} For each candidate, we construct a Kolmogorov--Smirnov-type statistic based on projected data and sample splitting. We derive its asymptotic distribution in a triangular-array framework and establish bootstrap validity under suitable regularity conditions. Axial symmetry about an unspecified direction (that is, invariance under reflection across an unknown one-dimensional subspace) and hyperplane (Householder) symmetry about an unspecified normal direction are the two leading particular cases: we treat them in detail, the former driving the simulation study and the latter the real-data application, to show how a concrete problem involving an unknown group representation can be successfully addressed.
	\end{abstract}
	
	
	\medskip
	

	\section{Introduction}
	
	Testing distributional invariance under a group action is a classical and powerful theme in multivariate statistics. It traditionally covers properties such as invariance under reflections, rotations, or general orthogonal transformations \citep{serfling2006multivariate}. Standard testing procedures usually assume that the transformation of interest is fixed in advance. A richer setting arises when the representation is indexed by an unknown parameter that must itself be estimated from the data. A general treatment of such parameterized group actions poses substantial theoretical challenges, but many important statistical problems fit naturally into this framework. In this paper we focus on a fundamental and highly interpretable case: testing whether a distribution in \(\mathbb R^d\) is invariant under an orthogonal transformation whose defining subspace is unknown and must be inferred from the data.
	
	To formalize this setting, let \(G\) be a compact group with Haar probability measure \(\lambda_G\), and let \(\Theta\) be a parameter space. For each \(\theta\in\Theta\), suppose that
	$T_{\theta,g}:\mathbb R^d\to\mathbb R^d,$ $g\in G,$
	defines a measurable action of \(G\), that is,
	$T_{\theta,e}=\mathrm{Id},$ $T_{\theta,g_1g_2}=T_{\theta,g_1}\circ T_{\theta,g_2}.$
	Writing \((T_{\theta,g})_\#P\) for the push-forward of a probability measure
	\(P\) by \(T_{\theta,g}\), define
	$\mathcal I_\theta 	= 	\left\{
	P:\ (T_{\theta,g})_\#P=P
	\text{ for all }g\in G
	\right\}.$ 	The corresponding invariant-distribution null hypothesis is
	\[
	H_0:\quad
	P\in\bigcup_{\theta\in\Theta}\mathcal I_\theta,
	\]
	or equivalently,
	$H_0: \text{there exists }\theta_0\in\Theta
	\text{ such that }
	(T_{\theta_0,g})_\#P=P
	\text{ for all }g\in G.$
	The Haar measure provides the natural projection onto each null class: the \(\lambda_G\)-average \(P\mapsto\int_G (T_{\theta,g})_\#P\,d\lambda_G(g)\) maps any law to a \(G\)-invariant one, so \(\mathcal I_\theta\) is never empty, and \(P\in\mathcal I_\theta\) exactly when \(P\) is a fixed point of this average. For \(G=\mathbb Z_2\), the case studied here, this projection is the symmetrization \(\tfrac12 P+\tfrac12 (T_{\theta,-1})_\#P\). Its empirical counterpart, with \(\theta\) replaced by an estimate, is precisely the resampling law of the bootstrap procedure we develop below.
	
	Several familiar symmetry hypotheses fit this framework. Central symmetry about
	an unknown center $\mu$ corresponds to \(G=\mathbb Z_2\) and
	\(T_{\mu,-1}(x)=2\mu-x\). Reflection through a line (axial symmetry) or through
	a hyperplane, rotational and cylindrical symmetries, and elliptical symmetry
	after an unknown affine transformation can all be formulated by choosing the
	group \(G\), the parameter \(\theta\), and the representation
	\(T_{\theta,\cdot}\) appropriately.
	
	\paragraph*{Orthogonal symmetries with an unknown subspace}
	This paper focuses on the case \(G=\mathbb Z_2\), where the nontrivial group
	element acts on centered observations through an orthogonal transformation
	whose fixed subspace is unknown. For a linear subspace \(V\subseteq\mathbb
	R^d\), let \(\PS_V\) denote the orthogonal projection onto \(V\), and consider
	the reflection \(\QS_V=2\PS_V-I_d\). This reflection fixes \(V\) pointwise and
	changes the sign of every vector in \(V^\perp\). The matrix \(\QS_V\) is
	symmetric and orthogonal, hence an involution (\(\QS_V^2=I_d\)): it is the
	orthogonal reflection with mirror subspace \(V\). When \(\dim V=1\), say
	\(V=\operatorname{span}(u)\), we recover the axial reflection
	\(R_u=2uu^\top-I_d\). When \(\dim V=d-1\), \(\QS_V\) is a Householder
	reflection. When \(V=\{0\}\), \(\QS_V=-I_d\) is central symmetry. So orthogonal
	invariance about an unspecified subspace means that there exists a subspace
	\(V_0\) such that \(X-\mathbb E[X]\stackrel{\mathcal L}{=}\QS_{V_0}(X-\mathbb
	E[X])\).
	
	A basic structural observation, developed in Section~\ref{sec:princomp}, drives
	the whole approach. Whenever the covariance matrix \(\Sigma\) exists, any
	orthogonal transformation \(Q\) that preserves the (centered) distribution must
	satisfy \(Q\Sigma=\Sigma Q\); that is, \(Q\) commutes with \(\Sigma\). Under a
	simple-spectrum assumption on \(\Sigma\) (all eigenvalues of \(\Sigma\) are
	different), the matrices that commute with \(\Sigma\) are precisely those that
	are diagonal in the principal frame. So an orthogonal symmetry must map each
	eigenvector \(u_i\) to \(\pm u_i\). Consequently, every admissible \(Q\) has the
	form
	\[
	Q \;=\; \sum_{i=1}^d \epsilon_i\,u_iu_i^\top
	\;=\; 2\PS_S-I_d,
	\qquad \epsilon=(\epsilon_1,\dots,\epsilon_d)\in\{-1,+1\}^d,\quad S=\{i:\epsilon_i=+1\},
	\]
	where \(\PS_S=\sum_{i\in S}u_iu_i^\top\) is the projection onto the principal
	subspace \(V_S=\operatorname{span}\{u_i:i\in S\}\). In other words, the only
	orthogonal transformations that can preserve the distribution are the
	reflections through principal subspaces. {This reduces the search over the
	orthogonal group \(\Od\) to a finite collection of \(2^d\) candidate
	reflections, indexed by the subsets \(S\subseteq\{1,\dots,d\}\).}
	
\paragraph*{Why the direction (subspace) is unknown}
The main statistical difficulty is that the representation of the group is
indexed by an unknown parameter, the mirror subspace \(V\) (the location
\(\mu\) enters only through a standard centering), and this parameter must be
inferred from the same data used for testing. Many existing symmetry tests
assume that the relevant direction, subspace, or transformation is specified in
advance. This assumption could be unrealistic in many cases. In shape analysis, medical
imaging, object inspection, and multivariate signal analysis, a natural
symmetry structure may exist but still be affected by pose, registration,
subject-specific variability, or latent orientation. For instance, in
three-dimensional medical images of approximately bilateral anatomical
structures, the apparent axis may depend on patient positioning, registration,
and anatomical variability. In vectorcardiographic data, the electrical
activity of the heart is represented through spatial loops whose orientation
carries diagnostic information. In both cases, symmetry or asymmetry features
are meaningful only after the relevant data-driven direction or subspace has
been identified. Testing symmetry around a fixed coordinate axis or subspace
may therefore confound genuine asymmetry with an arbitrary choice of
coordinates. A relevant inferential question is whether there exists a
data-driven subspace with respect to which the distribution is reflection
invariant, while properly accounting for the fact that this subspace has
itself been estimated from the data.

Testing symmetry about subspaces has an
established track record across several fields. In medical imaging, symmetry
provides an interpretable baseline against which structured deviations can be
assessed: symmetry-based features improve abnormality detection in chest
radiography \citep{Hogeweg2017}, and asymmetry patterns are studied as markers
of neurological and psychiatric conditions in neuroimaging
\citep{Kong2018,Kuo2022,Kong2022,Martos2018}. In industrial inspection and
heritage diagnostics, departures from expected symmetry are similarly used to
flag defects, damage, or loss of structural regularity
\citep{Ren2022,Mitra2013,SanchezAparicio2023}. These examples motivate why an
unspecified-direction symmetry test is a useful tool to have.

The application we do pursue, in Section~\ref{sec:forex-application}, is to
foreign-exchange returns, where symmetry has already been tested, but always
about a direction fixed in advance by the researcher.
\citet{rousseeuw2002depth} test angular symmetry of DEM/USD and JPY/USD
returns about the origin, and \citet{hudecova2021testing} test axial symmetry
of log-returns about fixed coordinate directions. Our method lets us
ask a sharper version of the same question for a basket of carry-trade
currencies: is the joint return distribution symmetric about a direction that
is not chosen by the researcher but estimated from the data as the common
risk factor of the basket? Section~\ref{sec:forex-application} develops this
application in detail.

Recent work has proposed tests for symmetry about specified directions or
subspaces
\citep[e.g.,][]{hudecova2021testing,hudecova2022testing,hudecova2021subspace,kalina2021common},
as well as tests for invariance under known group actions in which the null
transformation is fixed a priori and does not need to be
estimated from the data
\citep[e.g.,][]{fraiman2024application,chiu2023hypothesis}. By contrast, we
consider orthogonal-invariance problems in which the reflecting subspace is
unspecified and must be inferred from the data. To the best of our knowledge, no
asymptotically justified procedure is available for this problem.

	\paragraph*{Our contribution}
	We study the null hypothesis of invariance under an orthogonal transformation
	whose fixed subspace is unspecified, with axial symmetry (reflection across a
	one-dimensional subspace) as the leading particular case. {The structural
	reduction described above shrinks the search over \(\Od\) to the finite family
	\(\{\QS_S:S\subseteq\{1,\dots,d\}\}\) of reflections through principal
	subspaces}, but it does not by itself yield a valid test: the candidate
	transformations depend on the eigenvectors of \(\Sigma\), which are unknown
	and must be estimated from the same data used for inference. Our solution is
	based on sample splitting. We partition the sample into three independent
	subsamples. One subsample is used to estimate the covariance matrix and its
	eigenvectors, and the remaining subsamples are used to construct, for each
	candidate reflection, a Kolmogorov--Smirnov-type statistic based on projected
	and centered observations. This decouples the estimation step from the
	empirical process underlying the test statistic and yields a tractable
	asymptotic analysis.
	
	To study both the original statistic and its bootstrap analogue within a
	common framework, we work in a triangular-array setting. Under suitable
	regularity conditions, we derive the asymptotic distribution of the test
	statistic for each candidate reflection and establish validity of a bootstrap
	procedure based on data-dependent probability measures that satisfy the null
	invariance condition. The resulting method is a concrete step-by-step
	procedure: estimate the principal directions from one subsample, form the
	candidate principal-subspace reflections, compute candidate-specific
	Kolmogorov--Smirnov-type statistics on independent subsamples, calibrate them
	by bootstrap, and then form a global test for orthogonal symmetry.
	
	In \cite{romano1988}, bootstrap validity is established in a wide variety of
	settings, including tests of membership in a parametric family and tests of
	invariance under the action of a given group. Our setting falls outside those
	frameworks, because here the transformation that preserves the null
	distribution depends on an unknown geometric parameter, a principal subspace,
	that must itself be estimated from the data. The orthogonal-symmetry problem
	considered here provides a concrete instance of this broader inferential
	difficulty, and the bootstrap argument developed below suggests a possible
	route for related problems involving invariance under transformations indexed
	by unknown parameters.
	
	\paragraph*{Paper organization}
	The remainder of the paper is organized as follows.
	Section~\ref{sec:setup} introduces the basic notation and the orthogonal
	reflections through principal subspaces. Section~\ref{sec:princomp} shows how
	orthogonal symmetries are linked to the eigenvectors of the covariance matrix,
	reducing the problem to a finite family of candidate reflections and motivating
	principal-component estimators. Section~\ref{sec:test} formulates the testing
	problem and introduces the proposed split-sample statistics.
	Section~\ref{sec:mainth} develops the asymptotic theory, including the
	triangular-array weak limit, the bootstrap validity result, and the resulting
	testing procedure. Section~\ref{sec:axial} specializes the general theory to
	axial symmetry, and discusses hyperplane (Householder) and central symmetry.
	Section~\ref{sec:simus} reports a simulation study focused on the axial case, and
	Section~\ref{sec:forex-application} a real-data application to hyperplane
	symmetry. The statements and proofs of the auxiliary lemmas, together with all the remaining proofs, are deferred to the Supplementary Material.

	\section{Setup}\label{sec:setup}
	
	We work in $\mathbb{R}^{d}$, $d\ge2$, endowed with the Euclidean inner product $\langle x,y\rangle$ and norm $\|x\|=\sqrt{\langle x,x\rangle}$. Let $\mathbb{S}^{d-1}=\{u\in\mathbb{R}^{d}:\|u\|=1\}$ denote the unit sphere, and let $\Od$ denote the orthogonal group in dimension $d$.
	
	Let $X$ be an $\mathbb{R}^{d}$-valued random vector defined on $(\Omega,\mathcal{A},\mathbb{P})$ with mean $\mu=\mathbb{E}[X]$ and covariance matrix $\Sigma=\cov(X)$. We assume throughout that $\mathbb{E}\|X\|^2<\infty$. We write $Z\sim P$ to indicate that the random vector $Z$ has distribution $P$, and $X\overset{\mathcal L}{=}Y$ to indicate that the random vectors $X$ and $Y$ are identically distributed; in particular, $Z\sim\mathcal{N}(0,\Sigma)$ means that $Z$ has the centered Gaussian distribution with covariance matrix $\Sigma$. We denote convergence in distribution and in probability by $\xrightarrow{\mathcal{L}}$ and $\xrightarrow{\mathbb{P}}$, respectively.
	
	Let \(f:U\subset\mathbb{R}^d\to\mathbb{R}\) be a \(C^2\) function on an open set \(U\). For \(x\in U\), we write \(\nabla f(x)\) for the gradient of \(f\) at \(x\) and \(Hf(x)\) for the Hessian matrix (i.e., the matrix of second-order partial derivatives). For any function \(g\) and any matrix \(M\), \(\|g\|_\infty\) denotes the supremum norm of \(g\) and \(\|M\|_{\mathrm{op}}\) the operator norm of \(M\). Finally, for a probability measure \(P\), \(L_2(P)\) denotes the space of square-integrable functions with respect to \(P\).
	
	\subsection{Orthogonal reflections through principal subspaces}\label{subsec:reflections}
	
	For a linear subspace \(V\subseteq\mathbb{R}^d\), let \(\PS_V\) denote the orthogonal projection onto \(V\), and define the associated orthogonal reflection
	\begin{equation}\label{EqDefQ}
		\QS_V \;\defeq\; 2\PS_V-I_d,
	\end{equation}
	where \(I_d\) is the \(d\times d\) identity matrix. The map \(\QS_V\) fixes \(V\) pointwise and reverses the sign of every vector in the orthogonal complement \(V^\perp\); its eigenspaces are \(V\) (eigenvalue \(+1\)) and \(V^\perp\) (eigenvalue \(-1\)).
	
	The matrix \(\QS_V\) enjoys the two structural properties on which our analysis relies. First, it is symmetric, \(\QS_V^\top=2\PS_V^\top-I_d=\QS_V\). Second, it is orthogonal and, in fact, an involution:
	\[
	\QS_V^2=(2\PS_V-I_d)^2=4\PS_V^2-4\PS_V+I_d=I_d,
	\qquad \QS_V^{-1}=\QS_V^\top=\QS_V,
	\]
	using \(\PS_V^2=\PS_V\). Thus \(\QS_V\) is a symmetric orthogonal involution: the orthogonal reflection with mirror subspace \(V\).
	
	We say that the random vector \(X\) is reflection symmetric about the subspace \(V\) (equivalently, invariant under \(\QS_V\)) if its centered counterpart satisfies \(X-\mu\overset{\mathcal L}{=}\QS_V(X-\mu)\).
	The subspace \(V\) is the mirror of the reflection. The complementary reflection satisfies \(\QS_{V^\perp}=-\QS_V\), and since \(\QS_{V^\perp}\QS_V=-I_d\), invariance under both \(\QS_V\) and \(\QS_{V^\perp}\) is equivalent to invariance under one of them together with central symmetry; in general, invariance under \(\QS_V\) and invariance under \(\QS_{V^\perp}\) are distinct hypotheses, and we treat the two complementary mirrors as separate candidates.
	
\begin{example}\label{ex:mirrors}
	Several classical symmetries correspond to particular choices of \(V\).
	\begin{itemize}[nosep]
		\item \(\dim V=1\), \(V=\operatorname{span}(u)\), \(u\in\mathbb S^{d-1}\): then \(\PS_V=uu^\top\) and \(\QS_V=2uu^\top-I_d=:R_u\) is the axial reflection about the line \(\operatorname{span}(u)\). This is the case studied in detail in Section~\ref{sec:axial}.
		\item \(\dim V=d-1\), \(V=u^\perp\): then \(\QS_V=I_d-2uu^\top=:H_u\) is the Householder reflection across the hyperplane \(u^\perp\). Note \(H_u=-R_u\).
		\item \(V=\{0\}\): then \(\QS_V=-I_d\) is central symmetry about the mean.
		\item \(V=\mathbb R^d\): then \(\QS_V=I_d\) is the trivial (identity) transformation.
		\item \(V=\operatorname{span}\{e_i:i\in S\}\) for a coordinate subset \(S\subseteq\{1,\dots,d\}\): then \(\QS_V\) is the diagonal matrix of \(\pm1\)'s that preserves the coordinates indexed by \(S\) and flips the sign of the remaining ones.
	\end{itemize}
\end{example}

	\section{Symmetry and principal components}\label{sec:princomp}
	
	In this section we explain how orthogonal symmetries are tied to the eigenvectors of the covariance matrix. This yields the finite reduction of the orthogonal-invariance problem to a family of candidate principal-subspace reflections, and lets us derive the asymptotic distribution of the plug-in estimators of the principal components in a transparent manner.
	
	The following proposition is the key structural fact.
	
	\begin{proposition}\label{prop:commute}
		Let $X$ be a random vector in $\mathbb{R}^{d}$ whose covariance matrix $\Sigma$ exists, and let $Q\in\Od$. If $X-\mu\overset{\mathcal L}{=}Q(X-\mu)$, then $Q$ commutes with $\Sigma$: $Q\Sigma=\Sigma Q$. If, in addition, $\Sigma$ has simple spectrum with eigenpairs $(\lambda_i,u_i)$, $i=1,\dots,d$ (so that $\lambda_1>\cdots>\lambda_d\ge0$ and $u_1,\dots,u_d$ is an orthonormal basis), then $Qu_i=\epsilon_i u_i$ with $\epsilon_i\in\{-1,+1\}$ for every $i$; equivalently,
		\[
		Q=\sum_{i=1}^d \epsilon_i\,u_iu_i^\top=2\PS_S-I_d,
		\qquad \text{where}\quad S=\{i:\epsilon_i=+1\},\quad \PS_S=\sum_{i\in S}u_iu_i^\top.
		\]
		In particular, $Q$ is the orthogonal reflection $\QS_{V_S}$ through the principal subspace $V_S=\operatorname{span}\{u_i:i\in S\}$.
	\end{proposition}
	
\begin{remark}\label{rem:no-rotations}
		A general orthogonal matrix may have complex eigenvalues \(e^{\pm i\phi}\) associated with two-dimensional invariant subspaces (proper rotations). Proposition~\ref{prop:commute} shows that such rotations cannot preserve the distribution under a simple-spectrum assumption on \(\Sigma\). Since \(Q\Sigma=\Sigma Q\), every one-dimensional eigenspace of \(\Sigma\) is \(Q\)-invariant, so \(Qu_i=\pm u_i\) for every \(i\), and \(Q\) has only real eigenvalues. Hence the only orthogonal symmetries compatible with a simple spectrum are the principal-subspace reflections \(\QS_S\), \(S\subseteq\{1,\dots,d\}\). This is what makes the reduction to a finite candidate family possible.
	\end{remark}
	
	\begin{remark}\label{rem:householder}
		The one-dimensional instance \(S=\{i\}\) yields the axial reflection \(R_{u_i}\). So an axially symmetric \(X\) has its symmetry axis among the eigenvectors of \(\Sigma\). The \((d-1)\)-dimensional instance \(S=\{1,\dots,d\}\setminus\{i\}\) yields the Householder reflection \(H_{u_i}=I_d-2u_iu_i^\top=-R_{u_i}\), so the normal \(u_i\) to a symmetry hyperplane is likewise an eigenvector of \(\Sigma\). Both are covered by Proposition~\ref{prop:commute}.
	\end{remark}
	
	As a consequence of Proposition~\ref{prop:commute}, we recover and slightly extend the link between multiple symmetry directions and the covariance spectrum.

	\begin{corollary} \label{Coro4.2}
		Let $X$ be a random vector in $\mathbb{R}^{d}$ with covariance matrix $\Sigma$. Assume that there exists a set $\{u_1,\ldots,u_k\}\subset \mathbb{S}^{d-1}$ of axial-symmetry directions of $X$ (i.e., $X$ is invariant under each $R_{u_j}$) such that the vectors $u_{j-1}$ and $u_{j}$, $j=2,\ldots,k$, are not orthogonal. Let $W=\operatorname{span}\{u_1,\dots,u_k\}$, and let $\PS_W$ denote the matrix of the orthogonal projection onto $W$. Then there exists $\delta\ge0$ such that
		$\cov(\PS_W(X-\mu))=\delta \PS_W,$
		where $\mu=\mathbb E[X]$.
	\end{corollary}
	
\begin{remark}
		The uniform distribution on any regular polygon satisfies the assumptions of Corollary~\ref{Coro4.2}. Its covariance matrix is therefore a scalar multiple of the identity, where the proportionality constant depends on the scale of the polygon. The same conclusion holds for planar figures with at least two non-orthogonal axes of symmetry, for instance five-pointed stars. Analogous phenomena also arise in higher dimensions for highly symmetric bodies, such as regular polyhedra.
	\end{remark}
	
	\begin{named}{\textbf{(H1)}}\label{assum:H1}
		All eigenvalues of the covariance matrix $\Sigma$ are simple (pairwise distinct) and positive; that is, they satisfy $\lambda_1>\cdots>\lambda_d>0$.
	\end{named}
	
	Assume \ref{assum:H1} and let $u_1,\ldots,u_d$ be the eigenvectors associated with the ordered eigenvalues $\lambda_1>\ldots>\lambda_d>0$ of $\Sigma$. Under \ref{assum:H1}, Proposition~\ref{prop:commute} shows that the possible orthogonal symmetries reduce to the finite family
	\begin{equation}\label{eq:candidate-family}
		\big\{\QS_S=2\PS_S-I_d:\ S\subseteq\{1,\dots,d\}\big\},
		\qquad \PS_S=\sum_{i\in S}u_iu_i^\top,
	\end{equation}
	{comprising at most \(2^d\) reflections (with \(S=\{1,\dots,d\}\) giving the trivial \(I_d\)).}
	
	\begin{remark}\label{rem:spectral-gap}
		Assumption~\ref{assum:H1} underlies the whole construction: the reduction to a finite candidate family, the subspace identification, and the eigenvector linearization. The theory assumes fixed positive gaps (through Assumption \textup{(T4)} below), but the procedure is expected to be sensitive to nearly repeated eigenvalues in finite samples: the principal directions are then poorly identified, and the influence functions and linearization constants involve the inverse gaps \(1/(\lambda_i-\lambda_j)\), which blow up as \(\min_{i\ne j}|\lambda_i-\lambda_j|\to0\). In such regimes the estimated reflections \(\widehat\QS_S^3\) are unstable and the test should be applied with caution.
	\end{remark}
	
	Given an i.i.d. sample $\aleph=\{X_1,\dots,X_n\}$ from $X$, the empirical covariance matrix is
	$\widehat {\Sigma}_n = \frac{1}{ n} \sum_{j} (X_j-\bar{X}_n)(X_j-\bar{X}_n)^\top,$
	where $\bar X_n$ is the sample mean; the empirical principal components are, for $i=1,\dots,d$,
	\begin{equation}\label{eq:principal_component}
		\widehat{u}_{i,n} = \arg\max\bigl\{u^\top \widehat {\Sigma}_n u:\ \|u\|=1,\ u\perp\{\widehat u_{1,n},\dots,\widehat u_{i-1,n}\}\bigr\},
	\end{equation}
	with the convention that the orthogonality constraint is vacuous for $i=1$. For a candidate subset $S$, the (full-sample) plug-in reflection is
	\[
	\widehat \QS_{S,n}=2\widehat\PS_{S,n}-I_d,
	\qquad \widehat\PS_{S,n}=\sum_{i\in S}\widehat u_{i,n}\,\widehat u_{i,n}^\top.
	\]
	The estimator \(\widehat\QS_{S,n}\) above is stated only to fix ideas: it makes the candidate family estimable from data. In the testing procedure of Section~\ref{sec:test}, the eigenvectors are estimated from a subsample \(\aleph_3\) independent of the subsamples \(\aleph_1,\aleph_2\) used to compute the test statistic, and the plug-in reflection is the split-sample version \(\widehat\QS_S^3=2\sum_{i\in S}\widehat u_{i,n}^3(\widehat u_{i,n}^3)^\top-I_d\); the independence between the estimation subsample and the testing subsamples is what makes the asymptotic analysis tractable.

	\section{Testing symmetry}\label{sec:test}
	
	In this section we propose a hypothesis test to determine whether the data come from a distribution that is invariant under a (nontrivial) orthogonal transformation with an unspecified mirror subspace.
	
	Intrinsically, the null hypothesis is that \(X\) is invariant under reflection through some subspace of an admissible collection of dimensions. Fix a set of admissible mirror dimensions \(D\subseteq\{0,1,\dots,d-1\}\) and consider
	\begin{equation}\label{hs-intrinsic}
		\begin{cases}
			H_0: \text{there exists a subspace } V\subseteq\mathbb R^d \text{ with } \dim V\in D \text{ such that}\\
			\hphantom{H_0:{}} X-\mathbb E[X]\overset{\mathcal L}{=}\QS_V(X-\mathbb E[X]),\\
			H_1: H_0 \text{ is false}.
		\end{cases}
	\end{equation}
	This formulation makes no reference to the eigenvectors of \(\Sigma\). Under Assumption~\ref{assum:H1}, however, Proposition~\ref{prop:commute} shows that any mirror subspace \(V\) achieving invariance must be spanned by eigenvectors of \(\Sigma\); hence \(V=V_S\) for some index set \(S\) with \(|S|\in D\), and \eqref{hs-intrinsic} is equivalent to the reduced hypothesis
	\begin{equation}\label{hs}
		\begin{cases}
			H_0: \text{there exists } S\in\mathcal S \text{ such that } X - \mathbb{E}[X] \overset{\mathcal{L}}{=} \QS_S(X - \mathbb{E}[X]),  \\
			H_1:   H_0 \text{ is false},
		\end{cases}
	\end{equation}
	where \(\mathcal S=\{S\subseteq\{1,\dots,d\}:|S|\in D\}\) is the induced admissible family of index sets. Taking \(D=\{0,1,\dots,d-1\}\) gives {the default family \(\mathcal S=2^{\{1,\dots,d\}}\setminus\{\{1,\dots,d\}\}\) of all nontrivial principal-subspace reflections}; \(D=\{1\}\) gives \(\mathcal S=\{\{i\}:i=1,\dots,d\}\), i.e.\ axial symmetry (Section~\ref{sec:axial}); \(D=\{d-1\}\) gives \(\mathcal S=\{\{1,\dots,d\}\setminus\{i\}:i=1,\dots,d\}\), i.e.\ hyperplane symmetry; and \(D=\{0\}\) gives central symmetry. Because \(\mathcal S\) is defined through a set of admissible dimensions, it is closed under the reduction of Proposition~\ref{prop:commute}, so the equivalence between \eqref{hs-intrinsic} and \eqref{hs} is exact. In order to test \eqref{hs}, it is therefore enough to examine the candidate reflections \(\{\QS_S:S\in\mathcal S\}\).
	
	This reduction is naturally combined with the random-projection principle of \cite{cuesta2007sharp}. Indeed, their sharp form of the Cram\'er--Wold theorem shows that, if the absolute moments $\E\|X-\mu\|^n$ are finite and satisfy Carleman's condition
	\begin{equation}\label{carleman}
		\sum_{n\ge1} (\E\|X-\mu\|^n)^{-1/n}=\infty,
	\end{equation}
	then \(\mathcal L(X-\mu)\) is moment-determinate. Moreover, if two probability laws on \(\mathbb R^d\) are different, the set of directions on \(\mathbb S^{d-1}\) for which their one-dimensional projections coincide is contained in a projective hypersurface, and therefore has surface measure zero. Carleman's condition \eqref{carleman} is invariant under orthogonal transformations: since \(\QS_S\) is orthogonal, \(\|\QS_S(X-\mu)\|=\|X-\mu\|\), so \(\mathcal L(\QS_S(X-\mu))\) is moment-determinate whenever \(\mathcal L(X-\mu)\) is. This is what makes the identification symmetric in the two laws being compared; a moment-determinacy assumption of this type is needed for the random-projection statistic to characterize equality of the full laws, and it is imposed in the consistency result (Theorem~\ref{thm:power}).
	
	In particular, if \(\mathcal L(X-\mathbb E[X])\neq \mathcal L(\QS_S(X-\mathbb E[X]))\), then, by the sharp Cram\'er--Wold theorem, for surface-a.e.\ \(h\in\mathbb S^{d-1}\), the one-dimensional projected laws of \(\langle X-\mathbb E[X],h\rangle\) and \(\langle \QS_S(X-\mathbb E[X]),h\rangle\) are different. On the other hand, whenever \(X\) is invariant under \(\QS_S\), these projected distributions coincide for every \(h\in\mathbb S^{d-1}\); see \cite{cuesta2006,cuesta2007sharp}.
	
	Motivated by this fact, for each \(h\in\mathbb S^{d-1}\) and each orthogonal matrix \(Q\) we measure the discrepancy between the projected distributions through
	\begin{equation*}
		g_h(Q)
		=
		\sup_{t\in\mathbb R}
		\left|
		F_{\langle X-\mathbb E[X],h\rangle}(t)
		-
		F_{\langle Q(X-\mathbb E[X]),h\rangle}(t)
		\right|.
	\end{equation*}
	
	To construct our test statistic, we randomly divide the sample \(\aleph=\{X_1,\dots,X_n\}\) into three nearly balanced subsets \(\aleph_1\cup\aleph_2\cup\aleph_3\). Let \(m_n:=\lfloor n/3\rfloor\), with $|\aleph_1|=|\aleph_2|=m_n,$ $|\aleph_3|=n-2m_n.$ The subsample \(\aleph_3\) is used to estimate the principal directions (hence the candidate reflections), while \(\aleph_1\) and \(\aleph_2\) are used for the two-sample test.
	
	Given $h \in \mathbb{S}^{d-1}$ and an orthogonal matrix $Q$, we consider
	\begin{align}\label{empirical}
		{  \widehat {F}^1_{n,h}(t):}&= \frac{1}{m_n} \sum_{X_i \in \aleph_1} \ind\{\langle X_i-\bar{X}_n^1, h \rangle \leq t\}, \\
		{  \widehat {F}^2_{n,h,Q}(t):}&= \frac{1}{m_n} \sum_{X_i \in \aleph_2} \ind\{\langle Q(X_i-\bar{X}_n^2), h \rangle \leq t\}, \\
		\widehat  g_{n,h}(Q) &=  \left\| {  \widehat {F}^1_{n,h}}- {  \widehat {F}^2_{n,h,Q}}
		\right\|_\infty. 
	\end{align}
	
	To define the test statistic for \eqref{hs}, we use Proposition~\ref{prop:commute}, which shows that the mirror subspace of any symmetry, if one exists, must be spanned by eigenvectors of \(\Sigma\). Let \(\widehat u_{i,n}^3\) be the estimator, computed from \(\aleph_3\) only, of the eigenvector of \(\Sigma\) associated with \(\lambda_i\), as defined in \eqref{eq:principal_component}, and set \(\widehat \QS_{S}^3=2\sum_{i\in S}\widehat u_{i,n}^3(\widehat u_{i,n}^3)^\top-I_d\). For each \(S\in\mathcal S\), we consider the candidate-specific statistic $T_{n,S}(h)=\sqrt n\,\widehat  g_{n,h}(\widehat \QS_{S}^3)$. Its null distribution is approximated by bootstrap, which yields the corresponding candidate-specific bootstrap \(p\)-value.

	\section{Main theoretical results}\label{sec:mainth}
	
	This section provides the asymptotic distribution of our split-sample statistic \(\sqrt{n}\,\widehat g_{n,h}(\widehat \QS_{S}^3)\) under \(H_0\), and sets the stage for the bootstrap validity results proved later on.
	
	The limit distribution of \(\sqrt{n}\,\widehat g_{n,h}(\widehat \QS_{S}^3)\) depends on unknown features of \(P\) (projection densities, covariance structure, and the eigenstructure driving the asymptotics of the principal component estimators). Therefore, to implement the test we approximate critical values by bootstrap.
	
	To prove the consistency of this bootstrap approximation, it is natural (and technically convenient) to embed both the original sampling scheme and the bootstrap world into a single framework in which the underlying distribution may depend on \(n\). This is precisely the triangular-array setting of \cite{romano1988}, which allows us to apply Proposition~A.1 therein.
	
	More precisely, we consider a sequence of laws \(\{P_n\}_{n\ge1}\) on \(\mathbb R^d\) and, for every \(n\in\mathbb N\), an i.i.d. sample \(\{X_{n,1},\ldots,X_{n,n}\}\) taken from \(P_n\). The fixed-law case is recovered by setting \(P_n\equiv P\). In the methodological sections, we denote the original i.i.d.\ sample by \(X_1,\dots,X_n\). In the triangular-array arguments developed below, we write \(\{X_{n,1},\dots,X_{n,n}\}\) for the \(n\)-th row sampled from \(P_n\). In the smoothed bootstrap, the conditional resampling distribution is itself \(n\)-dependent (e.g., a kernel-smoothed version of the symmetrized empirical law), and it is precisely here that smoothness assumptions on \(P_n\) become essential. Throughout, let \(X_{n,1}\sim P_n\) and \(X\sim P\), and write \(\mu_n=\mathbb E[X_{n,1}]\).
	
	The next assumptions collect the regularity conditions needed for the three main ingredients of the argument: weak convergence of the centered empirical processes, linearization of the empirical principal components, and differentiability of the reflection map under the smoothed bootstrap laws.
	
	\begin{named}{\textbf{(T1)--(T4), (T2$'$)}}\label{assum:T1-T5}
		\begin{enumerate}
			\item[]
			\item[(T1)]\label{assum:T1}  $\sup_n \mathbb{E}\|X_{n,1}\|^{4+\eta}<\infty$ for some $\eta>0$.
			\item[(T2)] 
			For every \(a\in\mathbb S^{d-1}\), let \(f_{n,a}\) denote the density of \(a^\top(X_{n,1}-\mu_n)\) under \(P_n\), and let \(f_a\) denote the density of \(a^\top(X-\mu)\). Assume that
$K_0:=\sup_{n\ge1}\sup_{a\in\mathbb S^{d-1}}\|f_{n,a}\|_\infty<\infty,$ that there exists a deterministic sequence $K_{1,n}=o(\sqrt n)$ such that $\sup_{a\in\mathbb S^{d-1}}\|f'_{n,a}\|_\infty\le K_{1,n}$, and that $\sup_{a\in\mathbb S^{d-1}}\|f_a\|_\infty<\infty$ and $\sup_{a\in\mathbb S^{d-1}}\|f_{n,a}-f_a\|_\infty\to0.$
			\item[(T2$'$)]\label{assum:T2prime}
			The projected densities are Lipschitz in the direction, uniformly in the level: there exists a deterministic sequence \(L_n=o(\sqrt n)\) such that \(\sup_{b\in\mathbb R}|f_{n,a}(b)-f_{n,a'}(b)|\le L_n\|a-a'\|\) for all \(a,a'\in\mathbb S^{d-1}\).
			\item[(T3)]  Let $f_n$ and $f$ denote the densities of $P_n$ and $P$. For each $n$, assume $f_n \in C^2(\mathbb{R}^d)$, $f \in C^1(\mathbb{R}^d)$, $\int_{\mathbb R^d}(1+\|x\|)\|\nabla f(x)\| dx<\infty$, and $\int_{\mathbb R^d}(1+\|x\|)\|\nabla f_n(x)-\nabla f(x)\| dx \to 0$; and there exists a deterministic sequence \(A_n=o(\sqrt n)\) with $\int_{\mathbb R^d}(1+\|x\|^2)\|Hf_n(x)\|_{\mathrm{op}} dx \le A_n$ for all $n$.
			\item[(T4)]\label{assum:T4}
			Let $\Sigma_n= \cov(X_{n,1})$ (the covariance matrix of $P_n$) and assume $\|\Sigma_n-\Sigma\|_{\mathrm{op}}\to0$ for some positive definite matrix $\Sigma$ whose eigenvalues are simple: $\lambda_1>\cdots>\lambda_d>0$. Let $u_1,\dots,u_d$ be the associated unit eigenvectors of $\Sigma$.
		\end{enumerate}
	\end{named}
	
	Assumption \textup{(T4)} has a consequence that fixes notation used throughout: since $\Sigma_n\to\Sigma$ and $\Sigma$ has simple spectrum, $\Sigma_n$ has simple spectrum for all sufficiently large $n$; for such $n$, we let $u_{1,n},\dots,u_{d,n}$ denote the unit eigenvectors associated with its ordered eigenvalues, with signs fixed by $\langle u_{j,n},u_j\rangle\ge0$. The finitely many remaining indices may be completed by arbitrary orthonormal eigenbases, which do not affect any asymptotic statement.
	
	Conditions \textup{(T2)} and \textup{(T2$'$)} are not alternative formulations of the same requirement: \textup{(T2$'$)} supplements \textup{(T2)}, and is used only to control the perturbation of the projection direction that arises when the estimated reflection \(\widehat\QS_S^3\) replaces \(\QS_{S,n}\) inside the empirical process (see the proof of Theorem~\ref{teo:mainth1-TA} in the Supplementary Material), whereas several auxiliary results (for instance the weak-convergence lemma for the two centered split empirical processes) require \textup{(T2)} but not \textup{(T2$'$)}. For this reason the two conditions are stated, and cited, separately.
	
	To prepare the asymptotic analysis, we next introduce some empirical-process notation that will be used throughout this section. Since our statistic is defined through suprema of differences between projected distribution functions, the natural indexing class is the family of half-spaces in \(\mathbb R^d\), and, for a fixed direction \(h\), its corresponding one-dimensional subfamily. This framework is convenient both for the weak-convergence result below and for the bootstrap argument developed later. In particular, the main theorem is stated in a triangular-array form because this is precisely the setting needed for the smoothed bootstrap, where the bootstrap sample is generated from a data-dependent probability measure \(P_n\).
	
	Let $\mathcal F=\{\ga_{a,b}(x)=\mathbf 1\{a^\top x\le b\}: a\in\mathbb S^{d-1},\ b\in\mathbb R\},$ and, for \(h\in\mathbb S^{d-1}\), \(\mathcal F_h= \{\ga_{h,b}(x): b\in\mathbb R\}.\) Then, \(\mathcal F = \cup_h \mathcal F_h\). Given two probability distributions $P$ and $P'$ we denote \(d_{\mathcal F}(P,P')=\sup_{\ga\in\mathcal F}|P\ga - P'\ga|=\|P-P'\|_\mathcal{F}\).
	For any bounded map $z:\mathcal F\to\mathbb R$ and any subclass $\mathcal G\subset\mathcal F$, we write $\|z\|_{\mathcal G}=\sup_{\ga\in\mathcal G}|z(\ga)|$; a probability measure $P$ is identified with the map $\ga\mapsto P\ga$.
	
	All stochastic-process limits are taken in the space $\ell^\infty(\mathcal F)$ endowed with the $\sigma$-field generated by open balls; suprema are understood in outer probability whenever measurability could be an issue. As is standard, one may also avoid measurability concerns by working with a countable dense subfamily.
	
	In what follows $\{P_n\}_{n\geq 1}$ is a sequence of probability measures on $\mathbb{R}^d$ such that $d_\mathcal{F}(P_n,P)\to 0$. Theorem~\ref{teo:mainth1-TA} is our main asymptotic result: it establishes the weak convergence of the split statistic in the triangular-array framework (which is the one needed for the smoothed bootstrap arguments below). In particular, in the special case \(P_n\equiv P\), the theorem yields the asymptotic distribution of \(\sqrt n\,\widehat g_{n,h}(\widehat \QS_{S}^3)\) under \(H_0\).
	
	For an orthogonal matrix $Q$ and $\ga\in\mathcal F$ we write $(Q\ga)(x):=\ga(Qx)$; since $(Q\ga_{a,b})(x)=\mathbf 1\{a^\top Qx\le b\}=\ga_{Q^\top a,b}(x)$, and $\|Q^\top a\|=1$, the map $\ga\mapsto Q\ga$ is a bijection of $\mathcal F$ onto itself (for the symmetric reflections used throughout, $Q^\top=Q$). Let $P_n^c$ denote the law of $X_{n,1}-\mu_n$, and for $r=1,2$ write $\widehat P_n^{c,r}(\ga)=|\aleph_{n,r}|^{-1}\sum_{j\in\aleph_{n,r}}\ga(X_{n,j}-\bar X_n^r)$, $\ga\in\mathcal F$. The next theorem is the key probabilistic input of the paper: it identifies the weak limit of the split statistic for a fixed candidate reflection and provides the limit law that the bootstrap must reproduce.
	
	\begin{theorem} \label{teo:mainth1-TA}
		Fix $S\in\mathcal S$. Let $\aleph=\{X_{n,1},\dots,X_{n,n}\}$ be an i.i.d. sample of $P_n$, and assume \textup{(T1)}--\textup{(T4)} and \textup{(T2$'$)}. Suppose that, for all sufficiently large $n$, the centered law $P_n^c$ is invariant under the reflection $\QS_{S,n}:=2\sum_{i\in S}u_{i,n}u_{i,n}^\top-I_d$ (equivalently, $P_n$ is invariant under the affine reflection $x\mapsto\mu_n+\QS_{S,n}(x-\mu_n)$). We assume that $d_\mathcal{F}(P_n,P)\to 0$. For $i\in S$, let $\widehat u_{i,n}^3$ be the $i$-th empirical principal component computed only from $\aleph_{n,3}$, with sign chosen so that $\langle \widehat u_{i,n}^3,u_{i,n}\rangle\ge 0$, and set $\widehat\QS_S^3=2\sum_{i\in S}\widehat u_{i,n}^3(\widehat u_{i,n}^3)^\top-I_d$. Let $\widehat  g_{n,h}$ be the statistic defined in Section~\ref{sec:test}, and write $\mathbb T_{n,S}(\ga)=\sqrt n\bigl(\widehat  P_n^{c,1}(\ga)-\widehat  P_n^{c,2}(\widehat\QS_S^3\ga)\bigr)$ for the test process indexed by the class $\mathcal F$ of all half-space indicators. Then $\mathbb T_{n,S}\xrightarrow{\mathcal L}\mathbb L^{S}$ in $\ell^\infty(\mathcal F)$, where $\mathbb L^{S}$ is the centered Gaussian limit \eqref{eq:limit-representation} below, with covariance kernel \eqref{eq:cov-kernel}, obtained with the limit distribution $P$ and the limiting mirror subspace $V_S=\operatorname{span}\{u_i:i\in S\}$ determined by \textup{(T4)}. In particular, since $\sqrt n\,\widehat  g_{n,h}(\widehat\QS_S^3)=\sup_{\ga\in\mathcal F_h}|\mathbb T_{n,S}(\ga)|$ and $z\mapsto\sup_{\ga\in\mathcal F_h}|z(\ga)|$ is continuous on $\ell^\infty(\mathcal F)$, for every fixed $h\in\mathbb{S}^{d-1}$, $\sqrt{n}\,\widehat  g_{n,h}(\widehat\QS_S^3)\xrightarrow{\mathcal L}\|\mathbb L_h^{S}\|_{\mathcal{F}_h}$, where $\mathbb L_h^{S}$ denotes the restriction of $\mathbb L^{S}$ to $\mathcal F_h$.
	\end{theorem}
	
	The sign conventions \(\langle u_{j,n},u_j\rangle\ge0\) and \(\langle\widehat u_{i,n}^3,u_{i,n}\rangle\ge0\) are immaterial for the reflections themselves, which depend on the eigenvectors only through the projections \(u_iu_i^\top\); they are needed only so that the differences \(\widehat u_{i,n}^3-u_{i,n}\), and hence the eigenvector linearization carried out in the Supplementary Material, are well defined.
	
	The proof is given in the Supplementary Material; here we show the structure of the limit, which is what makes the bootstrap tractable. Writing \(\QS_S=2\sum_{i\in S}u_iu_i^\top-I_d\) for the limiting reflection, the process admits the representation
	\begin{equation}\label{eq:limit-representation}
		\mathbb L^{S}(\ga)=\mathbb L^{(1)}(\ga)-\mathbb L^{(2)}(\ga)-\sum_{i\in S}G_{u_i}^{S}(\ga)^\top Z_i^3,
		\qquad \ga\in\mathcal F,
	\end{equation}
	where \(\mathbb L^{(1)},\mathbb L^{(2)}\) are the independent Gaussian limits of the two centered split empirical processes, \((Z_i^3)_{i\in S}\) is a jointly Gaussian vector (the limit of the principal-component linearizations), independent of \((\mathbb L^{(1)},\mathbb L^{(2)})\), and
	\begin{equation}\label{eq:G-general}
		G_{u_i}^{S}(\ga)=2\int_{\mathbb R^d}\ga(x)\,\bigl((u_i^\top x)I_d+xu_i^\top\bigr)\,\nabla f(\QS_S x+\mu)\,dx,\qquad i\in S.
	\end{equation}
	The sum over \(i\in S\) in \eqref{eq:limit-representation} is the linearization term produced by the estimation of the mirror frame, with one summand for each estimated principal direction; when \(|S|=1\) it collapses to a single summand, and \eqref{eq:limit-representation}--\eqref{eq:G-general} recover the axial-symmetry formulas of Section~\ref{sec:axial}.
	
	The covariance kernel of \(\mathbb L_h^{S}\) is obtained directly from \eqref{eq:limit-representation}. Writing \(s(\ga_j)=f_h(b_j)h\) and \(m(\ga_j)=\mathbb E[(X-\mu)(\ga_j(X-\mu)-P^c\ga_j)]\) for \(\ga_j=\ga_{h,b_j}\in\mathcal F_h\), and recalling that
\(\psi_i(x)=\sum_{\ell\ne i}\frac{(u_\ell^\top(x-\mu))(u_i^\top(x-\mu))}{\lambda_i-\lambda_\ell}u_\ell\),
using that \(\mathbb L^{(1)},\mathbb L^{(2)},(Z_i^3)_{i\in S}\) are mutually independent with \(\mathbb L^{(r)}(\ga)=\sqrt3(\mathbb B^{r,c}(\ga)+s(\ga)^\top Z^{r})\), we obtain
	\begin{multline}\label{eq:cov-kernel}
		\mathrm{Cov}\big(\mathbb L_h^{S}(\ga_1),\mathbb L_h^{S}(\ga_2)\big)
		=
		6\Big(P^c(\ga_1\ga_2)-P^c\ga_1 P^c\ga_2
		+s(\ga_1)^\top \Sigma s(\ga_2)\\
		+s(\ga_1)^\top m(\ga_2)
		+s(\ga_2)^\top m(\ga_1)\Big)\\
		+3\sum_{i,j\in S} G_{u_i}^{S}(\ga_1)^\top P(\psi_i\psi_j^\top) G_{u_j}^{S}(\ga_2),
	\end{multline}
	where the factor \(6\) reflects the two independent split processes (each of asymptotic size \(n/3\), entering with opposite signs), and the double sum collects the correlated contributions of the mirror-direction linearizations. For \(|S|=1\) this reduces to the single-direction kernel of Section~\ref{sec:axial}. The derivation is given in the Supplementary Material.
	
	Because Theorem~\ref{teo:mainth1-TA} gives convergence of the whole process \(\mathbb T_{n,S}\) in \(\ell^\infty(\mathcal F)\), any fixed finite collection of projection directions is covered by the continuous mapping theorem, with all cross-direction dependence (in particular the common eigenvector shift \(\sum_{i\in S}G_{u_i}^{S}(\cdot)^\top Z_i^3\)) already encoded in the single limit \(\mathbb L^{S}\).
	
	\begin{corollary}\label{cor:finite-projections}
		Fix \(k\ge1\) and directions \(h_1,\dots,h_k\in\mathbb S^{d-1}\), and define the aggregated statistic \(T_{n,S}^{(k)}:=\frac1k\sum_{\ell=1}^k\sqrt n\,\widehat g_{n,h_\ell}(\widehat\QS_S^3)=\frac1k\sum_{\ell=1}^k\sup_{\ga\in\mathcal F_{h_\ell}}|\mathbb T_{n,S}(\ga)|\). Under the assumptions of Theorem~\ref{teo:mainth1-TA}, \(T_{n,S}^{(k)}\xrightarrow{\mathcal L}M_S^{(k)}:=\frac1k\sum_{\ell=1}^k\|\mathbb L^{S}\|_{\mathcal F_{h_\ell}}\), and the distribution function of \(M_S^{(k)}\) is continuous. Consequently, the bootstrap and \(p\)-value arguments of Theorem~\ref{thm:bootstrap} apply verbatim to the aggregated bootstrap statistic \(T_{n,S}^{(k),\dagger}=\frac1k\sum_{\ell=1}^k\sup_{\ga\in\mathcal F_{h_\ell}}|\mathbb T_{n,S}^{\dagger}(\ga)|\) and the aggregated \(p\)-value \(p_{n,S}^{(k)}=\mathbb P^\dagger(T_{n,S}^{(k),\dagger}\ge T_{n,S}^{(k)})\) (a single bootstrap \(p\)-value for the aggregate, not an average of \(k\) separate \(p\)-values). If \(h_1,\dots,h_k\) are drawn independently of the data, once, and kept fixed for the observed statistic and all bootstrap replicates, then the conclusions hold conditionally on their realized values, and hence also unconditionally.
	\end{corollary}
	
	\begin{remark}\label{rem:global-test}
		Under Assumption~\ref{assum:H1}, the candidate symmetries reduce to the reflections \(\{\QS_S:S\in\mathcal S\}\). For each \(S\in\mathcal S\), consider the null hypothesis
		$H_{0,S}:\ X-\mathbb E[X]\overset{\mathcal L}{=}\QS_S(X-\mathbb E[X]).$ Then the global null hypothesis in \eqref{hs} can be written as \(H_0=\cup_{S\in\mathcal S} H_{0,S}\). Accordingly, the natural global test rejects \(H_0\) only when all candidate null hypotheses are rejected. If, for each \(S\in\mathcal S\), \(p_{n,S}\) is an asymptotically valid \(p\)-value for testing \(H_{0,S}\), in the sense that
		$\limsup_{n\to\infty}\mathbb P\bigl(p_{n,S}\le \alpha\bigr)\le \alpha$ whenever $H_{0,S}\text{ holds},$ then
		\(p_n^{\mathrm{glob}}=\max_{S\in\mathcal S} p_{n,S}\) is an asymptotically valid \(p\)-value for the global null \(H_0\). Indeed, under \(H_0\) there exists at least one \(S_0\in\mathcal S\) such that \(H_{0,S_0}\) holds, and therefore
		\[
		\mathbb P\bigl(p_n^{\mathrm{glob}}\le \alpha\bigr)
		=
		\mathbb P\Bigl(\max_{S\in\mathcal S} p_{n,S}\le \alpha\Bigr)
		\le
		\mathbb P\bigl(p_{n,S_0}\le \alpha\bigr).
		\]
		Taking \(\limsup\) yields \(\limsup_{n\to\infty}\mathbb P(p_n^{\mathrm{glob}}\le \alpha)\le \alpha\). Hence no multiple-testing correction is required for asymptotic validity of the global test, regardless of the cardinality of \(\mathcal S\). However, the intersection structure (reject only if all candidates are rejected) can be conservative in finite samples, especially when the candidate family \(\mathcal S\) is large.
	\end{remark}
	
	\paragraph*{Algorithm 1: Testing orthogonal symmetry about an unspecified subspace}\label{alg:algoritmo}
	Select \(h\in\mathbb S^{d-1}\) at random, independently of the data, and keep it fixed. Fix the admissible family \(\mathcal S\).
	\begin{enumerate}
		\item Randomly split \(\aleph=\{X_{1},\dots,X_{n}\}\) into three balanced subsamples \(\aleph_1\cup\aleph_2\cup\aleph_3.\)
		\item Using only \(\aleph_3\), compute the empirical covariance matrix and its empirical eigenvectors \(\widehat u_{1,n}^3,\dots,\widehat u_{d,n}^3\).
		\item For each \(S\in\mathcal S\), form \(\widehat\QS_S^3=2\sum_{i\in S}\widehat u_{i,n}^3(\widehat u_{i,n}^3)^\top-I_d\) and compute the candidate-specific statistic \(T_{n,S}(h)=\sqrt n\,\widehat g_{n,h}(\widehat\QS_S^3)\), where \(\widehat g_{n,h}(Q)\) is defined in \eqref{empirical} from the independent subsamples \(\aleph_1,\aleph_2\).
		\item For each \(S\in\mathcal S\), use the bootstrap procedure described below to approximate the null distribution of \(T_{n,S}(h)\) under \(H_{0,S}\), and compute the corresponding bootstrap \(p\)-value \(p_{n,S}\).
		\item Form the global \(p\)-value \(p_n^{\mathrm{glob}}=\max_{S\in\mathcal S} p_{n,S}\), and reject the global null hypothesis in \eqref{hs} at level \(\alpha\) whenever \(p_n^{\mathrm{glob}}\le \alpha.\)
	\end{enumerate}

	\subsection{Bootstrap calibration and asymptotic validity}\label{subsec:bootstrap-validity-romano}
	
	We now describe the bootstrap calibration step in Algorithm~1 for a fixed candidate reflection \(S\in\mathcal S\). Assume that \(H_{0,S}\) holds, that is, the centered law \(P^c\) is invariant under \(\QS_S\), and let \(X_{1},\dots,X_{n}\) be i.i.d.\ from \(P\). Our goal is to approximate the null distribution of \(T_{n,S}(h)=\sqrt n\,\widehat g_{n,h}(\widehat\QS_S^3)\) and thereby construct the candidate-specific bootstrap \(p\)-value used in Step~4 of Algorithm~1.
	
	To this end, we construct from the original sample a data-dependent probability measure \(P_n\) on \(\mathbb R^d\) such that \(P_n\) satisfies the null invariance condition, also satisfies \textup{(T1)}--\textup{(T4)} and \textup{(T2$'$)}, and converges to \(P\) in the \(d_{\mathcal F}\)-metric as \(n\to\infty\). To avoid ambiguity, note that in the abstract result \(P_n\) denotes a generic bootstrap law, whereas in the present application the bootstrap distribution depends on \(S\), because it is symmetrized about the candidate reflection \(\widehat\QS_S^3\); thus, for each fixed \(S\), \(P_n\) should be read as \(P_{n,S}\). The construction of \(P_n\) is as follows. First, partition the original sample into three subsets. Using one third of the sample, estimate the covariance matrix and the candidate reflection \(\widehat\QS_S^3\). Let \(I_n\) denote the indices of the remaining two thirds of the sample, and let \(\widehat\mu_n\) be their empirical mean. Define, for $m\in\mathbb R^d$ and $Q\in\Od$, the affine reflection $A_{m,Q}(x):=m+Q(x-m)$; the symmetrization uses $A_{\widehat\mu_n,\widehat\QS_S^3}$, and we set
	\[
	\widehat P_n^{\mathrm{sym}}
	=
	\frac1{2|I_n|}
	\sum_{j\in I_n}
	\left\{
	\delta_{X_j}+
	\delta_{A_{\widehat\mu_n,\widehat\QS_S^3}(X_j)}
	\right\}.
	\]
	Since \(\widehat\QS_S^3\) is a symmetric orthogonal involution, \(A_{\widehat\mu_n,\widehat\QS_S^3}\) is its own inverse and \(\widehat P_n^{\mathrm{sym}}\) is exactly invariant under the affine reflection \(A_{\widehat\mu_n,\widehat\QS_S^3}\). Finally, define \(P_n=\widehat P_n^{\mathrm{sym}}*K_{b_n}\), where \(K\) is an orthogonally invariant (radial) kernel, i.e.\ \(K(Ox)=K(x)\) for all \(O\in\Od\). Orthogonal invariance of \(K\) is precisely what makes the convolution preserve the reflection symmetry: if \(Y\sim\widehat P_n^{\mathrm{sym}}\) and \(\xi\sim K\) are independent, then \(A_{\widehat\mu_n,\widehat\QS_S^3}(Y+b_n\xi)=A_{\widehat\mu_n,\widehat\QS_S^3}(Y)+b_n\widehat\QS_S^3\xi\), and both \(A_{\widehat\mu_n,\widehat\QS_S^3}(Y)\sim\widehat P_n^{\mathrm{sym}}\) and \(\widehat\QS_S^3\xi\sim K\), so \(P_n\) inherits the symmetry (merely fixing \(\cov(K)=I_d\) would not suffice). The convolution also provides the regularity required in \textup{(T2)}--\textup{(T3)}; see Remark~\ref{remsmoothPN} below for details.
	
	Now generate an i.i.d.\ bootstrap sample \(X_1^\dagger,\dots,X_n^\dagger\) from \(P_n\). Split it into three balanced blocks \(\aleph_{n,1}^\dagger\cup\aleph_{n,2}^\dagger\cup\aleph_{n,3}^\dagger,\) and let \(\bar X_{n,r}^{\dagger}\) be the sample mean over \(\aleph_{n,r}^\dagger\), \(r=1,2\). For \(i\in S\), let \(\widehat u_{i,n}^{3,\dagger}\) be the \(i\)-th empirical principal component computed from \(\aleph_{n,3}^\dagger\), with sign chosen so that \(\langle \widehat u_{i,n}^{3,\dagger},\widehat u_{i,n}^3\rangle\ge 0\), and set \(\widehat\QS_S^{3,\dagger}=2\sum_{i\in S}\widehat u_{i,n}^{3,\dagger}(\widehat u_{i,n}^{3,\dagger})^\top-I_d\).
	
	For \(r=1,2\), define
	\[
	\widehat  P_n^{c,r,\dagger}(\ga)
	=
	\frac{1}{|\aleph_{n,r}^\dagger|}
	\sum_{X_j^\dagger\in\aleph_{n,r}^\dagger}
	\ga\bigl(X_j^\dagger-\bar X_{n,r}^{\dagger}\bigr),
	\qquad \ga\in\mathcal F,
	\]
	and the bootstrap process $\mathbb T_{n,S}^\dagger(\ga)=\sqrt n(\widehat  P_n^{c,1,\dagger}(\ga)-\widehat  P_n^{c,2,\dagger}(\widehat\QS_S^{3,\dagger}\ga)),$ $\ga\in\mathcal F.$ For a randomly chosen fixed \(h\in\mathbb S^{d-1}\), define the associated bootstrap statistic by
	\[
	T_{n,S}^\dagger(h)
	=
	\sqrt n\,\widehat  g_{n,h}^\dagger(\widehat\QS_S^{3,\dagger})
	=
	\sup_{\ga\in\mathcal F_h}\bigl|\mathbb T_{n,S}^\dagger(\ga)\bigr|.
	\]
	
	{\sloppy Let \(J_{n,S}(x)=\mathbb P(T_{n,S}(h)\le x),\) \(J(x)=\mathbb P(\|\mathbb L_h^{S}\|_{\mathcal F_h}\le x),\) and $J_{n,S}^\dagger(x)=\mathbb P \left(T_{n,S}^\dagger(h)\le x \mid X_1,\dots,X_n\right).$ Let $c_{1-\alpha}:=\inf\{x:J(x)\ge1-\alpha\}$ and \(d_{n,S}^\dagger(1-\alpha)=\inf\{x:J_{n,S}^\dagger(x)\ge 1-\alpha\}\). The corresponding bootstrap \(p\)-value is defined by\par}
	\[
	p_{n,S}
	=
	\mathbb P \left(
	T_{n,S}^\dagger(h)\ge T_{n,S}(h)
	\mid X_1,\dots,X_n
	\right)
	=
	1-J_{n,S}^\dagger\bigl(T_{n,S}(h)-\bigr),
	\]
	where \(J_{n,S}^\dagger(x-)=\lim_{y\uparrow x}J_{n,S}^\dagger(y)\).
	
	In order to apply Proposition~A.1 of \cite{romano1988}, it is convenient to single out the class of deterministic sequences of laws satisfying the conditions required by our triangular-array limit theorem.
	
	\begin{definition}
		For a fixed \(S\in\mathcal S\), let \(\mathcal C_S(P)\) denote the class of deterministic sequences \((\nu_n)_{n\ge1}\) of probability measures on \(\mathbb R^d\) such that: (i) \(d_{\mathcal F}(\nu_n,P)\to0\); (ii) for all sufficiently large \(n\), the centered law \(\nu_n^c\) is invariant under the reflection \(\QS_{S,n}(\nu_n):=2\sum_{i\in S}u_{i,n}(\nu_n)u_{i,n}(\nu_n)^\top-I_d\) determined by the ordered eigenvectors of \(\Sigma(\nu_n)\); (iii) \((\nu_n)\) satisfies \textup{(T1)}--\textup{(T4)} and \textup{(T2$'$)}.
	\end{definition}
	
	\begin{theorem}\label{thm:bootstrap}
		Fix \(S\in\mathcal S\). Assume that the fixed law \(P\) satisfies \(H_{0,S}\) and the fixed-law versions (\(P_n\equiv P\)) of \textup{(T1)}--\textup{(T4)} and \textup{(T2$'$)}, so that Theorem~\ref{teo:mainth1-TA} applies to the original statistic. Let \(P_n\) be a data-dependent probability measure on \(\mathbb R^d\), measurable with respect to \(X_1,\dots,X_n\). Assume that there exist events \(\Omega_n\in\sigma(X_1,\dots,X_n)\) and random probability measures \(\widetilde P_n\), measurable with respect to \(X_1,\dots,X_n\), such that $\mathbb P(\Omega_n)\to1$, $\widetilde P_n=P_n$ on \(\Omega_n\), and the realized sequence \((\widetilde P_n)_{n\ge1}\) belongs to \(\mathcal C_S(P)\) almost surely. Then $\|J_{n,S}-J\|_\infty\to0,$ $\|J_{n,S}^\dagger -J\|_\infty \to0$ in probability, and hence $\|J_{n,S} -J_{n,S}^\dagger\|_\infty\to0$ in probability. In particular, the bootstrap \(p\)-value \(p_{n,S}\) is asymptotically valid: $\limsup_{n\to\infty}\mathbb P\bigl(p_{n,S}\le\alpha\bigr)\le\alpha.$ If, moreover, \(c_{1-\alpha}\) is the unique \((1-\alpha)\)-quantile of \(J\), then \(d_{n,S}^\dagger(1-\alpha)\xrightarrow{\mathbb P}c_{1-\alpha}\) and \(\mathbb P\bigl(T_{n,S}(h)>d_{n,S}^\dagger(1-\alpha)\bigr)\to\alpha\); in general, \(\limsup_{n}\mathbb P\bigl(T_{n,S}(h)>d_{n,S}^\dagger(1-\alpha)\bigr)\le\alpha\).
	\end{theorem}
	
	\begin{corollary}\label{cor:global-pvalue}
		Assume that the global null hypothesis \(H_0\) is satisfied. Let \(S_0\in\mathcal S\) be such that \(H_{0,S_0}\) holds and the assumptions of Theorem~\ref{thm:bootstrap} are satisfied for that index. Then the global \(p\)-value \(p_n^{\mathrm{glob}}=\max_{S\in\mathcal S} p_{n,S}\) is asymptotically valid for testing \(H_0\), that is, \(\limsup_{n\to\infty}\mathbb P\bigl(p_n^{\mathrm{glob}}\le \alpha\bigr)\le \alpha.\)
	\end{corollary}
	
	We next explain why the symmetrized and smoothed bootstrap construction described above fits the abstract framework of Theorem~\ref{thm:bootstrap}.
	
	\begin{remark}\label{remsmoothPN}
		Let \(P_n=\widehat P_n^{\mathrm{sym}}*K_{b_n}\), with \(K\in C_c^\infty(\mathbb R^d)\) a centered, orthogonally invariant kernel normalized so that \(\cov(K)=I_d\) (a radial kernel is orthogonally invariant, and its covariance is automatically a multiple of \(I_d\)). Then \(P_n\) is invariant under the affine reflection \(A_{\widehat\mu_n,\widehat\QS_S^3}\), has a \(C^\infty\) density, and \(\cov(P_n)=\cov(\widehat P_n^{\mathrm{sym}})+b_n^2I_d\). Under \(H_{0,S}\), whenever \(b_n\to0\) and \(n^{1/4}b_n\to\infty\): (a) with probability tending to one, \(\widehat\QS_S^3\) coincides with the reflection \(\QS_{S,n}(P_n)\) determined by the ordered eigenvectors of \(\cov(P_n)\), which is condition (ii) of the class \(\mathcal C_S(P)\), by the spectral-identification lemma of the Supplementary Material; (b) on the high-probability events supplied by \textup{(SB-a)} (stated below), the uniform density bound in \textup{(T2)} is $O(1)$, while the derivative, directional-Lipschitz and Hessian constants in \textup{(T2)}, \textup{(T2$'$)} and \textup{(T3)} are of order \(O_{\mathbb P}(b_n^{-2})=o_{\mathbb P}(\sqrt n)\); and (c) \textup{(T4)} follows from \(\cov(P_n)\xrightarrow{\mathbb P}\Sigma\) and the continuity of eigenvalues and eigenvectors under simple spectrum. The detailed verification is given in the Supplementary Material; the genuinely delicate conditions are the convergence parts of \textup{(T2)}--\textup{(T3)}, collected below as \textup{(SB)}.
	\end{remark}
	
	\begin{proposition}\label{prop:dF-smoothed-bootstrap}
		Assume \(H_{0,S}\), i.e.\ the centered law \(P^c\) is invariant under \(\QS_S\), and let $P_n=\widehat P_n^{\mathrm{sym}}*K_{b_n},$ where \(\widehat P_n^{\mathrm{sym}}\) is the symmetrized empirical measure constructed from the reflected sample, as in the bootstrap scheme described above. Assume moreover that \(\sup_{a\in\mathbb S^{d-1}}\|f_a\|_\infty<\infty\), where \(f_a\) denotes the density of \(a^\top(X-\mu)\); \(\mathbb E\|X\|^2<\infty\); \(\widehat u_{i,n}^3\to u_i\) in probability for every \(i\in S\); \(b_n\to0\); and the one-dimensional marginal \(k\) of \(K\) satisfies \(\int_{\mathbb R}|t| k(t) dt<\infty\). Then $d_{\mathcal F}(P_n,P)\to0$ in probability.
	\end{proposition}
	
	Combining Proposition~\ref{prop:dF-smoothed-bootstrap}, the spectral identification of the estimated mirror subspace, and the growth estimates of Remark~\ref{remsmoothPN}, we can state the validity of the smoothed bootstrap explicitly, as a self-contained consequence of Theorem~\ref{thm:bootstrap}. Some of the required conditions for \(P_n\) hold automatically for the kernel-smoothed symmetrized law. Two of them, uniform-in-direction consistency of the projected densities and weighted \(L_1\) consistency of the density gradients, are analytic statements belonging to kernel density-derivative estimation, which we collect as regularity conditions rather than re-derive here.
	
	\begin{named}{\textbf{(SB)}}\label{assum:SB}
		(Smoothed-bootstrap conditions.) The target law \(P\) has a \(C^1\) density \(f\) with \(\int_{\mathbb R^d}(1+\|x\|)\|\nabla f(x)\|\,dx<\infty\) and \(\sup_{a\in\mathbb S^{d-1}}\|f_a\|_\infty<\infty\); the kernel \(K\in C_c^\infty(\mathbb R^d)\) is a probability density (\(K\ge0\), \(\int K=1\)) that is centered and orthogonally invariant (\(\int xK(x)\,dx=0\) and \(K(Ox)=K(x)\) for all \(O\in\Od\)), normalized so that \(\int xx^\top K(x)\,dx=I_d\), with \(\int_{\mathbb R^d}(1+\|z\|^2)\|HK(z)\|_{\op}\,dz<\infty\); and the bandwidth satisfies \(b_n\to0\), \(n^{1/4}b_n\to\infty\), and \(nb_n^{d+2}\to\infty\). Moreover, for the kernel-smoothed symmetrized law \(P_n=\widehat P_n^{\mathrm{sym}}*K_{b_n}\),
		\begin{enumerate}[nosep]
			\item[(SB-a)] \(\displaystyle\sup_{a\in\mathbb S^{d-1}}\|f_{n,a}-f_a\|_\infty\to0\) in probability;
			\item[(SB-b)] \(\displaystyle\int_{\mathbb R^d}(1+\|x\|)\|\nabla f_n(x)-\nabla f(x)\|\,dx\to0\) in probability,
		\end{enumerate}
		where \(f_{n,a}\) and \(f_n\) are the projected and full densities of \(P_n\).
	\end{named}
	
	For a target law \(R\) with density \(f^R\) and projected densities \(f_a^R\), \textup{(SB)}\((R)\) denotes the same set of conditions with \(P\) replaced by \(R\) in (SB-a)--(SB-b); thus \textup{(SB)} abbreviates \textup{(SB)}\((P)\). Conditions (SB-a)--(SB-b) are the uniform-in-direction and weighted-\(L_1\) analogues of standard kernel density-derivative consistency, the non-standard feature being the uniformity over \(a\in\mathbb S^{d-1}\) combined with the data-dependent symmetrization through \(\widehat\QS_S^3\); sufficient conditions are expected to follow from uniform-in-direction kernel density-derivative results, but we take (SB-a)--(SB-b) as hypotheses on the pair \((P,b_n)\) (see the Supplementary Material for further discussion). The two families of assumptions in Proposition~\ref{prop:smoothed-bootstrap-valid} are not redundant: the fixed-law versions of \textup{(T1)}--\textup{(T4)},\,\textup{(T2$'$)} for \(P\) are used only to apply Theorem~\ref{teo:mainth1-TA} to the original statistic (in particular the fixed-law \textup{(T3)} requires \(C^2\) smoothness, stronger than the \(C^1\) smoothness in \textup{(SB)}), whereas \textup{(SB)} is used to verify that the bootstrap law \(P_n\) belongs to \(\mathcal C_S(P)\).
	
	\begin{proposition}\label{prop:smoothed-bootstrap-valid}
		Assume \(H_{0,S}\), Assumption~\ref{assum:H1}, the fixed-law versions of \textup{(T1)}--\textup{(T4)} and \textup{(T2$'$)} for \(P\) (needed to apply Theorem~\ref{teo:mainth1-TA} to the original statistic), and the smoothed-bootstrap conditions \textup{(SB)}. Then there exist events \(\Omega_n\in\sigma(X_1,\dots,X_n)\) with \(\mathbb P(\Omega_n)\to1\) and random probability measures \(\widetilde P_n\), measurable with respect to \(X_1,\dots,X_n\), such that \(\widetilde P_n=P_n\) on \(\Omega_n\) and the realized sequence \((\widetilde P_n)_{n\ge1}\) belongs to \(\mathcal C_S(P)\) almost surely, where \(P_n=\widehat P_n^{\mathrm{sym}}*K_{b_n}\). Consequently, by Theorem~\ref{thm:bootstrap}, the bootstrap \(p\)-value \(p_{n,S}\) is asymptotically valid: \(\limsup_{n\to\infty}\mathbb P(p_{n,S}\le\alpha)\le\alpha\).
	\end{proposition}
	
	The next result shows that, under fixed alternatives, the same bootstrap construction yields a consistent test, provided the data-dependent resampling law admits a modification belonging to the appropriate bootstrap class and converges to the corresponding population symmetrization.
	
	For \(S\in\mathcal S\), let \(A_S:=A_{\mu,\QS_S}\) denote the affine reflection with mirror \(V_S\) through \(\mu\), and define the symmetrized population law $P^{\mathrm{sym}}_S:=\frac12 P+\frac12\, P\circ A_S^{-1}.$ Let \(\mathcal C_S(P^{\mathrm{sym}}_S)\) denote the class obtained from \(\mathcal C_S(P)\) by replacing \(P\) with \(P^{\mathrm{sym}}_S\). (We write \(P^{\mathrm{sym}}_S\) for this probability measure to distinguish it from the reflection matrix \(\QS_S\).)
	
	\begin{theorem}\label{thm:power}
		Assume \ref{assum:H1}, Carleman's condition \eqref{carleman}, and the fixed-law analogues of \textup{(T1)}--\textup{(T2)}. Suppose that the global null hypothesis in \eqref{hs} is false. Assume moreover that, for each \(S\in\mathcal S\), the bootstrap law \(P_{n,S}\) satisfies the approximation condition stated above with limiting law \(P^{\mathrm{sym}}_S\), that is, there exist \(\Omega_{n,S}\in\sigma(X_1,\dots,X_n)\) and random probability measures \(\widetilde P_{n,S}\), measurable with respect to \(X_1,\dots,X_n\), such that \(\mathbb P(\Omega_{n,S})\to1\), \(\widetilde P_{n,S}=P_{n,S}\) on \(\Omega_{n,S}\), and \((\widetilde P_{n,S})_{n\ge1}\in\mathcal C_S(P^{\mathrm{sym}}_S)\) almost surely.
		
		Then, for surface-a.e.\ \(h\in\mathbb S^{d-1}\), for every \(S\in\mathcal S\), \(g_h(\QS_S)>0\), \(\widehat g_{n,h}(\widehat\QS_S^3)\xrightarrow{\mathbb P} g_h(\QS_S)\), \(T_{n,S}(h)\xrightarrow{\mathbb P}\infty\), the bootstrap critical value is bounded in probability, \(d_{n,S}^\dagger(1-\alpha)=O_{\mathbb P}(1)\), and \(p_{n,S}\xrightarrow{\mathbb P}0\). Consequently, \(\mathbb P(p_n^{\mathrm{glob}}\le\alpha)\to1\). Thus the global bootstrap test is consistent against fixed alternatives.
	\end{theorem}
	
	\begin{remark}\label{rem:power-hypotheses}
		For the consistency of the statistic \(T_{n,S}(h)\) itself, only \(\mathbb E\|X\|^2<\infty\), \(\sup_{a}\|f_a\|_\infty<\infty\), Carleman's condition and the simple spectrum are used. The stronger fixed-law \textup{(T1)}--\textup{(T2)} in the statement are retained only for compatibility with the regularity framework of the bootstrap (Theorem~\ref{thm:bootstrap} and Proposition~\ref{prop:smoothed-bootstrap-alt}).
	\end{remark}
	
	The approximation hypothesis of Theorem~\ref{thm:power} is discharged, for the concrete smoothed-bootstrap construction, by the following analogue of Proposition~\ref{prop:smoothed-bootstrap-valid} with target law \(P^{\mathrm{sym}}_S\). We state it for a general \(P\), since the construction and its limit \(P^{\mathrm{sym}}_S\) make sense whether or not \(H_{0,S}\) holds (if \(H_{0,S}\) holds then \(P^{\mathrm{sym}}_S=P\) and it reduces to Proposition~\ref{prop:smoothed-bootstrap-valid}).
	
	\begin{proposition}\label{prop:smoothed-bootstrap-alt}
		Fix \(S\in\mathcal S\). Assume Assumption~\ref{assum:H1}, the fixed-law versions of \textup{(T1)}--\textup{(T4)} and \textup{(T2$'$)} for the symmetrized law \(P^{\mathrm{sym}}_S\), the smoothed-bootstrap conditions \textup{(SB)}\((P^{\mathrm{sym}}_S)\), and moreover \(\sup_{a\in\mathbb S^{d-1}}\|f_a^{P}\|_\infty<\infty\) and \(\mathbb E_P\|X\|^2<\infty\) for the original law \(P\) (used to control the strip term below under \(P\); these are consequences of, but strictly weaker than, the fixed-law \textup{(T1)}--\textup{(T2)} for \(P\)). Let \(P_{n,S}=\widehat P_{n}^{\mathrm{sym}}*K_{b_n}\) be the symmetrized and smoothed bootstrap law built from the original sample with \(\widehat\QS_S^3\) and \(\widehat\mu_n\). Then there exist events \(\Omega_{n,S}\in\sigma(X_1,\dots,X_n)\) with \(\mathbb P(\Omega_{n,S})\to1\) and random probability measures \(\widetilde P_{n,S}\) with \(\widetilde P_{n,S}=P_{n,S}\) on \(\Omega_{n,S}\) and \((\widetilde P_{n,S})_{n\ge1}\in\mathcal C_S(P^{\mathrm{sym}}_S)\) almost surely.
	\end{proposition}
	
	Proposition~\ref{prop:smoothed-bootstrap-alt} concerns one fixed candidate \(S\); combining it over the finite family \(\mathcal S\) yields the global consistency statement.
	
	\begin{corollary}\label{cor:global-consistency}
		Assume that the hypotheses of Proposition~\ref{prop:smoothed-bootstrap-alt} hold for every \(S\in\mathcal S\) and, in addition, that Carleman's condition \eqref{carleman} and the fixed-law \textup{(T1)}--\textup{(T2)} for \(P\) hold. Then the approximation hypothesis of Theorem~\ref{thm:power} is satisfied for every \(S\in\mathcal S\); consequently, if the global null hypothesis \eqref{hs} is false, the global smoothed-bootstrap test is consistent against fixed alternatives.
	\end{corollary}
	
	The assumptions on \(P^{\mathrm{sym}}_S\) hold whenever \(P\) satisfies the corresponding fixed-law assumptions. Since \(A_S\) is affine with orthogonal linear part, symmetrization preserves the weighted \(C^2\) regularity of \textup{(T3)} and the projected-density bounds of \textup{(T2)}--\textup{(T2$'$)}. Because \(\QS_S\) commutes with \(\Sigma\), it also preserves the covariance eigenstructure, with \(\cov(P^{\mathrm{sym}}_S)=\Sigma\); the explicit verification is given in the Supplementary Material. Moreover, \(\widehat\QS_S^3\xrightarrow{\mathbb P}\QS_S\) and \(\widehat\mu_n\xrightarrow{\mathbb P}\mu\) hold whether or not \(H_{0,S}\) holds, using only \(\mathbb E\|X\|^2<\infty\) and the simple spectrum, because no symmetry of \(P\) is used to estimate the principal components.
	
	\begin{remark}\label{rem:random-h}
		Theorem~\ref{thm:bootstrap} and Corollary~\ref{cor:global-pvalue} hold for every fixed \(h\in\mathbb S^{d-1}\); hence the level guarantee is valid conditionally on the realized \(h\), and therefore also unconditionally when \(h\) is drawn independently of the data (integrate over the law of \(h\)). Theorem~\ref{thm:power}, by contrast, holds for surface-a.e.\ \(h\), and therefore with probability one over the independent randomization of \(h\). Thus randomizing \(h\) preserves the asymptotic level, and consistency holds almost surely over the choice of \(h\).
	\end{remark}
	
	\begin{remark}\label{rem:projections}
		The testing procedure developed in this paper is based on one-dimensional
		projections. This choice is supported by the random-projection principle:
		whenever $H_{0,S}$ holds, the projected distributions of $P$ and of its
		symmetrization about $Q_S$ coincide for every direction $h\in\mathbb{S}^{d-1}$,
		whereas under the alternative a randomly chosen direction detects the
		discrepancy with probability one under suitable identifiability conditions.
		
		Theorem~\ref{teo:mainth1-TA} is stated for a single direction, but the theory is not
		restricted to that case. For a fixed number $k$ of directions
		$h_1,\dots,h_k$, drawn i.i.d.\ uniformly on $\mathbb{S}^{d-1}$ independently of
		the sample and kept fixed for the observed statistic and for all bootstrap
		replicates, Corollary~\ref{cor:finite-projections} shows that the aggregated statistic
		\[
		T^{(k)}{n,S}=\frac{1}{k}\sum{\ell=1}^{k}\sqrt{n}\,
		\widehat{g}_{n,h_\ell}\bigl(\widehat{Q}^{3}_{S}\bigr)
		\]
		is covered by exactly the same asymptotic and bootstrap theory as the
		single-projection case, with all the cross-direction dependence---in
		particular the common eigenvector shift---already encoded in the single limit
		$\mathcal{L}^{S}$. This is the version actually used in the numerical work: the
		simulation study of Section~\ref{sec:simus} and the application of
		Section~\ref{sec:forex-application} report results for $k\in\{1,10,100\}$, and the
		power gain from $k=1$ to $k=10$ is substantial, while $k=10$ and $k=100$ are
		nearly indistinguishable.
		
	\end{remark}	
	
	\begin{remark} \label{rem:sample_splitting}
		The sample splitting procedure involves three balanced subsamples: $\aleph_1, \aleph_2$, and $\aleph_3$. Alternative schemes could be considered, such as allocating a larger sample size to the estimation of the principal directions or to the two-sample test. Optimizing this splitting criterion is beyond the scope of this paper, as it is not clear how an optimal balance should be determined. For instance, under the alternative hypothesis, since there is no true symmetry, reducing the variance of the principal components is less critical than increasing the sample size for the Kolmogorov--Smirnov test.
	\end{remark}

\section{Some particular cases 
}\label{sec:axial}

Three instances of the general theory deserve a separate discussion: reflection across a line, reflection across a hyperplane, and central symmetry. The first one, the one-dimensional case, is arguably the most important in applications and is the case that motivated this work. The second one is the natural formulation when a single direction carries the meaning of the problem, and it is the one used in the real-data application of Section~\ref{sec:forex-application}. We show here how the general results of Sections~\ref{sec:princomp}--\ref{sec:mainth} specialize to each of them.

Throughout, it is useful to keep in mind that the mirror subspace \(V\) is the set of points that the reflection leaves fixed, while the orthogonal complement \(V^\perp\) is the part whose sign is reversed. The axial and the hyperplane cases are complementary in this respect: the former fixes a single direction and flips the remaining \(d-1\), while the latter flips a single direction and fixes the remaining \(d-1\).

\subsection{The axial (line-reflection) case}

Take \(\mathcal S=\{\{i\}:i=1,\dots,d\}\), so that each candidate is a singleton \(S=\{i\}\). Then \(\PS_{\{i\}}=u_iu_i^\top\) and the reflection is the axial reflection \(\QS_{\{i\}}=2u_iu_i^\top-I_d=:R_{u_i}\), which fixes the line \(\operatorname{span}(u_i)\) (eigenvalue \(+1\)) and reverses \(u_i^\perp\) (eigenvalue \(-1\)). The random vector \(X\) is axially symmetric about \(u_0\in\mathbb S^{d-1}\) if \(X-\mu\overset{\mathcal L}{=}R_{u_0}(X-\mu)\); note that \(R_{u_0}=R_{-u_0}\), so the axis \(\operatorname{span}(u_0)\) does not depend on the sign of \(u_0\). Proposition~\ref{prop:commute} then reduces to the statement that any symmetry axis must be an eigenvector of \(\Sigma\), and the null hypothesis \eqref{hs} becomes
\[
\begin{cases}
	H_0: \text{there exists } u_0 \in \mathbb{S}^{d-1}: X - \mathbb{E}[X] \overset{\mathcal{L}}{=} R_{u_0}(X - \mathbb{E}[X]),  \\
	H_1:   H_0 \text{ is false},
\end{cases}
\]
equivalently \(H_0=\cup_{i=1}^d H_{0,i}\) with \(H_{0,i}:X-\mathbb E[X]\overset{\mathcal L}{=}R_{u_i}(X-\mathbb E[X])\).

\begin{example}
	For \(u=e_1=(1,0,\dots,0)^\top\), with \(\mathbf 0\) the \((d-1)\)-dimensional zero vector,
	\[
	R_{e_1}=\begin{pmatrix}
		1 & \mathbf{0}^\top\\
		\mathbf{0} & -I_{d-1}
	\end{pmatrix},
	\]
	i.e., the first coordinate is preserved and the remaining \(d-1\) coordinates are sign-flipped.
\end{example}

In this case the candidate-specific statistic is \(T_{n,i}(h)=\sqrt n\,\widehat g_{n,h}(R_{\widehat u_{i,n}^3})\). Writing \(G_{u_i}\) for \eqref{eq:G-general} with \(\QS_S=R_{u_i}\), \(G_{u_i}(\ga)=2\int_{\mathbb R^d}\ga(x)((u_i^\top x)I_d+xu_i^\top)\nabla f(R_{u_i}x+\mu)\,dx\). Theorem~\ref{teo:mainth1-TA} yields \(\sqrt n\,\widehat g_{n,h}(R_{\widehat u_{i,n}^3})\xrightarrow{\mathcal L}\|\mathbb L_h\|_{\mathcal F_h}\), where
\[
\mathbb L(\ga)=\mathbb L^{(1)}(\ga)-\mathbb L^{(2)}(\ga)-G_{u_i}(\ga)^\top Z_i^3,
\qquad Z_i^3\sim\mathcal N\bigl(0,3P(\psi_i\psi_i^\top)\bigr),
\]
with \(\psi_i(x)=\sum_{j\neq i}\frac{(u_j^\top(x-\mu))(u_i^\top(x-\mu))}{\lambda_i-\lambda_j}u_j.\) Correspondingly, the covariance kernel \eqref{eq:cov-kernel} reduces to
\begin{multline*}
	\mathrm{Cov}\big(\mathbb L_h(\ga_1),\mathbb L_h(\ga_2)\big)
	=
	6\Big(P^c(\ga_1\ga_2)-P^c\ga_1 P^c\ga_2
	+s(\ga_1)^\top \Sigma s(\ga_2)\\
	+s(\ga_1)^\top m(\ga_2)
	+s(\ga_2)^\top m(\ga_1)\Big)\\
	+3\,G_{u_i}(\ga_1)^\top P(\psi_i\psi_i^\top) G_{u_i}(\ga_2),
\end{multline*}
recovering the single-direction formulas. The bootstrap of Section~\ref{subsec:bootstrap-validity-romano} symmetrizes the sample using the affine axial reflection \(A_{\widehat\mu_n,R_{\widehat u_{i,n}^3}}\), and Theorems~\ref{thm:bootstrap}--\ref{thm:power}, together with Propositions~\ref{prop:smoothed-bootstrap-valid} and~\ref{prop:smoothed-bootstrap-alt} for the concrete smoothed-bootstrap implementation, give the validity and consistency of the resulting global axial-symmetry test under the stated regularity and \textup{(SB)} assumptions. This is the procedure implemented in the simulation study of Section~\ref{sec:simus}.

\subsection{Hyperplane (Householder) symmetry}\label{subsec:hyperplane}

The complementary instance is \(S=\{1,\dots,d\}\setminus\{i\}\), of cardinality \(d-1\). The mirror subspace is then the hyperplane \(V_S=\operatorname{span}(u_j:j\neq i)=u_i^\perp\) and the reflection is the Householder matrix \(\QS_S=I_d-2u_iu_i^\top=:H_{u_i}\). It fixes the hyperplane \(u_i^\perp\) and reverses the normal direction \(u_i\). In words, it changes the sign of the \(i\)-th principal component and leaves the other \(d-1\) components untouched.

\begin{example}
	For \(u=e_1\),
	\[
	H_{e_1}=\begin{pmatrix}
		-1 & \mathbf{0}^\top\\
		\mathbf{0} & I_{d-1}
	\end{pmatrix},
	\]
	i.e., the first coordinate is sign-flipped and the remaining \(d-1\) coordinates are preserved. Compare with \(R_{e_1}\) above: the roles of the two blocks are exchanged.
\end{example}

The hypothesis has a simple reading. Write \(S_i=u_i^\top(X-\mu)\) for the \(i\)-th principal component and \(W_i=(I_d-u_iu_i^\top)(X-\mu)\) for the remaining part. Then invariance under \(H_{u_i}\) is
\[
(S_i,W_i)\overset{\mathcal L}{=}(-S_i,W_i).
\]
A test of hyperplane symmetry therefore asks two questions at the same time: whether the single component \(S_i\) is symmetric around zero, and whether the way \(S_i\) interacts with the remaining components is also symmetric. This is the natural formulation when one specific direction carries the meaning of the problem and the interest lies in the sign of that direction, as in the application of Section~\ref{sec:forex-application}.

The two reflections are related by \(H_{u_i}=-R_{u_i}\), as already noted in Section~\ref{sec:setup}. Consequently, axial symmetry along \(u_i\) and hyperplane symmetry across \(u_i^\perp\) differ by a central symmetry. If a distribution has two of these three symmetries, then it also has the third one. In particular, the two invariances coincide when \(X-\mu\) is centrally symmetric, and otherwise neither of them implies the other. They are also different questions in practice: the axial hypothesis fixes \(S_i\) and flips \(W_i\), while the hyperplane hypothesis does the opposite. In either case, by Proposition~\ref{prop:commute}, the normal \(u_i\) to a symmetry hyperplane must be an eigenvector of \(\Sigma\).

Taking \(\mathcal S=\{\{1,\dots,d\}\setminus\{i\}:i=1,\dots,d\}\) turns the general procedure into a test for hyperplane symmetry about an unspecified normal direction. When the direction of interest is determined by the problem, one may instead prespecify a single candidate. The global test then reduces to that candidate and no correction for multiplicity is needed. This is what we do in Section~\ref{sec:forex-application}.

The implementation requires one remark. The mirror subspace has dimension \(d-1\), so the general linearization \eqref{eq:limit-representation} involves \(d-1\) eigenvector terms. Since \(\QS_S=I_d-2u_iu_i^\top\) depends only on the single normal \(u_i\), one may equivalently estimate that one eigenvector and use \(\widehat\QS_S=I_d-2\widehat u_{i,n}\widehat u_{i,n}^\top\). Indeed, by the orthonormality constraint \(\sum_j u_ju_j^\top=I_d\), the perturbation \(d\QS_S=-2(du_i\,u_i^\top+u_i\,du_i^\top)\) obtained from \(u_i\) coincides with the one obtained by summing over the other \(d-1\) eigenvectors. The two linearizations agree, so only \(u_i\) need be estimated in practice. All other steps are identical.

\subsection{Central symmetry}

The empty index set \(S=\varnothing\) gives \(\QS_\varnothing=-I_d\), i.e.\ central symmetry \(X-\mu\overset{\mathcal L}{=}-(X-\mu)\). This reflection does not depend on the eigenvectors at all (the mirror subspace is \(\{0\}\)), so no principal-component estimation is required for this candidate and the linearization term in \eqref{eq:limit-representation} is absent: the limit reduces to \(\mathbb L=\mathbb L^{(1)}-\mathbb L^{(2)}\). Central symmetry can thus be tested with the same split-sample statistic, using \(Q=-I_d\) in \eqref{empirical}, and is the simplest member of the family \eqref{eq:candidate-family}. As a convention, when \(S=\varnothing\) all sums over \(i\in S\) are understood as zero and \(\QS_{S,n}=-I_d\) deterministically; no principal component is estimated, so the sample is split into just two blocks of sizes \(n/2\). Additionally, in the limit expression the factors \(\sqrt3\) and \(6\) must then be replaced by \(\sqrt2\) and \(4\).

	\section{Simulation study} \label{sec:simus}
	We evaluate the finite-sample performance of the proposed test in the axial (line-reflection) case of Section~\ref{sec:axial}, i.e.\ with \(\mathcal S=\{\{i\}:i=1,\dots,d\}\). We sample from the mixture \(\frac12 P_1+\frac12 P_2O_\theta^{-1}\), where \(P_1\) is uniform on the rectangle with vertices \((\pm1,\pm3)\), \(P_2\) is uniform on the square with vertices \((\pm1/2,2)\) and \((\pm1/2,3)\), and \(O_\theta\) denotes rotation by angle \(-\theta\). We consider \(n\in\{300,500,1000,5000\}\) and \(\theta=k/3\), \(k=0,1,2,3\). \(\theta=0\) yields an axially symmetric mixture, so it is a null configurations, whereas \(\theta=1/3,2/3, 1\) are alternatives. Figure~\ref{fig:conjuntos} shows realizations of this mixture, together with the empirical principal components.
	
	Each configuration was replicated 1000 times, using \(B=500\) bootstrap resamples per replication. Within each replication, the random projection directions were drawn once and kept fixed for the observed statistic and all bootstrap replicates, and the candidate reflection was re-estimated from the third block of every bootstrap sample, as in Algorithm~1. For each candidate \(S\), the bootstrap \(p\)-value \(p_{n,S}\) was computed from the \(B\) bootstrap replicates, and the global \(p\)-value \(p_n^{\mathrm{glob}}=\max_{S\in\mathcal S}p_{n,S}\) was used for the rejection decision.
	The bootstrap draws in Step~3 of Algorithm~1 use an isotropic Gaussian kernel, \(K=\mathcal N(0,I_d)\), with bandwidth \(b_n=n^{-1/5}\). For this mixture model (\(d=2\)), this choice satisfies \(nb_n^{d+2}\to\infty\) with room to spare, since \(1/5<1/(d+2)=1/4\) here; this is unlike the boundary case \(1/(d+2)=1/5\) encountered at \(d=3\) in the real-data application of Section~\ref{sec:forex-application}, where the bandwidth rule and kernel choice are adjusted accordingly (see the discussion there). The bootstrap \(p\)-value is computed in its finite-\(B\) form \(\widehat p_{n,S}^{\,B}=\bigl(1+\sum_{b=1}^{B}\mathbf 1\{T_{n,S}^{\dagger,b}(h)\ge T_{n,S}(h)\}\bigr)/(B+1)\), where \(T_{n,S}^{\dagger,1},\dots,T_{n,S}^{\dagger,B}\) are the bootstrap replicates, and the global \(p\)-value is \(\widehat p_n^{\mathrm{glob},B}=\max_{S}\widehat p_{n,S}^{\,B}\); the \(+1\) in numerator and denominator is the standard, slightly conservative finite-\(B\) approximation to the conditional bootstrap tail probability. Figure~\ref{fig:potencias} reports the estimated power, computed as the Monte Carlo rejection proportion, together with boxplots of the bootstrap \(p\)-values. The rejection probability at the null configuration \(\theta=0\) decreases toward the nominal level as \(n\) grows and increases sharply at the intermediate and largest angles, where the \(p\)-values concentrate near zero.



	Motivated by Remark~\ref{rem:projections}, we also considered aggregated statistics based on \(k=10\) and \(k=100\) random directions. Power increases substantially from \(k=1\) to \(k=10\), while the curves for \(k=10\) and \(k=100\) are nearly indistinguishable, suggesting that most of the gain is already achieved with a modest number of projections. By Corollary~\ref{cor:finite-projections}, for any fixed collection of sampled directions the aggregated statistics with \(k=10\) and \(k=100\) are covered by the same asymptotic and bootstrap theory as the single-projection case.
	
	
	\begin{figure}[H]
		\centering
		\includegraphics[width = .7\textwidth]{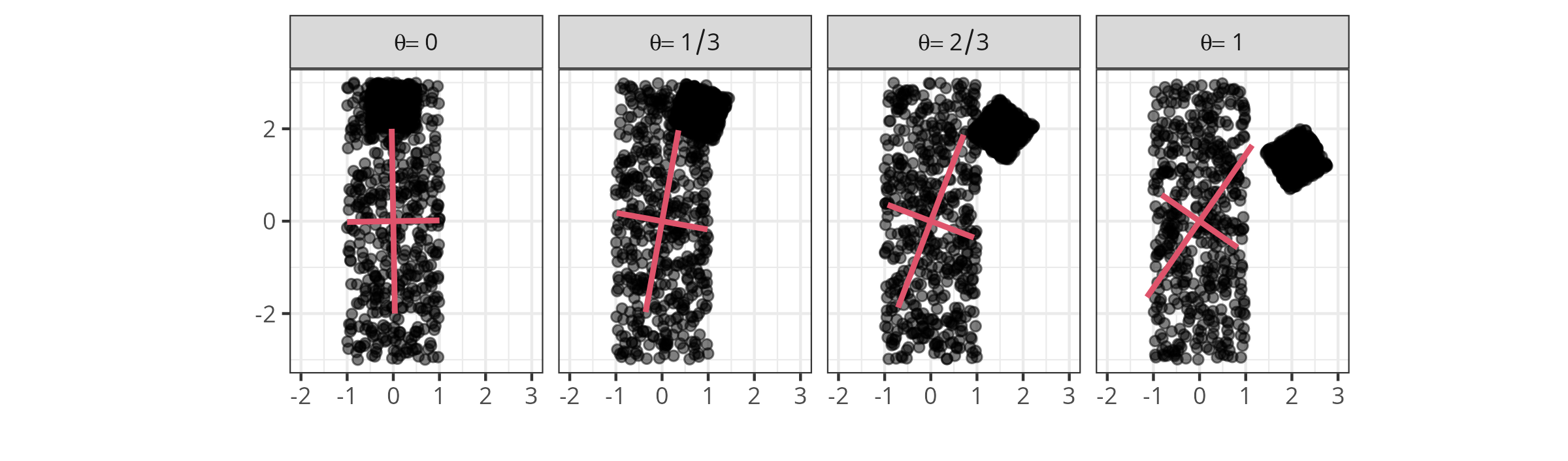}
		\caption{Realizations of $\frac12 P_1 + \frac12 P_2 O^{-1}_\theta$ with $n=1000$ and $\theta = k\pi/3$ for $k \in \{0,1,2,3\}$. The empirical principal components are also shown as solid lines.}
		\label{fig:conjuntos}
	\end{figure}
	
	\begin{figure}[H]
		\centering
		\includegraphics[width = .7\textwidth]{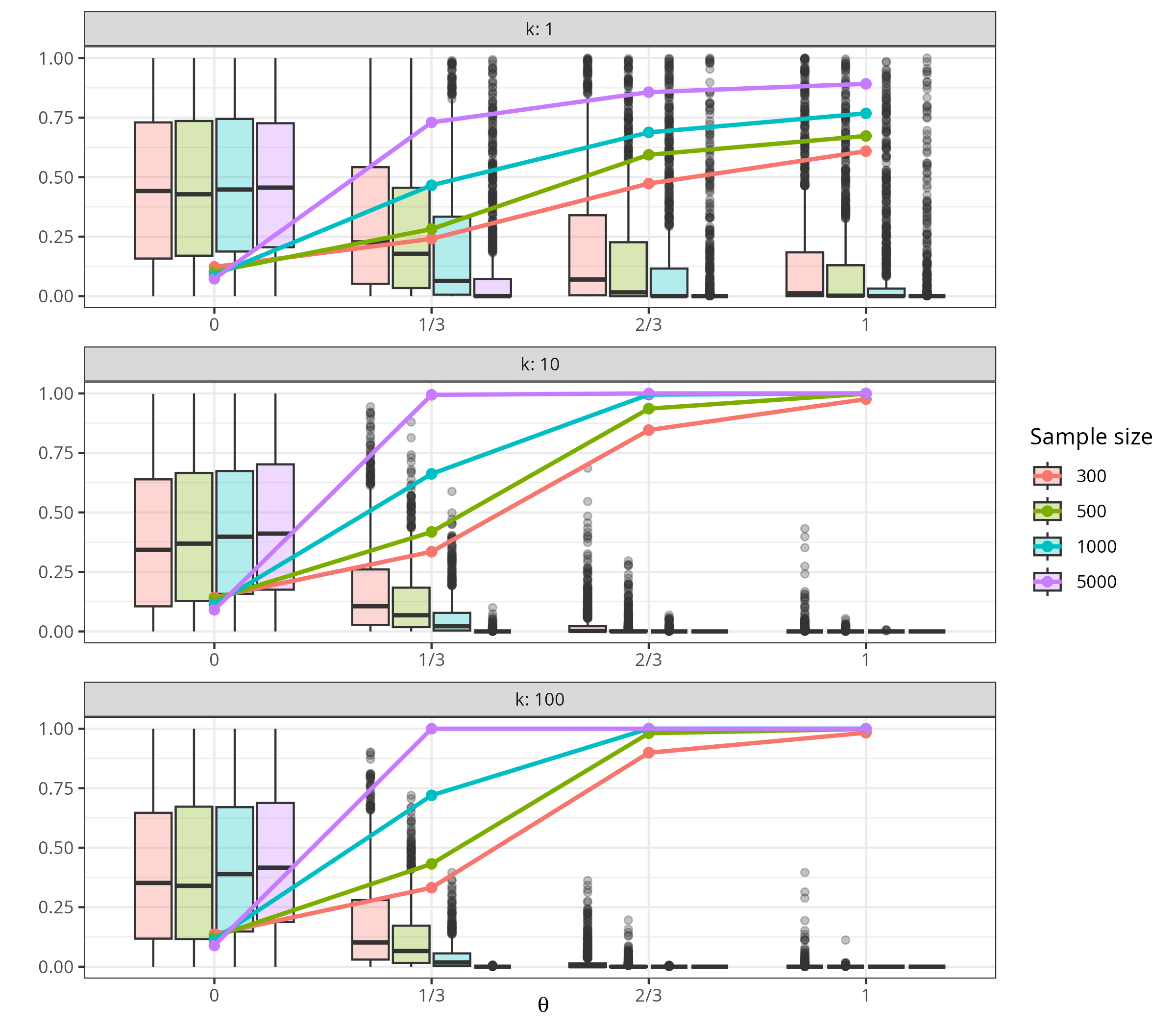}
		\caption{Estimated power curves and empirical p-value distributions under the null and alternative hypotheses for different sample sizes and numbers of random directions. Solid lines and points represent the estimated power based on the Monte Carlo rejection proportions. Boxplots display the corresponding distribution of the p-values across 1000 replications.}
		\label{fig:potencias}
	\end{figure}


\section{Real-data application: crash risk in carry trade exchange rates}
\label{sec:forex-application}

\paragraph{Data and motivation}

In this section we apply our methodology to a dataset of foreign exchange rates. We use daily closing prices for $\mathrm{AUD/JPY}$ (Australian dollar vs.\ Japanese yen), $\mathrm{MXN/JPY}$ (Mexican peso vs.\ Japanese yen), and $\mathrm{NZD/JPY}$ (New Zealand dollar vs.\ Japanese yen), downloaded from Yahoo Finance (tickers \texttt{AUDJPY=X}, \texttt{MXNJPY=X}, \texttt{NZDJPY=X}) for the period January 1, 2015 to January 1, 2026. After removing days with a missing quote for any of the three pairs and taking log-differences, this yields $n=2862$ daily log-returns. These currency pairs are typically used in a well-known financial strategy called the carry trade. In this strategy, investors borrow money in a currency with very low interest rates (the Japanese yen) and invest it in currencies that pay higher interest rates. Such positions usually accumulate small and steady gains during quiet periods. They can also suffer large and sudden losses during episodes of market panic, when many investors try to sell at the same time (a process known as deleveraging). We expect this structural behaviour to produce a strong asymmetry in the underlying distribution.

Our data capture only the daily price changes of the currencies, not the interest earned by the strategy. This application therefore focuses on the exchange-rate risk shared by typical yen-funded investments.

We order the variables as $(\mathrm{AUD/JPY}, \mathrm{MXN/JPY}, \mathrm{NZD/JPY})$. The sample correlations are $0.65, 0.86$ and $0.65$ for the first-second, first-third and second-third pairs of variables respectively, so the three series move closely together. The three marginal distributions show negative sample skewness ($-0.49$, $-0.84$ and $-0.37$) and substantial excess kurtosis. Both features point to the heavy left tails caused by sudden market drops.

Previous symmetry analyses, such as \citet{hudecova2021testing}, have also looked at exchange rates like $\mathrm{AUD/JPY}$, but using directions fixed in advance. Our approach is different: we do not fix the reflection direction beforehand, we estimate it from the covariance structure of this currency basket.

The leading eigenvector of the sample covariance matrix is $\widehat{u}_1 \approx (0.53, 0.71, 0.46)$. All three entries share the same sign and have similar size. Hence $\widehat{u}_1$ describes a market-wide movement in which the three currencies move together against the yen, and we read it as the common risk factor of the carry trade. The three sample eigenvalues are $1.52\times 10^{-4}$, $3.40\times 10^{-5}$ and $6.99\times 10^{-6}$, so this single direction accounts for $78.8\%$ of the total variance, against $17.6\%$ and $3.6\%$ for the other two. The remaining two directions are contrasts between the currencies and have no prespecified financial meaning.

The eigenvalues are also clearly distinct: consecutive ones differ by factors of $4.5$ and $4.9$. The simple-spectrum condition required by our theory is therefore comfortably satisfied in this dataset, and the leading direction is well identified. This is confirmed empirically: $\widehat{u}_1$ barely moves across random splits of the sample, as reported in Section~\ref*{sec:supp-sensitivity} of the supplementary material.

When a daily observation is nearly collinear with this direction, the market is reacting as a single block to global factors. This gives a simple geometric picture of the asymmetry of the strategy. Observations on the positive side of the axis correspond to periods of risk appetite: frequent but small joint appreciations. Movements on the negative side correspond to episodes of panic and simultaneous deleveraging. Since the daily mean log-return is close to zero, falling below the mean is essentially a negative return (for instance $\mathrm{AUD/JPY}_t < \mathrm{AUD/JPY}_{t-1}$, and likewise for the other pairs). Visually, these crashes appear as long vectors stretching into the left tail of the distribution. This is exactly the kind of asymmetry that our test is designed to detect.

\paragraph{The hypothesis under test}

The relevant question here is whether the sign of the common factor matters, so the natural null hypothesis is symmetry with respect to the hyperplane orthogonal to $u_1$. This is the particular case discussed in Section~\ref{subsec:hyperplane}, applied with $d=3$ and $i=1$. We prespecify the candidate family as $\mathcal{S}_{\mathrm{app}}=\bigl\{\{2,3\}\bigr\}$. The mirror subspace is $V_{\{2,3\}}=\operatorname{span}(u_2,u_3)=u_1^{\perp}$ and the corresponding reflection is $Q_{\{2,3\}}=I_3-2u_1u_1^{\top}=H_{u_1}$. The null hypothesis is therefore
$$H_{0,\{2,3\}}:\qquad X-\mu \overset{d}{=} H_{u_1}(X-\mu).$$
Since the candidate family has a single element, the global test reduces to that single candidate: $p_n^{\mathrm{glob}}=p_{n,\{2,3\}}$. In the notation of Section~\ref{subsec:hyperplane}, with $S=u_1^{\top}(X-\mu)$ the common market factor and $W=(I_3-u_1u_1^{\top})(X-\mu)$ the country-specific variation, the hypothesis reads $(S,W)\overset{d}{=}(-S,W)$. The test thus checks whether large joint drops are as likely as large joint rises, and also whether the link between the common factor and the country-specific part is the same on both sides.

The general procedure allows the candidate family to contain all $2^d$ subsets of principal directions. It is natural to ask why we do not scan all of them here. There are three reasons.

First, the question of interest is specific. We do not ask whether the data admit some symmetry or other. We ask whether the common factor of the basket is symmetric, since that is the direction with a financial meaning.

Second, scanning has a statistical cost. As noted in Remark~\ref{rem:global-test}, the intersection structure of the global test ($p_n^{\mathrm{glob}}=\max_S p_{n,S}$) requires no multiplicity correction for validity, but becomes more conservative as the candidate family grows, which reduces power. With a single candidate fixed in advance, our main focus here, this loss of power does not arise.

Third, only one of the three directions has a financial interpretation. As reported above, $\widehat{u}_1$ carries $78.8\%$ of the variance and admits a direct reading as the common factor of the basket. The other two directions are contrasts between the currencies, with $17.6\%$ and $3.6\%$ of the variance. A rejection along one of them would be much harder to translate into a statement about the carry trade.

We stress that this last point is about interpretation and not about estimation error. The precision of an estimated eigenvector is driven by the gaps between eigenvalues relative to their own size, and here all three eigenvalues are well separated: consecutive ones differ by factors of $4.5$ and $4.9$. Accordingly, the coefficients $\sqrt{\lambda_i\lambda_j}/(\lambda_i-\lambda_j)$ that govern the linearization term are $0.61$, $0.23$ and $0.57$ for the pairs $(1,2)$, $(1,3)$ and $(2,3)$. All three directions are thus estimated with comparable accuracy, and the full family could be used here at the cost of the power loss from the intersection structure discussed above.

A related question is why we test hyperplane symmetry and not axial symmetry. The mirror subspace is the set of points that the reflection leaves fixed. Axial symmetry along $u_1$ would keep the common factor untouched and flip the two remaining directions, so it would be a statement about the country-specific components. That is not the question raised by the carry trade. Hyperplane symmetry across $u_1^{\perp}$ flips the sign of the common factor and leaves the rest fixed, which is precisely the crash-risk question. The two hypotheses are also formally different: as noted in Section~\ref{subsec:hyperplane}, they differ by a central symmetry, so neither implies the other.

\paragraph{Implementation and results}

We apply Algorithm~1 with a single random sample split, generated from a fixed seed chosen before inspecting any result.

We use aggregated statistics based on $k\in\{1,10,100\}$ random projection directions. These directions are drawn independently of the data and are kept fixed for the observed statistic and for every bootstrap replicate. Because $k$ is fixed, these statistics are covered by Corollary~\ref{cor:finite-projections}.

For the primary analysis we run a bootstrap with $B=10\,000$ replicates. The draws from the symmetrized pool are smoothed with a $C^\infty$, compactly-supported, orthogonally invariant kernel $K$: $X_b=\text{draw}+b_nZ$, where $Z\sim K$ is drawn on the ball $\{\|z\|<1\}$ (a smooth bump kernel, $K(z)\propto\exp(-1/(1-\|z\|^2))$ for $\|z\|<1$) and rescaled to satisfy $\int zz^\top K(z)\,dz=I_3$, as required by \textup{(SB)}. We use this compactly-supported kernel, rather than a Gaussian, because \textup{(SB)} requires $K\in C_c^\infty(\mathbb R^3)$. The bandwidth is $b_n=0.35\cdot s_n\cdot n^{-1/6}$, where $s_n$ is the average standard deviation of the daily log-returns. This exponent, rather than the more familiar $n^{-1/5}$, is chosen because \textup{(SB)} requires $nb_n^{d+2}\to\infty$ with $d=3$, i.e.\ $nb_n^5\to\infty$: at $n^{-1/5}$ this product is asymptotically constant, exactly on the boundary of validity, whereas $n^{-1/6}$ satisfies the condition with room to spare ($nb_n^5=n^{1/6}\to\infty$). The bootstrap $p$-value is $\widehat{p}_n^{\,B}=\bigl(1+\#\{b: T_b\ge T_{\mathrm{obs}}\}\bigr)/(B+1)$.

Table~\ref{tab:final-run} reports the results. The null hypothesis $H_{0,\{2,3\}}$ is rejected at the $5\%$ level for the three values of $k$. With $B=10\,000$ the smallest attainable $p$-value is $1/10\,001\approx 10^{-4}$. The reported $p$-values ($0.0141$, $0.0017$ and $0.0029$) are well above that floor, so they can be compared across values of $k$.

\begin{table}[ht]
\centering
\begin{tabular}{cccc}
\toprule
$k$ & $T_{\mathrm{obs}}$ & boot.\ $q_{0.95}$ & $p$-value \\
\midrule
1   & 3.0842 & 2.6917 & \textbf{0.0141} \\
10  & 2.9328 & 2.3833 & \textbf{0.0017} \\
100 & 2.7024 & 2.2801 & \textbf{0.0029} \\
\bottomrule
\end{tabular}
\caption{Final run ($B=10\,000$ bootstrap replicates): observed statistic, bootstrap $95\%$ quantile, and bootstrap $p$-value, for each aggregation level $k$.}
\label{tab:final-run}
\end{table}

The procedure involves two user choices that are not part of the hypothesis: the random split of the sample and the bandwidth of the smoothed bootstrap. Checking that neither choice drives the conclusion is a routine diagnostic, to be run in any application of the method and not only here. We report the details in Section~\ref*{sec:supp-sensitivity} of the supplementary material and summarize them now. Across ten independent random splits, the estimated common direction moves by at most $2.3^{\circ}$, as expected from the well separated eigenvalues reported above. The tests with $k=10$ and $k=100$ reject at the $5\%$ level in all ten splits. Varying the bandwidth multiplier from $0.1$ to $2.0$ times its base value (so $b_n/s_n$ ranges from $0.0265$ to $0.5308$), and recomputing the test with $k=100$ and $B=300$ replicates, leaves the $p$-value near the minimum attainable value ($1/301\approx0.0033$) in every case.

\paragraph{Interpretation and limitations}

The rejection is consistent with the financial narrative of the carry trade: the strategy is exposed to a severe crash risk along a single common direction. The test also has natural limitations. It shows that the asymmetry is present, but it does not identify which investors trigger a crash, nor does it establish the chain of events during a sell-off, nor does it prove that crash risk is the only source of the profitability of the strategy. In addition, our theory assumes independent observations, while financial markets often go through prolonged periods of turbulence. Working with log-returns at least removes the instability of the raw exchange-rate levels and delivers a covariance-stationary series. The i.i.d. bootstrap we use could in principle be replaced by a block or stationary bootstrap to accommodate such temporal dependence directly, but we leave this extension for future work.

\section{Discussion and concluding remarks}\label{sec:discussion}

We have developed an asymptotically valid and consistent bootstrap test for invariance of a \(d\)-dimensional distribution under an orthogonal reflection about an unknown mirror subspace, of which only the dimension is restricted to a prespecified set \(D\subseteq\{0,1,\dots,d-1\}\). The solution combines three ingredients that, we believe, form a reusable template for testing invariance under group actions indexed by unknown parameters. First, a structural identity links the unknown parameter to an estimable population feature: any distributional symmetry must commute with $\Sigma$, and under a simple spectrum this reduces the parameter space from the orthogonal group to the finite family of principal-subspace reflections. Second, sample splitting decouples the estimation of the parameter from the empirical process defining the statistic, and makes the linearization of the plug-in transformation tractable. Third, we resample from a data-dependent law that satisfies the null invariance exactly: the empirical Haar average over the estimated action (for \(G=\mathbb Z_2\), the symmetrized empirical law), smoothed to meet the regularity required by the triangular-array limit theorem. Extending the approach to richer groups (rotational or cylindrical symmetry, elliptical symmetry after an unknown affine transformation) would require a group-specific reduction step for the first ingredient. For continuous groups the candidate family is no longer finite, and calibrating over an infinite family of estimated representations is the main open challenge.

The optimal allocation of the sample among the three splits (Remark~\ref{rem:sample_splitting}) remains as an open question within the present framework as well.

	
\section*{Supplementary Material}
This supplementary material contains all proofs: the proofs of Propositions~\ref{prop:commute}--\ref{prop:smoothed-bootstrap-alt}, Theorems~\ref{teo:mainth1-TA}, \ref{thm:bootstrap} and \ref{thm:power}, and Corollary~\ref{cor:finite-projections}, the derivation of the covariance kernel \eqref{eq:cov-kernel}, the detailed verification of the regularity conditions for the smoothed symmetrized bootstrap law, and the statements and proofs of all the auxiliary lemmas used in these arguments.

\setcounter{section}{0}
\renewcommand{\thesection}{S\arabic{section}}
\renewcommand{\theHsection}{S\arabic{section}}
\numberwithin{equation}{section}

	\section{Notation and conventions}

Results numbered with the prefix \textup{S} (Lemma~\ref{lem:muSigma}, Section~\ref{sec:supp-aux}, \eqref{eq:4-terms-TA-corrected}, etc.) belong to this Supplementary Material; all other references, such as ``Theorem~\ref{teo:mainth1-TA}'', ``Proposition~\ref{prop:commute}'' or ``\eqref{eq:cov-kernel}'', point to the main paper. 
Throughout we fix \(S\in\mathcal S\) and use the following notation. We write \(\aleph_{n,1},\aleph_{n,2},\aleph_{n,3}\) both for the three subsamples and for the corresponding partition of \(\{1,\ldots,n\}\), and \(\QS_{S,n}=2\sum_{i\in S}u_{i,n}u_{i,n}^\top-I_d\), \(\QS_S=2\sum_{i\in S}u_iu_i^\top-I_d\), \(\widehat\QS_S^3=2\sum_{i\in S}\widehat u_{i,n}^3(\widehat u_{i,n}^3)^\top-I_d\) -- the reflection associated with the \(n\)-th law, the limiting reflection and the plug-in reflection, respectively -- and \(\QS_{S,W}=2\sum_{i\in S}w_iw_i^\top-I_d\) for a frame \(W=(w_i)_{i\in S}\). We write \(P_n^c\) for the law of \(X_{n,1}-\mu_n\), \(\widehat P_n^{c,r}\) for its empirical counterpart over \(\aleph_{n,r}\), and \(\rho_n(\ga,\ga')=\|\ga-\ga'\|_{L_2(P_n^c)}\); \((Q\ga)(x)=\ga(Qx)\), so \(Q\ga_{a,b}=\ga_{Q^\top a,b}\) and \(Q^\top=Q\) for our reflections.

\section{Proofs of the structural results}

\begin{proof}[Proof of Proposition~\ref{prop:commute}]
	If $X-\mu\overset{\mathcal L}{=}Q(X-\mu)$, then, since $Q$ is linear and $\mathbb E[Q(X-\mu)]=0$, $\Sigma=\cov(X-\mu)=\cov(Q(X-\mu))=Q\Sigma Q^\top$. Because $Q$ is orthogonal, $Q^\top=Q^{-1}$, and right-multiplying by $Q$ gives $\Sigma Q=Q\Sigma$; that is, $Q$ commutes with $\Sigma$.
	
	Assume now that $\Sigma$ has simple spectrum. Fix $i$. From $Q\Sigma=\Sigma Q$, $\Sigma(Qu_i)=Q\Sigma u_i=\lambda_i (Qu_i)$, so $Qu_i$ is an eigenvector of $\Sigma$ with eigenvalue $\lambda_i$. Since $\lambda_i$ is simple, the corresponding eigenspace is $\operatorname{span}(u_i)$, whence $Qu_i\in\operatorname{span}(u_i)$. As $Q$ is orthogonal, $\|Qu_i\|=1$, so $Qu_i=\epsilon_i u_i$ with $\epsilon_i\in\{-1,+1\}$. Writing $Q$ in the basis $\{u_i\}$, $Q=\sum_i\epsilon_i u_iu_i^\top$; setting $S=\{i:\epsilon_i=+1\}$ and using $\sum_i u_iu_i^\top=I_d$, $Q=\sum_{i\in S}u_iu_i^\top-\sum_{i\notin S}u_iu_i^\top=2\PS_S-I_d$, which is $\QS_{V_S}$ from \eqref{EqDefQ}.
\end{proof}

\begin{proof}[Proof of Corollary~\ref{Coro4.2}]
	By the first part of Proposition~\ref{prop:commute} (which requires only invariance), each \(R_{u_j}\) commutes with \(\Sigma\). Hence \(R_{u_j}(\Sigma u_j)=\Sigma(R_{u_j}u_j)=\Sigma u_j\), so \(\Sigma u_j\) lies in the \(+1\)-eigenspace of \(R_{u_j}\), which is \(\operatorname{span}(u_j)\); therefore each \(u_j\) is an eigenvector of \(\Sigma\). Since eigenvectors for distinct eigenvalues are orthogonal, and \(u_{j-1},u_j\) are not orthogonal for \(j=2,\dots,k\), all \(u_1,\dots,u_k\) share the same eigenvalue \(\delta\ge0\). Thus \(\Sigma v=\delta v\) for every \(v\in W\), and \(\cov(\PS_W(X-\mu))=\PS_W\Sigma\PS_W=\delta\PS_W\), since \(\PS_W x\in W\).
\end{proof}

\section{Auxiliary results}\label{sec:supp-aux}

This section collects the auxiliary lemmas used in the proofs, each one followed by its own proof. The proofs of Lemmas~\ref{lem:muSigma}, \ref{lem:uniform-entropy-vw} and \ref{lem:lemaux-TA-global} do not involve the reflection map; the two lemmas that genuinely use the mirror structure -- the principal-component linearization (Lemma~\ref{lem:eigenvector-linearization}) and the differentiability of the reflection map (Lemma~\ref{lem:dif-QRu-TA}) -- are stated and proved for a general mirror subspace, the axial case being recovered when \(|S|=1\).

\begin{lemma}\label{lem:muSigma}
	Assume \(d_{\mathcal F}(P_n,P)\to 0\) and \textup{(T1)}. Let \(X_n\sim P_n\), \(X\sim P\). Set \(\mu_n=\mathbb{E}[X_n]\), \(\mu=\mathbb{E}[X]\), \(\Sigma_n=\cov(X_n)\), \(\Sigma=\cov(X)\). Then \(\mu_n\to\mu\), \(\Sigma_n\to\Sigma\), and for every \(u\in\mathbb S^{d-1}\) and \(s\in[1,2]\), \(\sup_{n}\mathbb{E}|u^\top(X_n-\mu_n)|^s<\infty\) and \(\sup_{n,u}\mathbb{P}(|u^\top(X_n-\mu_n)|>M)\le C M^{-2}\) for a finite constant \(C\) independent of \(n,u,M\).
\end{lemma}

\begin{proof}[Proof of Lemma~\ref{lem:muSigma}]
	Since half-spaces are convergence-determining, \(d_{\mathcal F}(P_n,P)\to 0\) implies \(P_n\xrightarrow{\mathcal L}P\). The bound \(\sup_n \mathbb{E}\|X_n\|^4<\infty\) from \textup{(T1)} makes \(\{\|X_n\|^p\}\) uniformly integrable for \(p\le2\); convergence of moments \cite[Thm.~3.5]{billingsley1999} gives \(\mu_n\to\mu\) and, entrywise, \(\Sigma_n\to\Sigma\). The uniform bounds on \(\mathbb{E}|u^\top(X_n-\mu_n)|^s\) and the tail estimate follow from Lyapunov's and Markov's inequalities together with \(\sup_n\mathbb{E}\|X_n\|^2<\infty\) and \(\sup_n\|\mu_n\|<\infty\).
\end{proof}

\begin{lemma}\label{lem:eigenvector-linearization}
	Assume \(d_{\mathcal F}(P_n,P)\to0\) and \textup{(T1)}--\textup{(T4)}. Let \(\widehat \Sigma_n^3\) be the sample covariance over \(\aleph_{n,3}\), \(m_n\sim n/3\), and for \(i=1,\dots,d\) let \(\widehat u_{i,n}^3\) be its \(i\)-th unit eigenvector with \(\langle \widehat u_{i,n}^3,u_{i,n}\rangle\ge0\). Then, for every \(i\), \(\sqrt n (\widehat u_{i,n}^3-u_{i,n})=\sqrt n (\widehat  P_n^{ 3}-P_n)\psi_{i,n}+o_{\mathbb P}(1)\), and for any \(T\subseteq\{1,\dots,d\}\), \(\bigl(\sqrt n (\widehat u_{i,n}^3-u_{i,n})\bigr)_{i\in T}\xrightarrow{\mathcal L}(Z_i^{3})_{i\in T}\), a centered Gaussian vector with \(\mathrm{Cov}(Z_i^3,Z_j^3)=3\,P(\psi_i\psi_j^\top)\), where
	\[
	\psi_{i,n}(x)=\sum_{j\neq i}\frac{(u_{j,n}^\top(x-\mu_n))(u_{i,n}^\top(x-\mu_n))}{\lambda_{i,n}-\lambda_{j,n}} u_{j,n},
	\qquad
	\psi_i(x)=\sum_{j\neq i}\frac{(u_j^\top(x-\mu))(u_i^\top(x-\mu))}{\lambda_i-\lambda_j} u_j.
	\]
\end{lemma}

\begin{proof}[Proof of Lemma~\ref{lem:eigenvector-linearization}]
	Write $E_n=\widehat \Sigma_n^3-\Sigma_n,$ $\delta_{i,n}=\min_{j\neq i}|\lambda_{i,n}-\lambda_{j,n}|,$ $\widehat  P_n^{ 3}f=\frac1{n-2m_n}\sum_{k\in\aleph_{n,3}}f(X_{n,k}),$ $Y_{n,k}=X_{n,k}-\mu_n,$ $\Delta_n=\frac1{n-2m_n}\sum_{k\in\aleph_{n,3}}Y_{n,k}.$ By \textup{(T4)}, \(\delta_{i,n}\ge c>0\) for \(n\) large. Since \(E_n=\frac1{n-2m_n}\sum_{k\in\aleph_{n,3}}(Y_{n,k}Y_{n,k}^\top-\Sigma_n)-\Delta_n\Delta_n^\top\) and \textup{(T1)} gives \(\sup_n\E\|X_n\|^4<\infty\), the first term is \(O_{\mathbb P}(m_n^{-1/2})\) in operator norm and \(\|\Delta_n\Delta_n^\top\|_{\op}=O_{\mathbb P}(m_n^{-1})\); hence \(\|E_n\|_{\op}=O_{\mathbb P}(m_n^{-1/2})\) and \(\mathbb P(\|E_n\|_{\op}<c/2)\to1\). On \(\{\|E_n\|_{\op}<c/2\}\), perturbation theory for simple eigenvalues \cite[Thm.~V.2.8]{stewart1990matrix} gives \(\widehat u_{i,n}^3-u_{i,n}=\sum_{j\neq i}\frac{u_{j,n}^\top E_nu_{i,n}}{\lambda_{i,n}-\lambda_{j,n}} u_{j,n}+r_{i,n}\), \(\|r_{i,n}\|\le C\|E_n\|_{\op}^2\). Since \(u_{j,n}^\top\Sigma_nu_{i,n}=0\) for \(j\neq i\), \(u_{j,n}^\top E_nu_{i,n}=\frac1{n-2m_n}\sum_{k}(u_{j,n}^\top Y_{n,k})(u_{i,n}^\top Y_{n,k})-(u_{j,n}^\top\Delta_n)(u_{i,n}^\top\Delta_n)\); summing over \(j\neq i\), \(\sum_{j\neq i}\frac{u_{j,n}^\top E_nu_{i,n}}{\lambda_{i,n}-\lambda_{j,n}} u_{j,n}=(\widehat  P_n^{ 3}-P_n)\psi_{i,n}+R_{n,1}\), \(\|R_{n,1}\|=O_{\mathbb P}(m_n^{-1})\), while \(P_n\psi_{i,n}=0\). Thus \(\widehat u_{i,n}^3-u_{i,n}=(\widehat  P_n^{ 3}-P_n)\psi_{i,n}+O_{\mathbb P}(m_n^{-1})\), and since \(m_n\sim n/3\), \(\sqrt n (\widehat u_{i,n}^3-u_{i,n})=\sqrt n (\widehat  P_n^{ 3}-P_n)\psi_{i,n}+o_{\mathbb P}(1)\).
	
	Joint limit. Fix \(T\subseteq\{1,\dots,d\}\) and set \(\Xi_{n,k}=(\psi_{i,n}(X_{n,k}))_{i\in T}\), \(k\in\aleph_{n,3}\); these are i.i.d.\ within the row and centered (\(P_n\psi_{i,n}=0\)). Because \(\delta_{i,n}\ge c\) and \(\sup_n\|\mu_n\|<\infty\) (Lemma~\ref{lem:muSigma}), \(\|\psi_{i,n}(X_{n,k})\|\le C\|X_{n,k}-\mu_n\|^2\), so \textup{(T1)} gives \(\sup_n\E\|\Xi_{n,1}\|^{2+\eta/2}<\infty\); and \(P_n(\psi_{i,n}\psi_{j,n}^\top)\to P(\psi_i\psi_j^\top)\). To justify the latter, split \(|P_n(\psi_{i,n}\psi_{j,n}^\top)-P(\psi_i\psi_j^\top)|\le|P_n(\psi_{i,n}\psi_{j,n}^\top-\psi_i\psi_j^\top)|+|P_n(\psi_i\psi_j^\top)-P(\psi_i\psi_j^\top)|\): the first term \(\to0\) by truncation at radius \(R\): \(P_n|\psi_{i,n}\psi_{j,n}^\top-\psi_i\psi_j^\top|\le P_n[|\psi_{i,n}\psi_{j,n}^\top-\psi_i\psi_j^\top|\mathbf 1_{\{\|x\|\le R\}}]+P_n[|\psi_{i,n}\psi_{j,n}^\top-\psi_i\psi_j^\top|\mathbf 1_{\{\|x\|>R\}}]\); on \(\{\|x\|\le R\}\) the integrand \(\to0\) uniformly (from \(u_{j,n}\to u_j\), \(\lambda_{j,n}\to\lambda_j\), \(\mu_n\to\mu\)), and the tail is \(\le C\,P_n[(1+\|X\|^4)\mathbf 1_{\{\|X\|>R\}}]\), uniformly small for large \(R\) by the \((4+\eta)\)-moment bound in \textup{(T1)}; the second \(\to0\) by \(P_n\xrightarrow{\mathcal L}P\) (Lemma~\ref{lem:muSigma}) applied to the \(O(\|x\|^4)\)-growth map \(\psi_i\psi_j^\top\), again by uniform integrability and truncation. Hence \(\Cov(\Xi_{n,1})\to\Gamma:=(P(\psi_i\psi_j^\top))_{i,j\in T}\). With \(N_{3,n}=n-2m_n\), Lyapunov's condition holds and the multivariate CLT gives \(N_{3,n}^{-1/2}\sum_k\Xi_{n,k}\xrightarrow{\mathcal L}\mathcal N(0,\Gamma)\). Since \(N_{3,n}/n\to1/3\), \((\sqrt n(\widehat P_n^3-P_n)\psi_{i,n})_{i\in T}\xrightarrow{\mathcal L}\mathcal N(0,3\Gamma)\), i.e.\ \((Z_i^3)_{i\in T}\) with \(\Cov(Z_i^3,Z_j^3)=3P(\psi_i\psi_j^\top)\). Combining with the linearization proves the claim; \(T=\{i\}\) recovers the axial marginal.
\end{proof}

\begin{lemma}\label{lem:uniform-entropy-vw}
	There exist \(A<\infty\), \(v<\infty\) such that for all \(\varepsilon\in(0,1)\), \(\sup_{n\ge1}\log N(\varepsilon,\mathcal F,\rho_n)\le v\log(A/\varepsilon)\), and consequently \(\int_0^1\sqrt{\sup_{n\ge1}\log N(\varepsilon,\mathcal F,\rho_n)}\,d\varepsilon<\infty\).
\end{lemma}

\begin{proof}[Proof of Lemma~\ref{lem:uniform-entropy-vw}]
	\(\mathcal F\) is a VC class of indicators with envelope \(1\); by \cite[Th.\ 2.6.7]{vaart2023empirical} there are universal \(A,v\) with \(N(\varepsilon,\mathcal F,L_2(Q))\le(A/\varepsilon)^v\) for every probability measure \(Q\) and \(\varepsilon\in(0,1)\). Taking \(Q=P_n^c\) and \(\sup_n\) gives the entropy bound, and \(\int_0^1\sqrt{v\log(A/\varepsilon)}\,d\varepsilon<\infty\) gives the integral bound.
\end{proof}

To handle the reflection map for a general mirror subspace, call an \emph{\(|S|\)-frame} any ordered tuple \(W=(w_i)_{i\in S}\) of vectors in \(\mathbb R^d\), said to be orthonormal when its entries are orthonormal vectors. We fix an open neighborhood \(\mathcal U\) of the set of orthonormal \(|S|\)-frames such that \(Q_{S,W}=2\sum_{i\in S}w_iw_i^\top-I_d\) is invertible for every frame \(W\in\mathcal U\).

\begin{lemma}\label{lem:dif-QRu-TA}
	Assume \textup{(T1)}--\textup{(T3)} and \(d_{\mathcal F}(P_n,P)\to0\). For a frame \(W=(w_i)_{i\in S}\in\mathcal U\), let \(Q_{S,W}=2\sum_{i\in S}w_iw_i^\top-I_d\) and \(\Phi_n^{S}(W)[\ga]:= \int_{\mathbb R^d}\ga(x)\, f_n(Q_{S,W}x+\mu_n)\, dx\). Then \(\Phi_n^{S}\) is Fr\'echet differentiable at every orthonormal frame \(U=(u_i)_{i\in S}\), with \(D\Phi_n^{S}(U)[V](\ga)=\sum_{i\in S}G_{n,U,i}(\ga)^\top v_i\), \(V=(v_i)_{i\in S}\),
	\[
	G_{n,U,i}(\ga)=2\int_{\mathbb R^d}\ga(x)\bigl((u_i^\top x)I_d+xu_i^\top\bigr)\nabla f_n(Q_{S,U}x+\mu_n)\, dx,\qquad i\in S.
	\]
	Moreover, there exist \(r_0>0\) and \(C_n=o(\sqrt n)\), independent of \(U\), such that whenever \(\|V\|\le r_0\) and \(U+V\in\mathcal U\), \(\bigl\|\Phi_n^{S}(U+V)-\Phi_n^{S}(U)-D\Phi_n^{S}(U)[V]\bigr\|_{\mathcal F}\le C_n\|V\|^2\) and \(\sup_{n\ge1}\sup_{U}\sup_{i\in S}\|G_{n,U,i}\|_\mathcal{F}<\infty\). If \(U_n\to U\) (orthonormal frames), then \(\|G_{n,U_n,i}-G_{u_i}^{S}\|_\mathcal{F}\to0\) for every \(i\in S\), with \(G_{u_i}^{S}\) as in \eqref{eq:G-general}. Finally, \(\Phi_n^{S}(U)(\ga)=P_n^c(Q_{S,U}\ga)\) for every orthonormal \(U\) and \(\ga\in\mathcal F\).
\end{lemma}

\begin{proof}[Proof of Lemma~\ref{lem:dif-QRu-TA}]
	Here \(\|V\|^2=\sum_{i\in S}\|v_i\|^2\). Since \(\QS_{S,W}\) is invertible on \(\mathcal U\), \(\|\Phi_n^{S}(W)\|_\mathcal{F}\le\int f_n(\QS_{S,W}x+\mu_n)\,dx=|\det(\QS_{S,W})|^{-1}<\infty\); and if \(U\) is orthonormal, \(\QS_{S,U}^{-1}=\QS_{S,U}\), so \(\Phi_n^{S}(U)(\ga)=P_n^c(\QS_{S,U}\ga)\). Fix an orthonormal \(U=(u_i)_{i\in S}\). The reflection increment is
	\[
	\Delta_{U,V}(x)=\QS_{S,U+V}x-\QS_{S,U}x=2\sum_{i\in S}\bigl((u_i^\top x)I_d+u_ix^\top\bigr)v_i+2\sum_{i\in S}(v_i^\top x)v_i,
	\]
	linear plus quadratic in \(V\); the candidate differential is \(D\Phi_n^{S}(U)[V](\ga)=\sum_{i\in S}G_{n,U,i}(\ga)^\top v_i\) (the transpose turns \(u_ix^\top\) into \(xu_i^\top\)). Since \(\|(u_i^\top x)I_d+xu_i^\top\|_{\op}\le2\|x\|\), the change of variables \(y=\QS_{S,U}x+\mu_n\) and \textup{(T3)} give \(\|G_{n,U,i}\|_\mathcal{F}\le C\int(1+\|y\|)\|\nabla f_n(y)\|\,dy<\infty\) uniformly, so \(\sup_{n,U,i}\|G_{n,U,i}\|_\mathcal{F}<\infty\). By Taylor's theorem,
	\begin{multline*}
	f_n(\QS_{S,U+V}x+\mu_n)=f_n(\QS_{S,U}x+\mu_n)+\nabla f_n(\QS_{S,U}x+\mu_n)^\top\Delta_{U,V}(x)\\
	+\tfrac12\Delta_{U,V}(x)^\top\!\Big(\!\int_0^1\!Hf_n(\cdots)dt\Big)\Delta_{U,V}(x).
	\end{multline*}
	The linear part of \(\Delta_{U,V}\) reproduces \(\sum_{i\in S}G_{n,U,i}(\ga)^\top v_i\) plus a term quadratic in \(V\) bounded by \(C\|V\|^2\int(1+\|y\|)\|\nabla f_n\|\le C\|V\|^2\). For the Hessian term, \(\|\Delta_{U,V}(x)\|\le6\sqrt{|S|}\|x\|\|V\|\), so after \(y=\QS_{S,U}x+\mu_n\), \(\|\Delta_{U,V}(x)\|^2\le C(1+\|y\|^2)\|V\|^2\); since \(\Delta_{U,V}\) is linear in \(x\), the substitution \(z=y+t\Delta_{U,V}(\QS_{S,U}^\top(y-\mu_n))\) has affine map with linear part \(I_d+2t(\sum_{i\in S}(v_iu_i^\top+u_iv_i^\top+v_iv_i^\top))\QS_{S,U}^\top\), invertible for \(\|V\|\) small with uniformly bounded inverse operator norm and determinant; hence \(\int(1+\|y\|^2)\|Hf_n(\cdots)\|_{\op}dy\le C\int(1+\|z\|^2)\|Hf_n(z)\|_{\op}dz\le CA_n\) by \textup{(T3)}. Thus \(\|\Phi_n^{S}(U+V)-\Phi_n^{S}(U)-D\Phi_n^{S}(U)[V]\|_\mathcal{F}\le C_n\|V\|^2\) with \(C_n=C_0+C_1A_n=o(\sqrt n)\), so \(\Phi_n^{S}\) is Fr\'echet differentiable at \(U\).
	
	For the last claim, let \(U_n\to U\) be orthonormal frames and fix \(i\). Splitting \(G_{n,U_n,i}-G_{u_i}^{S}\) into a density difference (bounded by \(C\int(1+\|y\|)\|\nabla f_n-\nabla f\|+C\int(1+\|y\|)\|\nabla f(\cdot+\mu_n-\mu)-\nabla f\|\to0\) by \textup{(T3)} and \(L_1\)-continuity of translations) and a frame difference (bounded by \(C\|u_{i,n}-u_i\|\int(1+\|y\|)\|\nabla f\|\) plus an integral that vanishes as \(\QS_{S,U_n}\QS_{S,U}\to I_d\), by \(L_1\)-approximation of \((1+\|y\|)\nabla f\) by a compactly supported continuous map), we get \(\|G_{n,U_n,i}-G_{u_i}^{S}\|_\mathcal{F}\to0\).
\end{proof}

\begin{lemma}\label{lem:lemaux-TA-global}
	Assume \(d_\mathcal{F}(P_n,P)\to 0\), \textup{(T1)} and \textup{(T2)}. Then \(\sqrt n(\widehat  P_n^{c,1}-P_n^c,\widehat  P_n^{c,2}-P_n^c)\xrightarrow{\mathcal L}(\mathbb L^{(1)},\mathbb L^{(2)})\) in \(\ell^\infty(\mathcal F)\), where \(\mathbb L^{(1)},\mathbb L^{(2)}\) are independent centered Gaussian processes with
	\[
	\mathbb L^{(r)}(\ga_{a,b})=\sqrt 3\big(\mathbb B^{r,c}(\ga_{a,b})+f_a(b) a^\top Z^{r}\big),\qquad (a,b)\in\mathbb S^{d-1}\times\mathbb R,
	\]
	with \(\mathbb B^{r,c}\) a \(P^c\)-Brownian bridge, \(Z^{r}\sim\mathcal N(0,\Sigma)\), jointly Gaussian with \(\mathrm{Cov}(\mathbb B^{r,c}(\ga_{a,b}), c^\top Z^{r})=P^c[(\ga_{a,b}-P^c\ga_{a,b}) c^\top Y]\), \(Y\sim P^c\).
\end{lemma}

\begin{proof}[Proof of Lemma~\ref{lem:lemaux-TA-global}]
	Fix \(r\in\{1,2\}\), \(m_n=|\aleph_{n,r}|\sim n/3\). Set \(\mathbb P_{n,r}^{\mu_n}(\ga)=\frac1{m_n}\sum_{j\in\aleph_{n,r}}\ga(X_{n,j}-\mu_n)\), \(\alpha_{n,r}=\sqrt{m_n}(\mathbb P_{n,r}^{\mu_n}-P_n^c)\), \(\Delta_{n,r}=\bar X_n^{r}-\mu_n\), \(\beta_{n,r}=\sqrt{m_n}\Delta_{n,r}\), \(\mathbb G_n^{r}=\sqrt{n/m_n}\,\alpha_{n,r}\). Since \(\widehat  P_n^{c,r}(\ga_{a,b})=\mathbb P_{n,r}^{\mu_n}(\ga_{a,b+a^\top\Delta_{n,r}})\),
	\begin{equation}\label{eq:decomp-global-fixed-corr}
		\sqrt n (\widehat  P_n^{c,r}-P_n^c)(\ga_{a,b})=\underbrace{\mathbb G_n^{r}(\ga_{a,b+a^\top\Delta_{n,r}})-\mathbb G_n^{r}(\ga_{a,b})}_{I_{n,1}}+\underbrace{\mathbb G_n^{r}(\ga_{a,b})}_{I_{n,2}}+\underbrace{\sqrt n(P_n^c\ga_{a,b+a^\top\Delta_{n,r}}-P_n^c\ga_{a,b})}_{I_{n,3}}.
	\end{equation}
	By Lemma~\ref{lem:uniform-entropy-vw} and \cite[Section~2.8.3, Theorem~2.8.9]{vaart2023empirical}, \(\alpha_{n,r}\) is asymptotically tight and \(\rho_n\)-equicontinuous in probability. Since \((\ga_{a,b+a^\top\delta}-\ga_{a,b})^2\le\mathbf 1\{|a^\top x-b|\le\|\delta\|\}\), \textup{(T2)} gives \(\sup_{a,b}\rho_n^2(\ga_{a,b+a^\top\Delta_{n,r}},\ga_{a,b})\le2K_0\|\Delta_{n,r}\|\) a.s.; as \(\Delta_{n,r}=O_{\mathbb P}(n^{-1/2})\) (Lemma~\ref{lem:muSigma}), \(\|I_{n,1}\|_{\mathcal F}=o_{\mathbb P}(1)\) by equicontinuity. By Lemma~\ref{lem:muSigma}, \(P_n^c\xrightarrow{\mathcal L}P^c\); indicators and their products have \(P^c\)-null boundaries by \textup{(T2)}, so \(\Cov(\alpha_{n,r}(\ga),\alpha_{n,r}(\ga'))\to P^c(\ga\ga')-P^c\ga P^c\ga'\), and \(I_{n,2}\xrightarrow{\mathcal L}\sqrt3\,\mathbb B^{r,c}\). For \(I_{n,3}\), Taylor and \textup{(T2)} give \(I_{n,3}(a,b)=f_{n,a}(b)a^\top\sqrt n\Delta_{n,r}+r_{n,r}\) with \(\sup_{a,b}|r_{n,r}|\le\frac12K_{1,n}\sqrt n\|\Delta_{n,r}\|^2=o_{\mathbb P}(1)\); with \(J_{n,r}(\ga_{a,b})=f_a(b)a^\top\sqrt n\Delta_{n,r}\) and \(\sup_{a,b}|f_{n,a}-f_a|\to0\), \(\|I_{n,3}-J_{n,r}\|_{\mathcal F}=o_{\mathbb P}(1)\).
	
	Joint convergence \((\alpha_{n,r},\beta_{n,r})\xrightarrow{\mathcal L}(\mathbb B^{r,c},Z^{r})\), \(Z^{r}\sim N(0,\Sigma)\), follows from asymptotic tightness of \(\alpha_{n,r}\), tightness of \(\beta_{n,r}\) (\(\mathbb E\|\beta_{n,r}\|^2=\mathrm{tr}(\Sigma_n)\)), and Cram\'er--Wold: for \(\ga_1,\dots,\ga_m\) and \(c\), \((\alpha_{n,r}(\ga_\ell),c^\top\beta_{n,r})=m_n^{-1/2}\sum_j\xi_{n,j}\) with bounded first coordinates and \(\sup_n\E\|Y_{n,1}\|^{2+\eta}<\infty\), so the multivariate CLT applies; the mixed covariance is \(P^c[(\ga_\ell-P^c\ga_\ell)c^\top Y]\), since \(P_n^c(\ga_\ell c^\top\cdot)\to P^c(\ga_\ell c^\top\cdot)\): the map \(y\mapsto\ga_\ell(y)c^\top y\) has linear growth and is discontinuous only on the half-space boundary \(\{a_\ell^\top y=b_\ell\}\), which is \(P^c\)-null (as \(a_\ell^\top Y\) has a density by \textup{(T2)}), so the convergence follows from \(P_n^c\xrightarrow{\mathcal L}P^c\) by truncating \(\ga_\ell(y)c^\top y\) at radius \(R\) and using the uniform integrability of \(\{\|c^\top Y_n\|\}\) supplied by \textup{(T1)}. Multiplying by \(\sqrt{n/m_n}\to\sqrt3\) and using continuity of \(T(z)(\ga_{a,b})=f_a(b)a^\top z\), the continuous mapping theorem, \(\|I_{n,3}-J_{n,r}\|_{\mathcal F}=o_{\mathbb P}(1)\), \(\|I_{n,1}\|_{\mathcal F}=o_{\mathbb P}(1)\) and Slutsky give \(\sqrt n(\widehat P_n^{c,r}-P_n^c)\xrightarrow{\mathcal L}\mathbb L^{(r)}\) with \(\mathbb L^{(r)}(\ga_{a,b})=\sqrt3(\mathbb B^{r,c}(\ga_{a,b})+f_a(b)a^\top Z^{r})\). Finally, independence of \(\aleph_{n,1},\aleph_{n,2}\) factorizes the joint characteristic function, so \(\mathbb L^{(1)},\mathbb L^{(2)}\) are independent.
\end{proof}

\begin{lemma}\label{lem:spectral-id}
	Let \(\Sigma\) be symmetric with simple spectrum \(\lambda_1>\cdots>\lambda_d\) and orthonormal eigenvectors \(u_1,\dots,u_d\), let \(S\subseteq\{1,\dots,d\}\), and \(V_S=\operatorname{span}\{u_i:i\in S\}\). Let \((\Sigma_n)\) be symmetric with \(\|\Sigma_n-\Sigma\|_{\op}\to0\), and for each \(n\) let \(W_n\subseteq\mathbb R^d\) be \(\Sigma_n\)-invariant with \(\dim W_n=|S|\) and \(\|\PS_{W_n}-\PS_{V_S}\|_{\op}\to0\). Then, for \(n\) large, \(W_n\) is the sum of the eigenspaces of \(\Sigma_n\) whose eigenvalues occupy the ranks in \(S\) within the decreasing ordering of the spectrum of \(\Sigma_n\); consequently \(2\PS_{W_n}-I_d=\QS_{S,n}\).
\end{lemma}

\begin{proof}[Proof of Lemma~\ref{lem:spectral-id}]
	By Weyl's inequality \(|\lambda_i(\Sigma_n)-\lambda_i|\le\|\Sigma_n-\Sigma\|_{\op}\to0\), so for \(n\) large the spectrum of \(\Sigma_n\) is simple with gap \(\ge g:=\tfrac12\min_{i\neq j}|\lambda_i-\lambda_j|>0\). Fix such \(n\) and let \(u_{1,n},\dots,u_{d,n}\) be its ordered unit eigenvectors. As \(\Sigma_n\) is symmetric and \(W_n\) is \(\Sigma_n\)-invariant, \(W_n=\bigoplus_{i\in T_n}\operatorname{span}(u_{i,n})\) for a unique \(T_n\) with \(|T_n|=|S|\), and \(\PS_{W_n}=\sum_{i\in T_n}u_{i,n}u_{i,n}^\top\). Davis--Kahan gives \(\|u_{i,n}u_{i,n}^\top-u_iu_i^\top\|_{\op}\le\|\Sigma_n-\Sigma\|_{\op}/g\to0\). If \(T_n\neq S\) along a subsequence, pick \(k\in T_n\setminus S\); then \(u_{k,n}\in W_n\), \(u_{k,n}\to u_k\), so \(\PS_{W_n}u_k\to u_k\) while \(\PS_{V_S}u_k=0\), giving \(\|\PS_{W_n}-\PS_{V_S}\|_{\op}\ge\tfrac12\), a contradiction. Hence \(T_n=S\) for \(n\) large and \(2\PS_{W_n}-I_d=\QS_{S,n}\).
\end{proof}

\begin{lemma}\label{lem:continuity-J}
	Assume the fixed-law versions of \textup{(T2)} and \textup{(T4)} (the limit uses the fixed law \(P\)) and the integrability \(\int_{\mathbb R^d}(1+\|x\|)\|\nabla f(x)\|\,dx<\infty\) (so that the coefficients \(G_{u_i}^{S}\) in \eqref{eq:G-general} are well defined and \(\mathbb L_h^{S}\) exists as in Theorem~\ref{teo:mainth1-TA}). Then \(J(x):=\mathbb P(\|\mathbb L_h^{S}\|_{\mathcal F_h}\le x)\) is continuous on \(\mathbb R\).
\end{lemma}

\begin{proof}[Proof of Lemma~\ref{lem:continuity-J}]
	Let \(Y\sim P^c\), \(W:=h^\top Y\). By \textup{(T4)}, \(\var(W)=h^\top\Sigma h>0\); by \textup{(T2)}, \(W\) has density \(f_h\) and continuous c.d.f.\ \(F_h\). Fix \(b_0\) with \(0<F_h(b_0)<1\). From \eqref{eq:cov-kernel}, \(\var(\mathbb L_h^{S}(\ga_{h,b}))=\var(\mathbb L^{(1)}(\ga_{h,b}))+\var(\mathbb L^{(2)}(\ga_{h,b}))+\var(\sum_{i\in S}G_{u_i}^{S}(\ga_{h,b})^\top Z_i^3)\), so it suffices that \(\var(\mathbb L^{(1)}(\ga_{h,b_0}))>0\), i.e.\ \(\var(\mathbf 1\{W\le b_0\}+f_h(b_0)W)>0\). If it vanished, \(\mathbf 1\{W\le b_0\}+f_h(b_0)W=c\) a.s.\ for some constant \(c\); if \(f_h(b_0)=0\) the vanishing of the variance would force \(\mathbf 1\{W\le b_0\}\) to be a.s.\ constant, contradicting \(0<F_h(b_0)<1\), while if \(f_h(b_0)\ne0\) it forces \(W\) to be constant on \(\{W>b_0\}\), contradicting that \(W\) has a density. Hence \(\mathbb P(\|\mathbb L_h^{S}\|_{\mathcal F_h}=0)=0\). We work with a separable modification of the process. Restricted to \(\mathcal F_h\), the marginal law of \(b\mapsto\mathbb B^{r,c}(\ga_{h,b})\), \(r=1,2\), coincides with that of \(b\mapsto B_r(F_h(b))\), since both are centered Gaussian processes with covariance \(F_h(b\wedge b')-F_h(b)F_h(b')\); as \(F_h\) is continuous, each restricted bridge admits a continuous modification, which we choose without altering its joint finite-dimensional distributions with \(Z^1,Z^2\) (the family \((\mathbb B^{r,c}(\ga_{h,b}),Z^r)_b\) is jointly Gaussian, so such a modification exists and leaves all cross-covariances, in particular \(\Cov(\mathbb B^{r,c}(\ga_{h,b}),c^\top Z^r)\), unchanged). Moreover \(b\mapsto f_h(b)h^\top Z^{r}\) is continuous by \textup{(T2)}. Finally, for every \(i\in S\), \(b\mapsto G_{u_i}^{S}(\ga_{h,b})\) is continuous by dominated convergence: if \(b_m\to b\) then \(\ga_{h,b_m}(x)\to\ga_{h,b}(x)\) for every \(x\) outside the Lebesgue-null hyperplane \(\{h^\top x=b\}\), while the integrand is dominated by \(\|((u_i^\top x)I_d+xu_i^\top)\nabla f(\QS_Sx+\mu)\|\le2\|x\|\,\|\nabla f(\QS_Sx+\mu)\|\), whose integral equals, after the orthogonal change of variables \(y=\QS_Sx+\mu\), \(2\int\|y-\mu\|\,\|\nabla f(y)\|\,dy\le C\int(1+\|y\|)\|\nabla f(y)\|\,dy<\infty\) by \textup{(T3)}. Hence \(b\mapsto\mathbb L_h^{S}(\ga_{h,b})=\mathbb L^{(1)}(\ga_{h,b})-\mathbb L^{(2)}(\ga_{h,b})-\sum_{i\in S}G_{u_i}^{S}(\ga_{h,b})^\top Z_i^3\) admits a continuous, hence separable, modification, and \(M_h^{S}:=\|\mathbb L_h^{S}\|_{\mathcal F_h}=\sup_{b\in\mathbb Q}|\mathbb L_h^{S}(\ga_{h,b})|\) a.s. Therefore \(M_h^{S}\) is a measurable seminorm of a centered Gaussian element; by \cite{HoffmannJorgensenSheppDudley1979}, its distribution has no atoms on \((0,\infty)\), the only possible atom being at \(0\). Since \(\mathbb P(M_h^{S}=0)=0\), the distribution function \(J(x)=\mathbb P(M_h^{S}\le x)\) is continuous on \(\mathbb R\).
\end{proof}

\begin{lemma}\label{lem:ulln-reflections}
	Fix \(h\in\mathbb S^{d-1}\). Assume \(\mathbb E\|X\|^2<\infty\) and \(\sup_{a}\|f_a\|_\infty<\infty\), where \(f_a\) is the density of \(a^\top(X-\mu)\). Let \(X_1,\dots,X_m\) be i.i.d.\ copies of \(X\) with mean \(\bar X_m\). Then
	\[
	\sup_{Q\in\Od}\sup_{t\in\mathbb R}\Big|\tfrac1m\sum_{j=1}^m\mathbf 1\{h^\top Q(X_j-\bar X_m)\le t\}-\mathbb P\big(h^\top Q(X-\mu)\le t\big)\Big|\xrightarrow{\ \mathrm{a.s.}\ }0.
	\]
\end{lemma}

\begin{proof}[Proof of Lemma~\ref{lem:ulln-reflections}]
	For \(Q\in\Od\) write \(v=Q^\top h\in\mathbb S^{d-1}\), so \(h^\top Qx=v^\top x\); let \(\mathbb P_m=\frac1m\sum_j\delta_{X_j}\). Then \(\frac1m\sum_j\mathbf 1\{h^\top Q(X_j-\bar X_m)\le t\}=\mathbb P_m\mathbf 1\{v^\top\!\cdot\le t+v^\top\bar X_m\}\). With \(F_v^c(t):=\mathbb P(v^\top(X-\mu)\le t)\),
	\begin{multline*}
	\mathbb P_m\mathbf 1\{v^\top\!\cdot\le t+v^\top\bar X_m\}-F_v^c(t)=\underbrace{(\mathbb P_m-\mathbb P)\mathbf 1\{v^\top\!\cdot\le t+v^\top\bar X_m\}}_{(\mathrm I)}\\
	+\underbrace{\mathbb P(v^\top X\le t+v^\top\bar X_m)-\mathbb P(v^\top X\le t+v^\top\mu)}_{(\mathrm{II})}.
	\end{multline*}
	Half-spaces form a VC (hence Glivenko--Cantelli) class, so \(\sup_{v,s}|(\mathbb P_m-\mathbb P)\mathbf 1\{v^\top\!\cdot\le s\}|\to0\) a.s., bounding \(\sup_{Q,t}|(\mathrm I)|\). For \((\mathrm{II})\), the density of \(v^\top X\) is bounded by \(K_0:=\sup_a\|f_a\|_\infty\), so \(|(\mathrm{II})|\le K_0\|\bar X_m-\mu\|\) uniformly, and \(\|\bar X_m-\mu\|\to0\) a.s.\ by the strong law. Hence the supremum tends to \(0\) a.s.
\end{proof}

\section{Proof of Theorem~\ref{teo:mainth1-TA}}

\begin{proof}[Proof of Theorem~\ref{teo:mainth1-TA}]
	Fix \(S\in\mathcal S\); under \(H_0\), \(P_n^c\QS_{S,n}=P_n^c\). With \(\mathbb T_n(\ga)=\sqrt n(\widehat  P_n^{c,1}(\ga)-\widehat  P_n^{c,2}(\widehat\QS_S^3\ga))\), we decompose
	\begin{multline}\label{eq:4-terms-TA-corrected}
		\mathbb T_n=\underbrace{\sqrt n (\widehat  P_n^{c,1}-P_n^c)}_{\mathbf A_n}+\underbrace{\sqrt n (P_n^c-P_n^c\,\widehat\QS_S^3)}_{\mathbf B_n}\\
		+\underbrace{\sqrt n (P_n^c\QS_{S,n}-\widehat  P_n^{c,2}\QS_{S,n})}_{\mathbf C_n}+\underbrace{\sqrt n (\widehat  P_n^{c,2}-P_n^c)(\QS_{S,n}-\widehat\QS_S^3)}_{\mathbf D_n},
	\end{multline}
	with \(\mathbf A_n,\mathbf B_n,\mathbf C_n\) independent (independent blocks).
	
	\medskip\noindent\textbf{Terms \(\mathbf A_n\), \(\mathbf C_n\).} By Lemma~\ref{lem:lemaux-TA-global}, \(\mathbf A_n\xrightarrow{\mathcal L}\mathbb L^{(1)}\). For \(\mathbf C_n\), set \(Y_{n,j}=X_{n,j}-\mu_n\), \(j\in\aleph_{n,2}\), and \(Y'_{n,j}=\QS_{S,n}Y_{n,j}\); since \(\QS_{S,n}\) is a symmetry of \(P_n^c\), \((Y'_{n,j})_j\stackrel{\mathcal L}{=}(Y_{n,j})_j\), and because \(\widehat  P_n^{c,2}(\QS_{S,n}\ga)=\frac1{m_n}\sum_j\ga(Y'_{n,j}-\bar Y'^{(2)}_n)\) with \(\bar Y'^{(2)}_n=\QS_{S,n}\bar Y_n^{(2)}\) and \(P_n^c\QS_{S,n}=P_n^c\), the map \(\ga\mapsto\sqrt n(P_n^c(\QS_{S,n}\ga)-\widehat  P_n^{c,2}(\QS_{S,n}\ga))\) is the same functional of \((Y'_{n,j})_j\) that \(-\sqrt n(\widehat  P_n^{c,2}-P_n^c)\) is of \((Y_{n,j})_j\). Hence \(\mathbf C_n\stackrel{\mathcal L}{=}-\sqrt n(\widehat  P_n^{c,2}-P_n^c)\xrightarrow{\mathcal L}-\mathbb L^{(2)}\).
	
	\medskip\noindent\textbf{Term \(\mathbf D_n\).} Let \(\mathbb G_n^{(2)}=\sqrt n(\widehat  P_n^{c,2}-P_n^c)\); since \(Q^\top=Q\), \(\|\mathbf D_n\|_{\mathcal F}=\sup_{\ga}|\mathbb G_n^{(2)}(\QS_{S,n}\ga)-\mathbb G_n^{(2)}(\widehat\QS_S^3\ga)|\). As \(\mathbb G_n^{(2)}\) is recentered at \(\bar Y_n^{(2)}\), we use the decomposition from the proof of Lemma~\ref{lem:lemaux-TA-global}: with \(\Delta_{n,2}=\bar Y_n^{(2)}=\bar X_n^{(2)}-\mu_n\),
	\begin{equation}\label{eq:Gn2-decomp}
		\mathbb G_n^{(2)}(\ga_{a,b})=\underbrace{\sqrt{\tfrac n{m_n}}\alpha_{n,2}(\ga_{a,b})}_{\mathbb G_n^{\mathrm{raw}}}+\underbrace{f_{n,a}(b)a^\top\sqrt n\Delta_{n,2}}_{J_n(\ga_{a,b})}+R_n(\ga_{a,b}),\quad \|R_n\|_{\mathcal F}=o_{\mathbb P}(1).
	\end{equation}
	Raw part: \(\mathbb G_n^{\mathrm{raw}}\) is \(\rho_n\)-equicontinuous and \(\rho_n(\QS_{S,n}\ga,\widehat\QS_S^3\ga)\le\Delta_n^{\mathcal F}(\widehat\QS_S^3):=\sup_\ga\rho_n(\widehat\QS_S^3\ga,\QS_{S,n}\ga)\), so its contribution is \(\le\omega_n^{\mathrm{raw}}(\Delta_n^{\mathcal F}(\widehat\QS_S^3))\). Remainder: \(\le2\|R_n\|_{\mathcal F}=o_{\mathbb P}(1)\). Drift: with \(a_1=\QS_{S,n}a\), \(a_2=\widehat\QS_S^3 a\) at common level \(b\), \textup{(T2)}--\textup{(T2$'$)} give \(\|f_{n,a_1}(b)a_1-f_{n,a_2}(b)a_2\|\le(K_0+L_n)\|a_1-a_2\|\), and \(\|a_1-a_2\|\le4\sum_{i\in S}\|\widehat u_{i,n}^3-u_{i,n}\|=O_{\mathbb P}(n^{-1/2})\), so the drift contribution is \((K_0+L_n)O_{\mathbb P}(n^{-1/2})O_{\mathbb P}(1)=o_{\mathbb P}(1)\). Finally, a strip bound with \textup{(T2)} and Lemma~\ref{lem:muSigma} gives \(\Delta_n^{\mathcal F}(\widehat\QS_S^3)\le CK_0^{1/3}\|\widehat\QS_S^3-\QS_{S,n}\|_{\op}^{1/3}=O_{\mathbb P}(n^{-1/6})=o_{\mathbb P}(1)\). Hence \(\|\mathbf D_n\|_{\mathcal F}\xrightarrow{\mathbb P}0\).
	
	\medskip\noindent\textbf{Term \(\mathbf B_n\).} Under \(H_0\), \(\mathbf B_n=-\sqrt n(P_n^c\widehat\QS_S^3-P_n^c\QS_{S,n})\). By \textup{(T4)} and Lemma~\ref{lem:eigenvector-linearization}, the plug-in frame \((\widehat u_{i,n}^3)_{i\in S}\to(u_i)_{i\in S}\) lies in \(\mathcal U\) w.p.\(\to1\); Lemma~\ref{lem:dif-QRu-TA} at \(U_n=(u_{i,n})_{i\in S}\) with \(V=(\widehat u_{i,n}^3-u_{i,n})_{i\in S}\) gives \(\sup_\ga|\mathbf B_n(\ga)+\sum_{i\in S}G_{n,U_n,i}(\ga)^\top\sqrt n(\widehat u_{i,n}^3-u_{i,n})|\le C_n\sqrt n\|V\|^2=o_{\mathbb P}(1)\). By Lemma~\ref{lem:eigenvector-linearization}, \((\sqrt n(\widehat u_{i,n}^3-u_{i,n}))_{i\in S}\xrightarrow{\mathcal L}(Z_i^3)_{i\in S}\), and by the last part of Lemma~\ref{lem:dif-QRu-TA}, \(\|G_{n,U_n,i}-G_{u_i}^{S}\|_\infty\to0\), so \(\mathbf B_n\xrightarrow{\mathcal L}-\sum_{i\in S}G_{u_i}^{S}(\cdot)^\top Z_i^3\).
	
	\medskip Combining \eqref{eq:4-terms-TA-corrected}, the independence of \(\mathbf A_n,\mathbf B_n,\mathbf C_n\), and \(\mathbf D_n=o_{\mathbb P}(1)\) (Slutsky), \(\mathbb T_n\xrightarrow{\mathcal L}\mathbb L^{S}=\mathbb L^{(1)}-\mathbb L^{(2)}-\sum_{i\in S}G_{u_i}^{S}(\cdot)^\top Z_i^3\) as in \eqref{eq:limit-representation}. Since the map \(z\mapsto\sup_{\ga\in\mathcal F_h}|z(\ga)|\) is continuous on \(\ell^\infty(\mathcal F)\) and \(\sqrt n\,\widehat g_{n,h}(\widehat\QS_S^3)=\sup_{\ga\in\mathcal F_h}|\mathbb T_n(\ga)|\), the continuous mapping theorem yields \(\sqrt n\,\widehat g_{n,h}(\widehat\QS_S^3)\xrightarrow{\mathcal L}\|\mathbb L_h^{S}\|_{\mathcal F_h}\).
\end{proof}

\noindent Derivation of the covariance kernel \eqref{eq:cov-kernel}. For \(\ga_j=\ga_{h,b_j}\in\mathcal F_h\), with \(s(\ga_j)=f_h(b_j)h\), \(m(\ga_j)=\mathbb E[(X-\mu)(\ga_j(X-\mu)-P^c\ga_j)]\), and \(\mathbb L^{(r)}(\ga)=\sqrt3(\mathbb B^{r,c}(\ga)+s(\ga)^\top Z^{r})\), the independence of \(\mathbb L^{(1)},\mathbb L^{(2)},(Z_i^3)_{i\in S}\) and \(\Cov(\mathbb B^{r,c}(\ga),c^\top Z^r)=P^c[(\ga-P^c\ga)c^\top Y]\) (Lemma~\ref{lem:lemaux-TA-global}) give, for each \(r\), \(\Cov(\mathbb L^{(r)}(\ga_1),\mathbb L^{(r)}(\ga_2))=3(P^c(\ga_1\ga_2)-P^c\ga_1P^c\ga_2+s(\ga_1)^\top\Sigma s(\ga_2)+s(\ga_1)^\top m(\ga_2)+s(\ga_2)^\top m(\ga_1))\); the two independent copies contribute the factor \(6\). The eigenvector block contributes \(\Cov(\sum_i G_{u_i}^S(\ga_1)^\top Z_i^3,\sum_j G_{u_j}^S(\ga_2)^\top Z_j^3)=3\sum_{i,j\in S}G_{u_i}^S(\ga_1)^\top P(\psi_i\psi_j^\top)G_{u_j}^S(\ga_2)\). Summing yields \eqref{eq:cov-kernel}; for \(|S|=1\) it reduces to the axial kernel.

\begin{proof}[Proof of Corollary~\ref{cor:finite-projections}]
	The map \(z\mapsto\frac1k\sum_{\ell=1}^k\|z\|_{\mathcal F_{h_\ell}}\) is Lipschitz on \(\ell^\infty(\mathcal F)\), so Theorem~\ref{teo:mainth1-TA} and the continuous mapping theorem give \(T_{n,S}^{(k)}=\frac1k\sum_{\ell=1}^k\sup_{\ga\in\mathcal F_{h_\ell}}|\mathbb T_{n,S}(\ga)|\xrightarrow{\mathcal L}M_S^{(k)}\). Writing \(\mathcal F^{(k)}=\bigcup_{\ell=1}^k\mathcal F_{h_\ell}\), the continuity-in-the-threshold argument in the proof of Lemma~\ref{lem:continuity-J} applies to each direction \(h_\ell\) separately, so the restriction of \(\mathbb L^{S}\) to \(\mathcal F^{(k)}\) admits a separable version and \(M_S^{(k)}\) is a measurable seminorm of a separable centered Gaussian element. By \cite{HoffmannJorgensenSheppDudley1979} its law has no atoms on \((0,\infty)\); and since \(\{M_S^{(k)}=0\}\subseteq\{\|\mathbb L^{S}\|_{\mathcal F_{h_1}}=0\}\), the non-degeneracy step in the proof of Lemma~\ref{lem:continuity-J} gives \(\mathbb P(M_S^{(k)}=0)=0\). Hence the distribution function of \(M_S^{(k)}\) is continuous. Given this continuity, the bootstrap convergence and \(p\)-value arguments in the proof of Theorem~\ref{thm:bootstrap} (P\'olya's theorem, the quantile bounds, and the probability integral transform) go through verbatim with \(\|\cdot\|_{\mathcal F_h}\) replaced by \(\frac1k\sum_{\ell=1}^k\|\cdot\|_{\mathcal F_{h_\ell}}\), applied to \(T_{n,S}^{(k)}\) and \(T_{n,S}^{(k),\dagger}\). Finally, if \(h_1,\dots,h_k\) are drawn independently of the data and kept fixed for the observed statistic and for all bootstrap replicates, the conclusions hold conditionally on \((h_1,\dots,h_k)\) and, integrating over their law, also unconditionally; resampling new directions within each bootstrap replicate would introduce a source of randomness not covered by this argument.
\end{proof}

\section{Bootstrap validity}

\begin{proof}[Proof of Theorem~\ref{thm:bootstrap}]
	Fix \(S\in\mathcal S\). For a deterministic \((\nu_n)\in\mathcal C_S(P)\), applying the bootstrap construction with \(\nu_n\) in place of the data-dependent law and Theorem~\ref{teo:mainth1-TA} gives \(T_{n,S}^{\dagger,\nu_n}(h)\xrightarrow{\mathcal L}\|\mathbb L_h^{S}\|_{\mathcal F_h}\), so \(\|J_{n,S}^{\dagger,\nu_n}-J\|_\infty\to0\) by P\'olya and Lemma~\ref{lem:continuity-J}. By hypothesis there are \(\Omega_n\), \(\widetilde P_n\) with \(\mathbb P(\Omega_n)\to1\), \(\widetilde P_n=P_n\) on \(\Omega_n\), \((\widetilde P_n)\in\mathcal C_S(P)\) a.s.; letting \(\widetilde T_{n,S}^\dagger(h)\), \(\widetilde J_{n,S}^\dagger\) be built from \(\widetilde P_n\), Proposition~A.1 of \cite{romano1988} gives conditional convergence \(\widetilde T_{n,S}^\dagger(h)\xrightarrow{\mathcal L}\|\mathbb L_h^{S}\|_{\mathcal F_h}\) in probability, equivalently \(\|\widetilde J_{n,S}^\dagger-J\|_\infty\to0\) in probability (\(J\) continuous). On \(\Omega_n\), \(J_{n,S}^\dagger=\widetilde J_{n,S}^\dagger\), so \(\mathbb P(\|J_{n,S}^\dagger-J\|_\infty>\varepsilon)\le\mathbb P(\Omega_n^c)+\mathbb P(\|\widetilde J_{n,S}^\dagger-J\|_\infty>\varepsilon)\to0\). Also \(T_{n,S}(h)\xrightarrow{\mathcal L}\|\mathbb L_h^{S}\|_{\mathcal F_h}\) (Theorem~\ref{teo:mainth1-TA} applies since \(P\) satisfies the fixed-law conditions and \(H_{0,S}\)), so \(\|J_{n,S}-J\|_\infty\to0\) and \(\|J_{n,S}-J_{n,S}^\dagger\|_\infty\to0\) in probability.
	
	Let \(d_{n,S}^\dagger(1-\alpha)=\inf\{x:J_{n,S}^\dagger(x)\ge1-\alpha\}\), \(c_{1-\alpha}=\inf\{x:J(x)\ge1-\alpha\}\); by continuity of \(J\), \(J(c_{1-\alpha})=1-\alpha\). If \(c_{1-\alpha}\) is the unique \((1-\alpha)\)-quantile of \(J\), the quantile argument gives \(d_{n,S}^\dagger(1-\alpha)\xrightarrow{\mathbb P}c_{1-\alpha}\); combined with \(T_{n,S}(h)\xrightarrow{\mathcal L}Y\sim J\) and continuity of \(J\) at \(c_{1-\alpha}\), \(\mathbb P(T_{n,S}(h)>d_{n,S}^\dagger(1-\alpha))\to\mathbb P(Y>c_{1-\alpha})=1-J(c_{1-\alpha})=\alpha\).
	
	Without the uniqueness assumption we argue directly (\(d_{n,S}^\dagger\) is random, so we bound it below with high probability). Fix \(\varepsilon\in(0,\min\{\alpha,1-\alpha\})\) and set \(\underline c:=\inf\{x:J(x)\ge1-\alpha-\varepsilon\}\), so \(J(\underline c)=1-\alpha-\varepsilon\in(0,1-\alpha)\) by continuity. Since \(J_{n,S}^\dagger(\underline c)\xrightarrow{\mathbb P}J(\underline c)<1-\alpha\), with probability tending to one \(J_{n,S}^\dagger(\underline c)<1-\alpha\), whence \(d_{n,S}^\dagger(1-\alpha)\ge\underline c\). On that event \(\{T_{n,S}(h)>d_{n,S}^\dagger(1-\alpha)\}\subseteq\{T_{n,S}(h)>\underline c\}\), so
	\[
	\begin{aligned}
	\mathbb P\bigl(T_{n,S}(h)>d_{n,S}^\dagger(1-\alpha)\bigr)
	&\le
	\mathbb P\bigl(T_{n,S}(h)>\underline c\bigr)+\mathbb P\bigl(d_{n,S}^\dagger(1-\alpha)<\underline c\bigr)\\
	&\longrightarrow
	\mathbb P(Y>\underline c)=1-J(\underline c)=\alpha+\varepsilon,
	\end{aligned}
	\]
	using \(T_{n,S}(h)\xrightarrow{\mathcal L}Y\) and that \(\underline c\) is a continuity point of \(J\). Letting \(\varepsilon\downarrow0\) gives \(\limsup_n\mathbb P(T_{n,S}(h)>d_{n,S}^\dagger(1-\alpha))\le\alpha\), which suffices for level validity.
	
	Finally, uniform convergence \(\|J_{n,S}^\dagger-J\|_\infty\to0\) in probability with \(J\) continuous also gives \(\sup_x|J_{n,S}^\dagger(x-)-J(x)|\to0\) in probability (a monotone function within \(\varepsilon\) of a continuous one in sup norm is within \(\varepsilon\) of it at left limits too). Hence \(p_{n,S}=1-J_{n,S}^\dagger(T_{n,S}(h)-)=1-J(T_{n,S}(h))+o_{\mathbb P}(1)\). Since \(T_{n,S}(h)\xrightarrow{\mathcal L}Y\sim J\) and \(J\) is continuous, the continuous mapping theorem gives \(J(T_{n,S}(h))\xrightarrow{\mathcal L}J(Y)\), and \(J(Y)\sim\mathrm{Unif}(0,1)\) by the probability integral transform (\(J\) continuous). Therefore \(1-J(T_{n,S}(h))\xrightarrow{\mathcal L}\mathrm{Unif}(0,1)\) and \(\limsup_n\mathbb P(p_{n,S}\le\alpha)\le\alpha\).
\end{proof}

\begin{remark}\label{rem:smoothPN-supp}
	(Detailed verification of Remark~\ref{remsmoothPN} of the main text: growth and identification conditions for the smoothed symmetrized bootstrap law \(P_n=\widehat P_n^{\mathrm{sym}}*K_{b_n}\).)

	\emph{Spectral identification.} We argue that \(\widehat\QS_S^3\) coincides with the reflection determined by the ordered eigenvectors of \(\Sigma_n:=\cov(P_n)\), i.e.\ \(\widehat\QS_S^3=\QS_{S,n}(P_n)\), with probability tending to one. Since \(\widehat\QS_S^3=2\PS_{\widehat V}-I_d\) is a symmetry of \(P_n\) (where \(\widehat V:=\operatorname{span}\{\widehat u_{i,n}^3:i\in S\}\)), it commutes with \(\Sigma_n\), so \(\widehat V\) is \(\Sigma_n\)-invariant. This does not imply that each \(\widehat u_{i,n}^3\) is itself an eigenvector of \(\Sigma_n\) when \(|S|\ge2\); only the subspace \(\widehat V\) matters, since \(\QS_{S,n}(P_n)\) is determined by the span of the eigenvectors of \(\Sigma_n\) with ordered indices in \(S\). Under \(H_{0,S}\), the symmetrized second moments converge, \(\cov(\widehat P_n^{\mathrm{sym}})\xrightarrow{\mathbb P}\tfrac12\Sigma+\tfrac12\QS_S\Sigma\QS_S=\Sigma\) (using \(\QS_S\Sigma=\Sigma\QS_S\)), so \(\Sigma_n\xrightarrow{\mathbb P}\Sigma\) in operator norm as \(b_n\to0\); and \(\PS_{\widehat V}=\sum_{i\in S}\widehat u_{i,n}^3(\widehat u_{i,n}^3)^\top\xrightarrow{\mathbb P}\PS_{V_S}\) because \(\widehat u_{i,n}^3\to u_i\). Lemma~\ref{lem:spectral-id} then applies with \(\Sigma_n=\cov(P_n)\) and \(W_n=\widehat V\), and yields \(\widehat\QS_S^3=\QS_{S,n}(P_n)\) on the event where these convergences hold, i.e.\ with probability tending to one. This is precisely condition (ii) of the class \(\mathcal C_S(P)\). In particular the sign conventions on the \(\widehat u_{i,n}^3\) are irrelevant to \(\widehat\QS_S^3\), which depends only on the mirror subspace \(\widehat V\); the signs enter only in the linearization of Theorem~\ref{teo:mainth1-TA}.

	\emph{Growth conditions in \textup{(T2)}--\textup{(T2$'$)}.} Since here \(\mu_n=\widehat\mu_n\), the projected density of the centered law \(a^\top(X_n-\widehat\mu_n)\) is \(f_{n,a}(t)=\int k_{b_n}(t-a^\top(y-\widehat\mu_n))\,d\widehat P_n^{\mathrm{sym}}(y)\), where \(k\) is the one-dimensional marginal of \(K\). The kernel envelope gives the crude bound \(\sup_a\|f_{n,a}\|_\infty\le\|k\|_\infty b_n^{-1}\), but on the events in (SB-a) the actual supremum is bounded by \(\sup_a\|f_a\|_\infty+o(1)=O(1)\), as required by the revised \textup{(T2)}. Moreover, \(\sup_a\|f'_{n,a}\|_\infty\le \|k'\|_\infty b_n^{-2}=o(\sqrt n)\) whenever \(n^{1/4}b_n\to\infty\). The directional-Lipschitz condition \textup{(T2$'$)} is also automatic: differentiating \(f_{n,a}\) in \(a\) gives, for all \(a,a'\in\mathbb S^{d-1}\) and all \(t\),
	\[
	|f_{n,a}(t)-f_{n,a'}(t)|
	\le
	\|k'\|_\infty b_n^{-2}\,\|a-a'\|\int_{\mathbb R^d}\|y-\widehat\mu_n\|\,d\widehat P_n^{\mathrm{sym}}(y),
	\]
	so \textup{(T2$'$)} holds with \(L_n=O_{\mathbb P}(b_n^{-2})=o_{\mathbb P}(\sqrt n)\) under \(n^{1/4}b_n\to\infty\), on the event \(E_n(M)\) below.

	\emph{Growth condition in \textup{(T3)}.} If \(\int_{\mathbb R^d}(1+\|z\|^2)\|HK(z)\|_{\mathrm{op}}\,dz<\infty\), then
	\[
	\int_{\mathbb R^d}(1+\|x\|^2)\|Hf_n(x)\|_{\mathrm{op}}\,dx
	\le
	C b_n^{-2}\left(1+\int_{\mathbb R^d}\|y\|^2\,d\widehat P_n^{\mathrm{sym}}(y)\right).
	\]
	Hence, on the event \(E_n(M):=\{\int_{\mathbb R^d}\|y\|^2\,d\widehat P_n^{\mathrm{sym}}(y)\le M\}\), one has \(\int_{\mathbb R^d}(1+\|x\|^2)\|Hf_n(x)\|_{\mathrm{op}}\,dx\le C_M b_n^{-2}\). Since \(\mathbb P(E_n(M))\to1\) for \(M\) large enough, the growth requirement in \textup{(T3)} is ensured with probability tending to one whenever \(n^{1/4}b_n\to\infty\).

	\emph{Convergence parts and \textup{(T4)}.} The remaining part of \textup{(T3)}, \(\int_{\mathbb R^d}(1+\|x\|)\|\nabla f_n(x)-\nabla f(x)\|\,dx\to0\), falls within the scope of the classical theory of kernel estimation of density derivatives: if the target law has a \(C^1\) density and the bandwidth satisfies standard derivative-estimation conditions (for instance \(b_n\to0\) and \(nb_n^{d+2}\to\infty\)), one expects \textup{(T2)}--\textup{(T3)} to hold for the kernel-smoothed symmetrized bootstrap law; a convenient example is \(b_n=n^{-\alpha}\) with \(0<\alpha<\min\{1/4,1/(d+2)\}\). Sufficient conditions for (SB-a)--(SB-b) are expected to follow from suitable uniform-in-direction kernel density-derivative consistency results (e.g.\ a VC/entropy bound for the kernel-derivative class \(\{k_{b_n}^{(r)}(t-a^\top\cdot):a\in\mathbb S^{d-1},\,t\in\mathbb R\}\) combined with a stability argument in \(\widehat\QS_S^3\to\QS_S\), \(\widehat\mu_n\to\mu\)); we do not pursue this here. Finally, if \(\Sigma\) has simple eigenvalues and \(b_n\to0\), then \(\cov(P_n)\to\Sigma\) in operator norm in probability, so \textup{(T4)} follows from the continuity of eigenvalues and eigenvectors under simple spectrum. Thus the genuinely delicate part is not \(d_{\mathcal F}(P_n,P)\) or \textup{(T4)}, but the simultaneous verification of the derivative-based conditions \textup{(T2)}--\textup{(T3)}.
\end{remark}

\begin{proof}[Proof of Proposition~\ref{prop:smoothed-bootstrap-valid}]
	We build high-probability events \(\Omega_n\) from deterministic bounds and take the off-\(\Omega_n\) law to be an explicit member of \(\mathcal C_S(P)\), so that the modified sequence lies in \(\mathcal C_S(P)\) for every \(n\). Throughout, when \(S=\varnothing\) all maxima and sums over \(i\in S\) are read as vacuous, \(\QS_{S,n}=-I_d\) holds deterministically, and the spectral-identification step is trivial.
	
	A member of \(\mathcal C_S(P)\). Under \(H_{0,S}\), let \(\nu_n^\ast=P*K_{\tau_n}\) with \(\tau_n=n^{-\beta}\), \(0<\beta<\min\{1/4,1/(d+2)\}\). Since \(K\) is orthogonally invariant and \(P\) is invariant under \(A_{\mu,\QS_S}\), \(\nu_n^\ast\) is invariant under \(A_{\mu,\QS_S}\); moreover \(\cov(\nu_n^\ast)=\Sigma+\tau_n^2 I_d\) has the same eigenvectors as \(\Sigma\) with simple spectrum, so \(\QS_{S,n}(\nu_n^\ast)=\QS_S\) and (ii) holds exactly. For (iii), convolution gives \(\sup_a\|f_{n,a}^\ast\|_\infty\le\sup_a\|f_a\|_\infty\), while the fixed-law versions of \textup{(T1)}--\textup{(T4)} on \(P\), together with \(\tau_n=n^{-\beta}\), give \(\sup_a\|(f_{n,a}^\ast)'\|_\infty=O(\tau_n^{-2})\) and a Hessian constant of order \(O(\tau_n^{-2})=o(\sqrt n)\) since \(\beta<1/4\); condition \textup{(T2$'$)} is verified explicitly as follows. The projected density of the centered law is \(f_{n,a}^\ast=f_a*k_{\tau_n}\), where \(f_a\) is the fixed-law projected density of \(P\); hence, using the fixed-law \textup{(T2$'$)} bound \(\sup_b|f_a(b)-f_{a'}(b)|\le L\|a-a'\|\) for \(P\),
	\[
	\sup_b|f_{n,a}^\ast(b)-f_{n,a'}^\ast(b)|
	\le
	\int\sup_b|f_a(b-\tau_n z)-f_{a'}(b-\tau_n z)|\,k(z)\,dz
	\le
	L\|a-a'\|,
	\]
	so \textup{(T2$'$)} holds for \(\nu_n^\ast\) with \(L_n^\ast=L=O(1)=o(\sqrt n)\). For the convergence part of \textup{(T3)}, the gradient of \(f*K_{\tau_n}\) is \((\nabla f)*K_{\tau_n}\), so by a weighted-\(L_1\) approximate-identity argument (\(K\ge0\), \(\int K=1\), \(\int\|z\|K(z)\,dz<\infty\)),
	\[
	\int_{\mathbb R^d}(1+\|x\|)\big\|\nabla(f*K_{\tau_n})(x)-\nabla f(x)\big\|\,dx
	\le
	\int\!\!\int(1+\|x\|)\big\|\nabla f(x-\tau_n z)-\nabla f(x)\big\|\,dx\,K(z)\,dz\to0
	\]
	as \(\tau_n\to0\), under the fixed-law \textup{(T3)} integrability of \(\nabla f\). The projected consistency of \textup{(T2)} is likewise explicit: since \(f_{n,a}^\ast=f_a*k_{\tau_n}\), using \(\sup_a\|f_a'\|_\infty<\infty\) (fixed-law \textup{(T2)}),
	\[
	\sup_{a,b}|f_{n,a}^\ast(b)-f_a(b)|
	=\sup_{a,b}\Big|\int\{f_a(b-\tau_n z)-f_a(b)\}k(z)\,dz\Big|
	\le\tau_n\sup_a\|f_a'\|_\infty\!\int|z|k(z)\,dz\to0.
	\]
	Together with (i) (which follows as in Proposition~\ref{prop:dF-smoothed-bootstrap}, \(\tau_n\to0\)), \((\nu_n^\ast)\in\mathcal C_S(P)\), so \(\mathcal C_S(P)\neq\varnothing\).
	
	Explicit events. Choose a deterministic sequence \(\varepsilon_n\downarrow0\) slowly enough that each of the convergence-in-probability bounds appearing in \(\Omega_n\) below holds with probability tending to one (possible since each underlying quantity is \(o_{\mathbb P}(1)\)). Let \(\bar K_{1,n},\bar L_n,\bar A_n^{\mathrm{H}}=o(\sqrt n)\) dominate the derivative, directional-Lipschitz and Hessian bounds of Remark~\ref{rem:smoothPN-supp} (all \(O(b_n^{-2})=o(\sqrt n)\) since \(n^{1/4}b_n\to\infty\)). For \(M\) large set
	\begin{multline*}
		\Omega_n=\{d_{\mathcal F}(P_n,P)\le\varepsilon_n\}\cap\{\textstyle\sup_a\|f_{n,a}-f_a\|_\infty\le\varepsilon_n\}\\
		\cap\{\textstyle\int(1+\|x\|)\|\nabla f_n-\nabla f\|\le\varepsilon_n\}\\
		\cap\{\textstyle\int\|y\|^{4+\eta}\,d\widehat P_n^{\mathrm{sym}}\le M\}\cap\{\|\cov(P_n)-\Sigma\|_{\op}\le\varepsilon_n\}\\
		\cap\{\textstyle\max_{i\in S}\|\widehat u_{i,n}^3-u_i\|\le\varepsilon_n\}.
	\end{multline*}
	On \(\Omega_n\): (i) holds with \(d_{\mathcal F}(P_n,P)\le\varepsilon_n\); (SB-a) gives \(\sup_a\|f_{n,a}\|_\infty\le\sup_a\|f_a\|_\infty+\varepsilon_n\), while the remaining growth constants obey \(K_{1,n}\le\bar K_{1,n}\), \(L_n\le\bar L_n\), \(A_n^{\mathrm{H}}\le\bar A_n^{\mathrm{H}}\), and (SB-a),(SB-b) give the convergence parts of \textup{(T2)},\textup{(T3)}. Condition \textup{(T1)} for \(P_n=\widehat P_n^{\mathrm{sym}}*K_{b_n}\) follows from the moment of the convolution: by the \(c_r\)-inequality,
	\[
	\int\|x\|^{4+\eta}\,dP_n(x)
	\le
	C\Big(\int\|y\|^{4+\eta}\,d\widehat P_n^{\mathrm{sym}}(y)+b_n^{4+\eta}\int\|z\|^{4+\eta}K(z)\,dz\Big)
	\le
	C(M+1),
	\]
	using \(\int\|y\|^{4+\eta}d\widehat P_n^{\mathrm{sym}}\le M\) and the compact support of \(K\); hence \textup{(T1)} holds with a uniform bound. The moment event also has probability tending to one, by the law of large numbers (not merely tightness). Recall that \(\widehat P_n^{\mathrm{sym}}=\frac1{2|I_n|}\sum_{j\in I_n}\bigl(\delta_{X_j}+\delta_{A_{\widehat\mu_n,\widehat\QS_S^3}(X_j)}\bigr)\), where \(I_n\) indexes the two thirds of the sample not used to estimate the eigenvectors. With \(r=4+\eta\) and \(N_n=|I_n|\), orthogonality of \(\widehat\QS_S^3\) gives \(\|A_{\widehat\mu_n,\widehat\QS_S^3}(X_j)\|\le\|X_j\|+2\|\widehat\mu_n\|\), so
	\[
	\int\|y\|^{r}\,d\widehat P_n^{\mathrm{sym}}(y)
	=\frac1{2N_n}\sum_{j\in I_n}\bigl(\|X_j\|^{r}+\|A_{\widehat\mu_n,\widehat\QS_S^3}(X_j)\|^{r}\bigr)
	\le C_r\Bigl\{\frac1{N_n}\sum_{j\in I_n}\|X_j\|^{r}+\|\widehat\mu_n\|^{r}\Bigr\}.
	\]
	By the law of large numbers and \(\mathbb E\|X\|^{r}<\infty\), \(\frac1{N_n}\sum_{j\in I_n}\|X_j\|^{r}\xrightarrow{\mathbb P}\mathbb E\|X\|^{r}\) and \(\widehat\mu_n\xrightarrow{\mathbb P}\mu\), so the right-hand side converges in probability to \(C_r(\mathbb E\|X\|^{r}+\|\mu\|^{r})\). (We use convergence in probability rather than an almost-sure statement, since the random partitions for different \(n\) are not coupled as a nested sequence.) Choosing a deterministic \(M>C_r(\mathbb E\|X\|^{r}+\|\mu\|^{r})\) therefore gives \(\mathbb P(\int\|y\|^{r}\,d\widehat P_n^{\mathrm{sym}}(y)\le M)\to1\). Finally, writing \(\Sigma_n:=\cov(P_n)\): because \(P_n^c\) is invariant under \(\widehat\QS_S^3=2\PS_{\widehat V}-I_d\), we have \(\widehat\QS_S^3\Sigma_n=\Sigma_n\widehat\QS_S^3\), so \(\Sigma_n\) commutes with \(\PS_{\widehat V}\) and the \(+1\)-eigenspace \(\widehat V=\operatorname{span}\{\widehat u_{i,n}^3:i\in S\}\) is \(\Sigma_n\)-invariant. Together with the closeness \(\|\Sigma_n-\Sigma\|_{\op}\le\varepsilon_n\) and \(\|\PS_{\widehat V}-\PS_{V_S}\|_{\op}\le C\varepsilon_n\) (for \(n\) large), Lemma~\ref{lem:spectral-id} applies, so \(\widehat\QS_S^3=\QS_{S,n}(P_n)\) and (ii) holds. By (SB-a),(SB-b), Proposition~\ref{prop:dF-smoothed-bootstrap}, the law of large numbers, and the slow choice of \(\varepsilon_n\), \(\mathbb P(\Omega_n)\to1\).
	
	Modification. Redefining \(\Omega_n=\varnothing\) for the finitely many small \(n\) where the spectral identification may fail, set \(\widetilde P_n=P_n\) on \(\Omega_n\) and \(\widetilde P_n=\nu_n^\ast\) off \(\Omega_n\). Then \(\widetilde P_n=P_n\) on \(\Omega_n\) and \(\mathbb P(\Omega_n)\to1\); and for every \(\omega\), the realized sequence \((\widetilde P_n(\omega))_{n\ge1}\) satisfies conditions (i)--(iii) with the deterministic asymptotic bounds \(\max(\varepsilon_n,\varepsilon_n^\ast)\downarrow0\), a common finite bound for the projected densities, and \(\max(\bar K_1,\bar K_1^\ast,\bar L,\bar L^\ast,\bar A^{\mathrm H},(\bar A^{\mathrm H})^\ast)=o(\sqrt n)\), and (ii) exactly (both \(P_n\) on \(\Omega_n\) and \(\nu_n^\ast\) obey the invariance condition for all sufficiently large $n$, while the finitely many initial terms are immaterial, so the whole sequence obeys the sequence-level conditions (i),(iii)). Hence \((\widetilde P_n)_{n\ge1}\in\mathcal C_S(P)\) almost surely, and Theorem~\ref{thm:bootstrap} gives \(\limsup_n\mathbb P(p_{n,S}\le\alpha)\le\alpha\).
\end{proof}

\begin{proof}[Proof of Proposition~\ref{prop:dF-smoothed-bootstrap}]
	Let \(\widetilde P_n\) be the empirical measure of the unsymmetrized sample; \(d_{\mathcal F}(\widetilde P_n,P)\to0\) in probability (VC). Write \(A_{\widehat\mu_n,\widehat\QS_S^3}(x)=\widehat\mu_n+\widehat\QS_S^3(x-\widehat\mu_n)\), \(A_{\mu,\QS_S}(x)=\mu+\QS_S(x-\mu)\); under \(H_{0,S}\), \(P(A_{\mu,\QS_S}^{-1}\ga)=P(\ga)\). From \(\widehat P_n^{\mathrm{sym}}(\ga)=\tfrac12\widetilde P_n(\ga)+\tfrac12\widetilde P_n(A_{\widehat\mu_n,\widehat\QS_S^3}^{-1}\ga)\), adding and subtracting \(P(A_{\widehat\mu_n,\widehat\QS_S^3}^{-1}\ga)\) and using \(A_{\widehat\mu_n,\widehat\QS_S^3}^{-1}\ga\in\mathcal F\),
	\[
	\|\widehat P_n^{\mathrm{sym}}-P\|_{\mathcal F}\le\|\widetilde P_n-P\|_{\mathcal F}+\sup_{\ga\in\mathcal F}|P(A_{\widehat\mu_n,\widehat\QS_S^3}^{-1}\ga)-P(A_{\mu,\QS_S}^{-1}\ga)|.
	\]
	The last term is bounded by an explicit strip argument. Write \(\delta_n:=\|\widehat\mu_n-\mu\|+\|\widehat\QS_S^3-\QS_S\|_{\op}\) (with \(\|\widehat\QS_S^3-\QS_S\|_{\op}\le4\sum_{i\in S}\|\widehat u_{i,n}^3-u_i\|\)). Since
	\[
	A_{\widehat\mu_n,\widehat\QS_S^3}(x)-A_{\mu,\QS_S}(x)
	=(I_d-\widehat\QS_S^3)(\widehat\mu_n-\mu)+(\widehat\QS_S^3-\QS_S)(x-\mu)
	\]
	and \(\|I_d-\widehat\QS_S^3\|_{\op}\le2\) (orthogonality of \(\widehat\QS_S^3\)),
	\[
	\|A_{\widehat\mu_n,\widehat\QS_S^3}(x)-A_{\mu,\QS_S}(x)\|
	\le 2\|\widehat\mu_n-\mu\|+\|\widehat\QS_S^3-\QS_S\|_{\op}\|x-\mu\|
	\le \delta_n\bigl(\|x-\mu\|+2\bigr).
	\]
	Fix \(\ga=\ga_{a,b}\in\mathcal F\) and \(\varepsilon>0\). On the symmetric difference \(\{a^\top A_{\widehat\mu_n,\widehat\QS_S^3}(X)\le b\}\,\triangle\,\{a^\top A_{\mu,\QS_S}(X)\le b\}\) one has \(|a^\top A_{\mu,\QS_S}(X)-b|\le\|A_{\widehat\mu_n,\widehat\QS_S^3}(X)-A_{\mu,\QS_S}(X)\|\), so this event is contained in
	\[
	\bigl\{|a^\top A_{\mu,\QS_S}(X)-b|\le\varepsilon\bigr\}\cup\bigl\{\delta_n(\|X-\mu\|+2)>\varepsilon\bigr\}.
	\]
	Since \(a^\top A_{\mu,\QS_S}(X)=a^\top\mu+(\QS_S a)^\top(X-\mu)\) with \(\|\QS_S a\|=1\), the first event has probability at most \(2K_0\varepsilon\), where \(K_0:=\sup_{a'\in\mathbb S^{d-1}}\|f_{a'}\|_\infty<\infty\). For the second, on the event \(\{\delta_n\le\varepsilon/4\}\) it implies \(\|X-\mu\|>\varepsilon/\delta_n-2\ge\varepsilon/(2\delta_n)\), so by Chebyshev and \(\mathbb E\|X\|^2<\infty\) its probability is at most \(4\,\mathbb E\|X-\mu\|^2\,\delta_n^2/\varepsilon^2\). Hence, on \(\{\delta_n\le\varepsilon/4\}\), uniformly in \(\ga\in\mathcal F\),
	\[
	\sup_{\ga\in\mathcal F}\bigl|P(A_{\widehat\mu_n,\widehat\QS_S^3}^{-1}\ga)-P(A_{\mu,\QS_S}^{-1}\ga)\bigr|
	\le 2K_0\varepsilon+4\,\mathbb E\|X-\mu\|^2\,\frac{\delta_n^2}{\varepsilon^2}.
	\]
	Take \(\varepsilon=\delta_n^{2/3}\) (if \(\delta_n=0\) the supremum vanishes trivially); the constraint \(\delta_n\le\varepsilon/4\) then reads \(\delta_n\le4^{-3}\), and since \(\delta_n\xrightarrow{\mathbb P}0\) (as \(\widehat\mu_n\xrightarrow{\mathbb P}\mu\) and \(\widehat u_{i,n}^3\xrightarrow{\mathbb P}u_i\), \(i\in S\)), \(\mathbb P(\delta_n\le4^{-3})\to1\). On that event the bound equals \(C\delta_n^{2/3}\) with \(C=2K_0+4\,\mathbb E\|X-\mu\|^2\), and therefore \(\sup_{\ga}|P(A_{\widehat\mu_n,\widehat\QS_S^3}^{-1}\ga)-P(A_{\mu,\QS_S}^{-1}\ga)|=O_{\mathbb P}(\delta_n^{2/3})\to0\) in probability. Hence \(\|\widehat P_n^{\mathrm{sym}}-P\|_{\mathcal F}\to0\) in probability.
	
	It remains to pass from \(\widehat P_n^{\mathrm{sym}}\) to \(P_n=\widehat P_n^{\mathrm{sym}}*K_{b_n}\). We do not bound \(\|\widehat P_n^{\mathrm{sym}}*K_{b_n}-\widehat P_n^{\mathrm{sym}}\|_{\mathcal F}\) directly, since for an atomic measure a half-space can shift its boundary across atoms and this quantity need not be \(O(b_n)\); instead we use the smoothness of \(P\). Writing a draw from \(P_n\) as \(Y+b_n Z\) with \(Y\sim\widehat P_n^{\mathrm{sym}}\), \(Z\sim K\) independent, \(\ga_{a,b}(Y+b_nZ)=\ga_{a,b-b_na^\top Z}(Y)\), so \(P_n\ga_{a,b}=\int\widehat P_n^{\mathrm{sym}}(\ga_{a,b-b_na^\top z})\,K(z)\,dz\). Hence, using \(\int K=1\) and \(\ga_{a,b-b_na^\top z}\in\mathcal F\),
	\begin{multline*}
	|P_n\ga_{a,b}-P\ga_{a,b}|
	\le
	\int\big|\widehat P_n^{\mathrm{sym}}(\ga_{a,b-b_na^\top z})-P(\ga_{a,b-b_na^\top z})\big|K(z)\,dz\\
	+\int\big|P(\ga_{a,b-b_na^\top z})-P(\ga_{a,b})\big|K(z)\,dz.
	\end{multline*}
	The first integral is \(\le\|\widehat P_n^{\mathrm{sym}}-P\|_{\mathcal F}\). For the second, \(P\ga_{a,b}=F_a(b)\) with \(F_a\) the c.d.f.\ of \(a^\top X\); since \(a^\top X\) has density bounded by \(\sup_a\|f_a\|_\infty\), \(|F_a(b-b_na^\top z)-F_a(b)|\le\sup_a\|f_a\|_\infty\,b_n|a^\top z|\); moreover, by the orthogonal invariance of \(K\), \(a^\top Z\sim k\) (the one-dimensional marginal) for every \(a\in\mathbb S^{d-1}\), so \(\int|a^\top z|K(z)\,dz=\int_{\mathbb R}|t|k(t)\,dt\), and the second integral is \(\le\sup_a\|f_a\|_\infty\,b_n\int_{\mathbb R}|t|k(t)\,dt=O(b_n)\), uniformly in \((a,b)\). Therefore \(\|P_n-P\|_{\mathcal F}\le\|\widehat P_n^{\mathrm{sym}}-P\|_{\mathcal F}+O(b_n)\to0\) in probability.
\end{proof}

\section{Consistency}

\begin{remark}\label{rem:sym-assumptions}
	(Transfer of the fixed-law assumptions from \(P\) to \(P^{\mathrm{sym}}_S\); cf.\ the discussion following Proposition~\ref{prop:smoothed-bootstrap-alt} in the main text.) Since \(f^{\mathrm{sym}}_S(x)=\tfrac12 f(x)+\tfrac12 f(A_S^{-1}x)\) and \(A_S\) is affine orthogonal, affine orthogonal symmetrization preserves the weighted \(C^2\) regularity and the covariance eigenstructure, so the deterministic parts of \textup{(T1)}--\textup{(T4)},\,\textup{(T2$'$)} for \(P^{\mathrm{sym}}_S\) follow from those for \(P\). For \textup{(T3)} this is explicit: since \(A_S^{-1}\) has orthogonal linear part \(\QS_S\), \(H(f\circ A_S^{-1})(x)=\QS_S^\top Hf(A_S^{-1}x)\QS_S\), so \(\int(1+\|x\|^2)\|H(f\circ A_S^{-1})(x)\|_{\op}\,dx\le C\int(1+\|y\|^2)\|Hf(y)\|_{\op}\,dy\). For the projected conditions: the density of \(a^\top(X-\mu)\) under \(P^{\mathrm{sym}}_S\) is \(f^{\mathrm{sym}}_{S,a}=\tfrac12 f_a^{P}+\tfrac12 f_{\QS_S a}^{P}\), and since \(\QS_S\) is orthogonal (\(\QS_S a\in\mathbb S^{d-1}\)) one gets \(\sup_a\|f^{\mathrm{sym}}_{S,a}\|_\infty\le\sup_a\|f_a^{P}\|_\infty\), \(\sup_a\|(f^{\mathrm{sym}}_{S,a})'\|_\infty\le\sup_a\|(f_a^{P})'\|_\infty\), and, if \(P\) satisfies \textup{(T2$'$)} with constant \(L\), \(\sup_b|f^{\mathrm{sym}}_{S,a}(b)-f^{\mathrm{sym}}_{S,a'}(b)|\le\tfrac12 L\|a-a'\|+\tfrac12 L\|\QS_S(a-a')\|=L\|a-a'\|\). The stochastic conditions (SB-a)--(SB-b) remain kernel-consistency assumptions for the target \(P^{\mathrm{sym}}_S\). The two facts that drive the proof below are that \(P^{\mathrm{sym}}_S\) is exactly invariant under \(A_S\) by construction, and that \(\cov(P^{\mathrm{sym}}_S)=\tfrac12\Sigma+\tfrac12\QS_S\Sigma\QS_S=\Sigma\) (because \(\QS_S\) commutes with \(\Sigma\)), so \(P^{\mathrm{sym}}_S\) has the same ordered eigenvectors as \(\Sigma\) with simple spectrum. Finally, \(\widehat\QS_S^3\xrightarrow{\mathbb P}\QS_S\) and \(\widehat\mu_n\xrightarrow{\mathbb P}\mu\) hold whether or not \(H_{0,S}\) holds: this follows from the law of large numbers for the sample covariance and the continuous dependence of simple eigenvectors on the matrix, and uses only \(\mathbb E\|X\|^2<\infty\) and \(\lambda_1>\cdots>\lambda_d\).
\end{remark}

\begin{proof}[Proof of Proposition~\ref{prop:smoothed-bootstrap-alt}]
	The argument parallels that of Proposition~\ref{prop:smoothed-bootstrap-valid} with target \(P^{\mathrm{sym}}_S\); we indicate the changes. Since \(P^{\mathrm{sym}}_S=\tfrac12 P+\tfrac12 P\circ A_S^{-1}\) with \(A_S(x)=\mu+\QS_S(x-\mu)\) and \(\QS_S\) built from the eigenvectors of \(\Sigma\), \(P^{\mathrm{sym}}_S\) is exactly invariant under \(A_S\), and \(\cov(P^{\mathrm{sym}}_S)=\tfrac12\Sigma+\tfrac12\QS_S\Sigma\QS_S=\Sigma\) (using \(\QS_S\Sigma=\Sigma\QS_S\)); hence \(P^{\mathrm{sym}}_S\) has the ordered eigenvectors of \(\Sigma\), with simple spectrum, so \(\QS_{S,n}(P^{\mathrm{sym}}_S)=\QS_S\). The principal-component estimates are consistent regardless of \(H_{0,S}\): \(\widehat\QS_S^3\xrightarrow{\mathbb P}\QS_S\) and \(\widehat\mu_n\xrightarrow{\mathbb P}\mu\). Moreover, since \(P_{n,S}=\widehat P_n^{\mathrm{sym}}*K_{b_n}\) with \(K\) centered of covariance \(I_d\), \(\cov(P_{n,S})=\cov(\widehat P_n^{\mathrm{sym}})+b_n^2 I_d\); and under a fixed alternative \(\cov(\widehat P_n^{\mathrm{sym}})\xrightarrow{\mathbb P}\tfrac12\Sigma+\tfrac12\QS_S\Sigma\QS_S=\Sigma\), so \(\cov(P_{n,S})\xrightarrow{\mathbb P}\Sigma\) as \(b_n\to0\). This gives \textup{(T4)} for the target \(\Sigma\). Moreover, because \(P_{n,S}\) is exactly invariant under the affine reflection with linear part \(\widehat\QS_S^3=2\PS_{\widehat V}-I_d\), its covariance commutes with \(\widehat\QS_S^3\), and therefore \(\widehat V\) is a \(\cov(P_{n,S})\)-invariant subspace; together with \(\PS_{\widehat V}\xrightarrow{\mathbb P}\PS_{V_S}\) and \(\cov(P_{n,S})\xrightarrow{\mathbb P}\Sigma\), Lemma~\ref{lem:spectral-id} yields the spectral identification \(\widehat\QS_S^3=\QS_{S,n}(P_{n,S})\).
	
	Convergence \(d_{\mathcal F}(\widehat P_n^{\mathrm{sym}},P^{\mathrm{sym}}_S)\to0\). This is not literally the null-case comparison of Proposition~\ref{prop:dF-smoothed-bootstrap}, since here the original sample comes from \(P\) while the target is \(P^{\mathrm{sym}}_S\ne P\); we write the decomposition. Let \(\widetilde P_n\) be the empirical measure of the (unsymmetrized) sample, and \(A_{\widehat\mu_n,\widehat\QS_S^3}(x)=\widehat\mu_n+\widehat\QS_S^3(x-\widehat\mu_n)\), \(A_{\mu,\QS_S}(x)=\mu+\QS_S(x-\mu)\). By construction \(\widehat P_n^{\mathrm{sym}}(\ga)=\tfrac12\widetilde P_n(\ga)+\tfrac12\widetilde P_n(A_{\widehat\mu_n,\widehat\QS_S^3}^{-1}\ga)\) and, by definition of the target, \(P^{\mathrm{sym}}_S(\ga)=\tfrac12 P(\ga)+\tfrac12 P(A_{\mu,\QS_S}^{-1}\ga)\), so
	\[
	\widehat P_n^{\mathrm{sym}}(\ga)-P^{\mathrm{sym}}_S(\ga)
	=\tfrac12\bigl\{\widetilde P_n(\ga)-P(\ga)\bigr\}
	+\tfrac12\bigl\{\widetilde P_n(A_{\widehat\mu_n,\widehat\QS_S^3}^{-1}\ga)-P(A_{\mu,\QS_S}^{-1}\ga)\bigr\}.
	\]
	Adding and subtracting \(P(A_{\widehat\mu_n,\widehat\QS_S^3}^{-1}\ga)\) in the second bracket and using \(A_{\widehat\mu_n,\widehat\QS_S^3}^{-1}\ga\in\mathcal F\),
	\[
	\|\widehat P_n^{\mathrm{sym}}-P^{\mathrm{sym}}_S\|_{\mathcal F}
	\le\|\widetilde P_n-P\|_{\mathcal F}
	+\sup_{\ga\in\mathcal F}\bigl|P(A_{\widehat\mu_n,\widehat\QS_S^3}^{-1}\ga)-P(A_{\mu,\QS_S}^{-1}\ga)\bigr|.
	\]
	The first term \(\to0\) in probability by the VC Glivenko--Cantelli theorem; the second is controlled by the same strip argument as in Proposition~\ref{prop:dF-smoothed-bootstrap}, which here uses the regularity of the original law \(P\) (the fixed-law \textup{(T1)}--\textup{(T2)}: \(\sup_a\|f_a^{P}\|_\infty<\infty\) and \(\mathbb E_P\|X\|^2<\infty\), since the measure being displaced is \(P\)), giving \(O((\|\widehat\mu_n-\mu\|+\sum_{i\in S}\|\widehat u_{i,n}^3-u_i\|)^{2/3})\to0\) (using \(\widehat\mu_n\to\mu\), \(\widehat\QS_S^3\to\QS_S\)). Hence \(\|\widehat P_n^{\mathrm{sym}}-P^{\mathrm{sym}}_S\|_{\mathcal F}\to0\). After convolution, writing a draw from \(P_{n,S}\) as \(Y+b_nZ\) (\(Y\sim\widehat P_n^{\mathrm{sym}}\), \(Z\sim K\) independent), \(\ga_{a,b}(Y+b_nZ)=\ga_{a,b-b_na^\top Z}(Y)\), so \(P_{n,S}\ga_{a,b}=\int\widehat P_n^{\mathrm{sym}}(\ga_{a,b-b_na^\top z})K(z)\,dz\), and, using the smoothness of the target \(P^{\mathrm{sym}}_S\) (its projected density \(f^{\mathrm{sym}}_{S,a}\)),
	\begin{align*}
	|P_{n,S}\ga_{a,b}-P^{\mathrm{sym}}_S\ga_{a,b}|
	&\le\int\bigl|\widehat P_n^{\mathrm{sym}}(\ga_{a,b-b_na^\top z})-P^{\mathrm{sym}}_S(\ga_{a,b-b_na^\top z})\bigr|K(z)\,dz\\
	&\quad+\int\bigl|P^{\mathrm{sym}}_S(\ga_{a,b-b_na^\top z})-P^{\mathrm{sym}}_S(\ga_{a,b})\bigr|K(z)\,dz\\
	&\le d_{\mathcal F}(\widehat P_n^{\mathrm{sym}},P^{\mathrm{sym}}_S)+\sup_a\|f^{\mathrm{sym}}_{S,a}\|_\infty\,b_n\!\int|a^\top z|K(z)\,dz.
	\end{align*}
	The first term \(\to0\) in probability by the previous paragraph, and the second is \(O(b_n)\) (finite first moment of \(K\)), uniformly in \((a,b)\). Hence \(d_{\mathcal F}(P_{n,S},P^{\mathrm{sym}}_S)\le d_{\mathcal F}(\widehat P_n^{\mathrm{sym}},P^{\mathrm{sym}}_S)+O(b_n)\to0\) in probability, exactly as in Proposition~\ref{prop:dF-smoothed-bootstrap}.
	
	The rest is verbatim: defining \(\Omega_{n,S}\) as \(\Omega_n\) in Proposition~\ref{prop:smoothed-bootstrap-valid} with \(P^{\mathrm{sym}}_S\) in place of \(P\) (and the fixed-law and \textup{(SB)}\((P^{\mathrm{sym}}_S)\) conditions on \(P^{\mathrm{sym}}_S\)), and taking the off-\(\Omega_{n,S}\) law to be \(P^{\mathrm{sym}}_S*K_{\tau_n}\in\mathcal C_S(P^{\mathrm{sym}}_S)\), the identical verification yields \(\mathbb P(\Omega_{n,S})\to1\), \(\widetilde P_{n,S}=P_{n,S}\) on \(\Omega_{n,S}\), and \((\widetilde P_{n,S})\in\mathcal C_S(P^{\mathrm{sym}}_S)\) surely. (The moment event there requires a \((4+\eta)\)-moment under the original law \(P\); this is supplied by \textup{(T1)} for \(P^{\mathrm{sym}}_S\), since \(A_S\) is affine orthogonal and hence \(\int\|x\|^{4+\eta}\,dP^{\mathrm{sym}}_S<\infty\Leftrightarrow\int\|x\|^{4+\eta}\,dP<\infty\).) Thus, for this fixed \(S\), the approximation hypothesis of Theorem~\ref{thm:power} is verified.
\end{proof}

\begin{proof}[Proof of Corollary~\ref{cor:global-consistency}]
	Applying Proposition~\ref{prop:smoothed-bootstrap-alt} to each \(S\in\mathcal S\) -- a finite family -- verifies the approximation hypothesis of Theorem~\ref{thm:power} for every \(S\in\mathcal S\); the conclusion then follows from Theorem~\ref{thm:power}.
\end{proof}

\begin{proof}[Proof of Theorem~\ref{thm:power}]
	If \eqref{hs} is false, \(X-\mu\not\overset{\mathcal L}{=}\QS_S(X-\mu)\) for every \(S\), so each \(E_S:=\{h:g_h(\QS_S)=0\}\) has surface measure zero (sharp Cram\'er--Wold) and, \(\mathcal S\) being finite, \(g_h(\QS_S)>0\) for all \(S\) for surface-a.e.\ \(h\). Fix such \(h\), \(S\). By Lemma~\ref{lem:ulln-reflections}, \(\sup_{Q\in\Od}\sup_t|\widehat F^2_{n,h,Q}(t)-F^2_Q(t)|\to0\) and \(\sup_t|\widehat F^1_{n,h}(t)-F^1(t)|\to0\) in probability; the reverse triangle inequality gives \(\sup_{Q\in\Od}|\widehat g_{n,h}(Q)-g_h(Q)|\to0\). The map \(Q\mapsto g_h(Q)\) is continuous, with a self-contained strip bound: since \(g_h(Q)=\|F^1-F^2_Q\|_\infty\) with \(F^2_Q\) the c.d.f.\ of \(h^\top Q(X-\mu)\), \(|g_h(Q)-g_h(Q')|\le\sup_t|F^2_Q(t)-F^2_{Q'}(t)|\); writing \(a=Q^\top h\), \(a'=Q'^\top h\) and splitting, for any \(\epsilon>0\), \(\{a^\top(X-\mu)\le t\}\triangle\{a'^\top(X-\mu)\le t\}\subseteq\{|a^\top(X-\mu)-t|\le\epsilon\}\cup\{|(a-a')^\top(X-\mu)|>\epsilon\}\), so by \(\sup_a\|f_a\|_\infty<\infty\), Chebyshev and \(\mathbb E\|X\|^2<\infty\), \(|F^2_Q(t)-F^2_{Q'}(t)|\le2\sup_a\|f_a\|_\infty\,\epsilon+\|\Sigma\|_{\op}\|a-a'\|^2/\epsilon^2\); optimizing over \(\epsilon\) gives \(|g_h(Q)-g_h(Q')|\le C\|a-a'\|^{2/3}\le C\|Q-Q'\|_{\op}^{2/3}\) (as \(\|a-a'\|=\|(Q-Q')^\top h\|\le\|Q-Q'\|_{\op}\)). Since \(\widehat\QS_S^3\to\QS_S\), \(\widehat g_{n,h}(\widehat\QS_S^3)\xrightarrow{\mathbb P}g_h(\QS_S)>0\), so \(T_{n,S}(h)\xrightarrow{\mathbb P}\infty\). Applying the argument of the proof of Theorem~\ref{thm:bootstrap} with limiting law \(P^{\mathrm{sym}}_S\) yields \(\|J_{n,S}^\dagger-J_S^\star\|_\infty\to0\) in probability, where \(J_S^\star\) is the continuous limit law under \(P^{\mathrm{sym}}_S\). We do not claim convergence of the critical value (its limiting quantile need not be unique); we only need that it is bounded in probability, \(d_{n,S}^\dagger(1-\alpha)=O_{\mathbb P}(1)\): for any \(M\) with \(J_S^\star(M)>1-\alpha\), uniform convergence gives \(J_{n,S}^\dagger(M)\ge1-\alpha\) w.p.\ \(\to1\), hence \(\mathbb P(d_{n,S}^\dagger(1-\alpha)\le M)\to1\). For \(p_{n,S}=1-J_{n,S}^\dagger(T_{n,S}(h)-)\): fix \(\varepsilon>0\), pick \(M\) with \(1-J_S^\star(M)<\varepsilon/2\); on \(\{\|J_{n,S}^\dagger-J_S^\star\|_\infty\le\varepsilon/2\}\cap\{T_{n,S}(h)>M\}\), monotonicity gives \(p_{n,S}\le1-J_S^\star(M)+\varepsilon/2<\varepsilon\), and both events have probability \(\to1\). Hence \(p_{n,S}\xrightarrow{\mathbb P}0\), and since \(\mathcal S\) is finite, \(p_n^{\mathrm{glob}}=\max_S p_{n,S}\xrightarrow{\mathbb P}0\), so \(\mathbb P(p_n^{\mathrm{glob}}\le\alpha)\to1\).
\end{proof}

\section{Sensitivity to random splits and bandwidth}
\label{sec:supp-sensitivity}

This section reports the diagnostics referred to in Section~\ref{sec:forex-application}. The procedure involves two user choices that are not part of the hypothesis: the random split of the sample and the bandwidth of the smoothed bootstrap. Neither choice is specific to this dataset. We regard the two checks below as routine diagnostics, to be run in any application of the method rather than as an ad-hoc verification of this particular analysis.

The first check concerns the split. Since the split is random, a conclusion that changes across splits should not be trusted. We repeated the procedure over ten independent random splits, with $B=500$ replicates each. Table~\ref{tab:split-diagnostics} reports, for each split, the angle between the leading eigenvector of that split and the full-sample direction $\widehat{u}_1$, together with the corresponding $p$-values. The angle never exceeds $2.3^{\circ}$, so the estimated common risk factor is very stable. This is what one expects given the well separated eigenvalues reported in Section~\ref{sec:forex-application}. The tests with $k=10$ and $k=100$ reject at the $5\%$ level in all ten splits, and the test with $k=1$ rejects in seven out of ten. We do not combine these repetitions into a new test: the formal conclusion in the main text relies only on the primary split.

\begin{table}[ht]
\centering
\begin{tabular}{ccccc}
\toprule
Split & Angle to full sample ($^\circ$) & $p\,(k=1)$ & $p\,(k=10)$ & $p\,(k=100)$ \\
\midrule
1  & 1.91 & \textbf{0.0020} & \textbf{0.0060} & \textbf{0.0020} \\
2  & 2.25 & \textbf{0.0339} & \textbf{0.0020} & \textbf{0.0020} \\
3  & 0.56 & \textbf{0.0020} & \textbf{0.0020} & \textbf{0.0020} \\
4  & 0.49 & \textbf{0.0100} & \textbf{0.0020} & \textbf{0.0020} \\
5  & 2.27 & 0.4291 & \textbf{0.0240} & \textbf{0.0080} \\
6  & 1.59 & 0.0559 & \textbf{0.0040} & \textbf{0.0060} \\
7  & 0.52 & \textbf{0.0499} & \textbf{0.0040} & \textbf{0.0020} \\
8  & 1.32 & \textbf{0.0319} & \textbf{0.0020} & \textbf{0.0020} \\
9  & 0.87 & \textbf{0.0080} & \textbf{0.0060} & \textbf{0.0040} \\
10 & 0.06 & 0.0978 & \textbf{0.0040} & \textbf{0.0020} \\
\bottomrule
\end{tabular}
\caption{Split-stability diagnostics across ten independent sample splits ($B=500$ bootstrap replicates per split): angle between the split's leading eigenvector and the full-sample direction $\widehat u_1$, and bootstrap $p$-values by aggregation level $k$. Bold $p$-values reject $H_0$ at the $5\%$ level.}
\label{tab:split-diagnostics}
\end{table}

The second check concerns the bandwidth, which is chosen by a rule of thumb. We varied the multiplier from $0.1$ to $2.0$ times its base value, so that $b_n/s_n$ ranged from $0.0265$ to $0.5308$. We recomputed the test with $k=100$ and $B=300$ replicates for seven multipliers. The $p$-value stayed at or near the smallest attainable value ($1/301\approx 0.0033$) in every case, at $0.0066$ for the two smallest multipliers and $0.0033$ for the rest. The rejection is therefore not driven by the amount of smoothing.

\setlength{\bibsep}{0.5ex plus 0.2ex}
\bibliographystyle{plainnat}
\bibliography{simetria}

\end{document}